\documentclass[12pt,leqno,a4paper]{amsart}
\usepackage{amssymb}

\newcommand{\FF}{{\mathbb{F}}}
\newcommand{\QQ}{{\mathbb{Q}}}
\newcommand{\ZZ}{{\mathbb{Z}}}
\newcommand{\RR}{{\mathbb{R}}}
\newcommand{\CC}{{\mathbb{C}}}

\newcommand{\bB}{{\mathbf{B}}}
\newcommand{\bG}{{\mathbf{G}}}
\newcommand{\bL}{{\mathbf{L}}}

\newcommand{\bT}{{\mathbf{T}}}

\newcommand{\cB}{{\mathcal{B}}}
\newcommand{\cC}{{\mathcal{C}}}

\newcommand{\cH}{{\mathcal{H}}}
\newcommand{\cO}{{\mathcal{O}}}

\newcommand{\cR}{{\mathcal{R}}}

\newcommand{\fS}{{\mathfrak{S}}}

\newcommand{\rX}{{\mathrm{X}}}

\newcommand{\Irr}{{\operatorname{Irr}}}
\newcommand{\IBr}{{\operatorname{IBr}}}
\newcommand{\Hom}{{\operatorname{Hom}}}

\newcommand{\GL}{{\operatorname{GL}}}
\newcommand{\SL}{{\operatorname{SL}}}
\newcommand{\GU}{{\operatorname{GU}}}
\newcommand{\SU}{{\operatorname{SU}}}
\newcommand{\PGU}{{\operatorname{PGU}}}
\newcommand{\SO}{{\operatorname{SO}}}
\newcommand{\Sp}{{\operatorname{Sp}}}
\newcommand{\RLG}{{R_\bL^\bG}}

\newcommand{\tw}[1]{{}^#1\!}
\newcommand{\Ph}[1]{\Phi_{#1}}
\newcommand{\Chevie}{{\sf Chevie}{}}

\newcommand\vr{{\vrule width16pt height3pt depth-2pt}\,}
\newcommand{\bla}{\boldsymbol{\lambda}}
\newcommand{\bmu}{{\boldsymbol{\mu}}}
\newcommand{\bnu}{{\boldsymbol{\nu}}}
\newcommand{\pl}{{\!+\!}}
\newcommand{\mn}{{\!-\!}}

\let\la=\lambda
\let\vhi=\varphi

\newtheorem{thm}{Theorem}[section]
\newtheorem{lem}[thm]{Lemma}
\newtheorem{prop}[thm]{Proposition}
\newtheorem{conj}[thm]{Conjecture}
\newtheorem{cor}[thm]{Corollary}

\newtheorem*{thmA}{Theorem A}
\newtheorem*{thmB}{Theorem B}

\theoremstyle{remark}
\newtheorem{rem}[thm]{Remark}

\begin{document}

\title[Decomposition matrices for unitary groups]{Decomposition matrices\\
  for low rank unitary groups}

\date{\today}

\author{Olivier Dudas}
\address{Universit\'e Paris Diderot, UFR de Math\'ematiques,
B\^atiment Sophie Germain, 5 rue Thomas Mann, 75205 Paris CEDEX 13, France.}
\email{dudas@math.jussieu.fr}

\author{Gunter Malle}
\address{FB Mathematik, TU Kaiserslautern, Postfach 3049,
         67653 Kaisers\-lautern, Germany.}
\email{malle@mathematik.uni-kl.de}

\thanks{The second author gratefully acknowledges financial support by ERC
  Advanced Grant 291512.}



\begin{abstract}
We study the decomposition matrices of the unipotent $\ell$-blocks of
finite special unitary groups $\SU_n(q)$ for unitary primes $\ell$ larger
than $n$. Up to few unknown entries, we give a complete solution for
$n=2,\ldots,10$. We also prove a general result for two-column partitions
when $\ell$ divides $q+1$. This is achieved using projective modules coming
from the $\ell$-adic cohomology of Deligne--Lusztig varieties.
\end{abstract}

\maketitle

%%\pagestyle{myheadings}
%%\markboth{}{}
%%\markboth{for personal use only}{preliminary}

%%%%%%%%%%%%%%%%%%%%%%%%%%%%%%%%%%%%%%%%%%%%%%%%%%%%%%%%%%%%%%%%%%%%%%%%%
\section{Introduction} \label{sec:intro}

One of the fundamental tasks in the representation theory of finite groups is
the determination of all decomposition matrices of the finite simple groups.
This paper aims at understanding the decomposition matrices of unipotent blocks
of finite unitary groups for unitary primes. This is considered a very hard
problem; for example, the $3\times3$-decomposition matrix for the 3-dimensional
unitary groups $\SU_3(q)$ has only been determined in 2002. Recall that for
so-called \emph{linear} primes, cuspidal unipotent modules in positive
characteristic all occur as $\ell$-modular reduction of cuspidal unipotent
modules in characteristic zero, and in that case the modular representation
theory of unitary groups is controlled by the representation theory of
$q$-Schur algebras, as shown in \cite{GrHi}. No analog approach is known to
exist for \emph{unitary} primes, and in fact many new cuspidal modules can
arise. The strategy of this paper is to adapt the theory of Deligne and
Lusztig to the modular framework, and construct these cuspidal modules and
their projective covers via the mod-$\ell$ cohomology of suitably chosen
Deligne--Lusztig varieties.

Let $\bG$ be a connected reductive algebraic group defined over
$\overline{\FF}_q$ with an additional $\FF_q$-structure. Given an element $w$
in the Weyl group $W$ of $\bG$, we can construct a virtual projective
$\bG(\FF_q)$-module $R_w$ given by the cohomology with compact support of
the Deligne--Lusztig variety $\rX(w)$. The classification of the irreducible
characters (in characteristic zero) of $\bG(\FF_q)$ given by Lusztig \cite{Lu84}
relies, amongst other things, on the following remarkable property of $R_w$:
\begin{center}
    \parbox{14cm}{If $\rho$ is a ``new'' cuspidal unipotent character occurring
    in $R_w$, then it occurs only in the middle degree of the cohomology
    of $\rX(w)$ (see \cite[Ex. 3.10(c)]{Lu78}).}
\end{center}
The modular counterpart of this property is proved in \cite[Prop. 8.10]{BR03}.
It gives some control on the multiplicity of projective covers of cuspidal
modules in the virtual module $R_w$, and has already been shown to be a powerful
tool to determine decomposition numbers in \cite{Du13,HN13}. In this paper we
combine this new ingredient with standard methods involving Harish-Chandra
induction from proper Levi subgroups, but also Harish-Chandra restriction from
suitable groups of larger rank, tensor products of characters and the known
fact that the decomposition matrix is uni-triangular to determine, up to only
few unknown entries when $n = 10$, the decomposition matrices of unipotent
$\ell$-blocks of $\SU_n(q)$ for $n \leq 10$ and $\ell$ a unitary prime, as
well as the repartition of the simple modules into $\ell$-modular
Harish-Chandra series. Thus, parts of our paper can be considered as
complementing James's determination \cite{Jam90} of the decomposition matrices
of $\GL_n(q)$ for $n\le10$, as well as Geck, Hiss and the second authors
determination \cite{GHM} of the decomposition matrices of unitary groups
$\SU_n(q)$ for $n \leq 10$ at linear primes.

\begin{thmA}
 Let $\ell > 2n-3$ be a prime number not dividing $q$. Assume that the order
 of $-q$ modulo $\ell$ is odd. Then the decomposition matrices of the unipotent
 $\ell$-blocks of $\SU_n(q)$ with $2\le n\le 10$ are as given in
 Tables~\ref{tab:2A}--\ref{tab:2A9,d=10}.
\end{thmA}

(See the individual statements for better bounds on $\ell$ in the various
cases.) Note that as expected, the entries we determine depend only on the
order of $-q$ modulo $\ell$, and not on $q$ and $\ell$ individually.
\par
In terms of Harish-Chandra theory, the most difficult situation is when the
order of $-q$ modulo $\ell$ is $1$, that is when $\ell$ divides $q+1$, in
which case most cuspidal modules will occur. However, the virtual modules
$R_w$ are of particular interest here, and it is to be expected that they
provide an alternative parametrization of cuspidal modules in terms of twisted
conjugacy classes of the symmetric group. This is achieved for classes
corresponding to certain two-column partitions, and the corresponding part of
the decomposition matrix is given as follows (see Theorem~\ref{thm:corner}):

\begin{thmB}   \label{thm:mainB}
 Assume $\ell>n$ and $\ell \mid (q+1)$. Let $a:=\lfloor n/3\rfloor+1$ and
 $b \leq a$. Then the multiplicity of the unipotent character of $\SU_n(q)$
 indexed by $2^b 1^{n-2b}$ in the projective indecomposable module indexed by
 $2^c 1^{n-2c}$ is given by $\binom{n-c-b}{c-b}$ if $b\le c$ and $0$ otherwise.
\end{thmB}

The entries in that lower right-hand corner of the decomposition matrix are
precisely those for which the previously known approaches tended not to yield
any information at all. This example is also very instructive as it gives
strong evidence that an analogue of James's row removal rule for $\GL_n(q)$
(see \cite[Rule~5.8]{Jam90}) should hold for $\SU_n(q)$ as well, see \S 5.4,
although we do not even have a conjectural explanation for why that should be
the case.
\par
We hope to apply our new methods to other exceptional and classical type groups
of low rank in a forthcoming paper.
\par
The paper is organized as follows. In Section~\ref{sec:DL} we set our notation
and recall some results from Deligne--Lusztig theory. In
Section~\ref{sec:conjclasses} we investigate a certain order relation on the
set of twisted conjugacy classes in symmetric groups. This is then applied,
after some preparations in Section~4, to prove in Section~\ref{sec:qplus1} the
main result on decomposition numbers of $\SU_n(q)$ for 2-column partitions in
Theorem~\ref{thm:corner}.
Finally, in Section~\ref{sec:smallrank} we apply the previously developed
methods and results to determine most entries of the decomposition matrices of
$\SU_n(q)$ for $n\le10$ and unitary primes.
\goodbreak

%\noindent{\bf Acknowledgement:}
%We would like to thank Jean Michel for helping us to compute the characters
%of the modules $Q_w$ in \Chevie\ \cite{MChv}, especially for $\SU_9(q)$ and
%$\SU_{10}(q)$.

%%%%%%%%%%%%%%%%%%%%%%%%%%%%%%%%%%%%%%%%%%%%%%%%%%%%%%%%%%%%%%%%%%%%%%%%%
\section{Projective modules from Deligne--Lusztig characters}  \label{sec:DL}

\subsection{Characters and basic sets\label{se:setting}}
Let $\ell$ be a prime number and
$(K,\mathcal{O},k)$ be an $\ell$-modular system. We assume that it is large
enough for all the finite groups encountered. Furthermore, since we will be
working with $\ell$-adic cohomology we will assume throughout this paper that
$K$ is a finite extension of $\QQ_\ell$.
\par
Let $H$ be a finite group.
A representation of $H$ will always be assumed to be finite-dimensional.
Given a simple $kH$-module $N$, we shall denote by $\vhi_N$ (resp.~$P_N$,
resp.~$\Psi_N$) its Brauer character (resp.~its projective cover, resp.~the
character of its projective cover). The restriction of an ordinary
character $\chi$ of $H$ to the set of $\ell'$-elements will be denoted
by $\chi^0$. It decomposes uniquely on the family of irreducible
Brauer characters as $\chi^0 = \sum d_{\chi,\vhi} \vhi$. The coefficients
$(d_{\chi,\vhi})_{\chi \in \Irr H, \vhi\in \IBr H}$
form the \emph{decomposition matrix} of $H$. Equivalently, if $\vhi = \vhi_N$
is the Brauer character of a simple $kH$-module $N$, then
$\Psi_N = \sum_{\chi \in \Irr H} d_{\chi,\vhi} \chi$ by Brauer
reciprocity.
\par
Recall that a \emph{basic set of characters} is a set of irreducible
characters $\cB$ of $H$ such that $\cB^0
= \{\chi^0 \, | \, \chi \in \cB\}$ is a $\ZZ$-basis
of $\ZZ \IBr H$. This means that the restriction of the
decomposition matrix to $\cB$ is invertible.
Let $\langle -,- \rangle$ be the usual inner product on the $\CC$-vector space
of class functions on $H$. For simple $kH$-modules $N$ and $M$ we have
$\langle \Psi_N,\vhi_M \rangle =
0$ if $N$ is not isomorphic to $M$ and $\langle \Psi_N,\vhi_N \rangle =1$.
This inner product can be also computed using the basic set: write
$\vhi = \sum_{\chi \in \cB} a_{\vhi,\chi} \chi^0$. Then for
any projective character $\Psi$, we have $ \langle \Psi,\vhi \rangle
= \langle \Psi, \sum_{\chi \in \cB} a_{\vhi,\chi} \chi \rangle$
since $\Psi$ vanishes outside the set of $\ell'$-elements. In other words,
$\langle \Psi,\vhi \rangle$ can be computed from the orthogonal projection
of $\Psi$ to the subspace of $\CC\Irr H$ spanned by the basic set
$\cB$ and the expression of $\vhi$ in $\cB^0$.

\subsection{Finite reductive groups\label{se:finiteredgroups}}
Let $\bG$ be a connected reductive linear algebraic group over an algebraically
closed field of positive characteristic $p$, and $F:\bG\rightarrow\bG$ be a
Steinberg endomorphism. There exists a positive integer $\delta$ such that
$F^\delta$ defines a split $\FF_{q^\delta}$-structure on $\bG$ (with
$q \in \RR_+$), and we will choose $\delta$ minimal for this property.
We set $G:=\bG^F$.
\par
Throughout this paper, we shall make the following assumptions on $\ell$:
\begin{itemize}
  \item $\ell \neq p$ (non-defining characteristic),
  \item $\ell$ is good for $\bG$ and $\ell \nmid |(Z(\bG)/Z(\bG)^\circ)^F|$.
\end{itemize}
In this situation, the unipotent characters lying in a given unipotent
$\ell$-block of $G$ form a basic set of this block \cite{GH91,Ge93}.
Consequently, the restriction of the
decomposition matrix of the block to the unipotent characters is invertible.
In particular every (virtual) unipotent character is a virtual projective
character, up to adding and removing some non-unipotent characters.
If $N$ is a simple unipotent $kG$-module and $P$ is a projective $kG$-module
with character $\Psi$, the multiplicity of $P_N$ as a direct summand of $P$
is given by $\langle \Psi, \vhi_N \rangle = \langle P , N\rangle =
\dim \Hom_{kG}(P,N)$. By extension, if $P$ is a virtual unipotent module,
we shall denote by $\langle P , N\rangle$ the multiplicity of $P_N$ in any
virtual projective module obtained from $P$ by adding or removing
non-unipotent modules (this does not depend on the modules added and removed
since $N$ is a unipotent module). Note that if we decompose $P$
and $N$ on the basic set of unipotent characters, $\langle P , N\rangle$
coincides with the usual scalar product of characters (see \S\ref{se:setting}).
\subsection{Deligne--Lusztig characters\label{se:dlchar}}
In this section we recall/prove some general properties of modules related to
Deligne--Lusztig induction.
\par
We fix a pair $(\bT,\bB)$ consisting of a maximal torus contained in a Borel
subgroup of $\bG$, both of which are assumed to be $F$-stable. We denote by $W$
the Weyl group of $\bG$, and by $S$ the set of simple reflections in $W$
associated with $\bB$. Then $F^\delta$ acts trivially on $W$.
Following Lusztig \cite[17.2]{Lu86}, for any $F$-stable irreducible character
$\chi$ of $W$, one can choose a preferred extension $\widetilde \chi$ of
$\chi$ to the group $W \rtimes \langle F \rangle$ which is trivial on
$F^\delta$. The almost character associated to $\widetilde \chi$ is then the
following uniform character
$$ R_{\widetilde \chi}
       = \frac{1}{|W|} \sum_{w\in W} \widetilde\chi(wF) R_{\bT_{wF}}^\bG (1).$$
By \cite[Thm.~4.23]{Lu84} the decomposition of almost characters in terms of
unipotent characters is known explicitly. Conversely, for $w \in W$ the
orthogonality relations for the Deligne--Lusztig characters yield
$$R_{\bT_{wF}}^\bG (1) := \sum_{i\in\ZZ} (-1)^i \big[H^i(\rX(w))\big]
       = \sum_{\chi \in (\Irr\, W)^F} \widetilde\chi(wF) R_{\widetilde\chi}$$
as a virtual $KG$-module, where $\rX(w)$ is the Deligne--Lusztig variety
associated with $w$, and $H^i(\rX(w))$ denotes the $\ell$-adic cohomology
with coefficients in $K \supset \QQ_\ell$.
\par
The Frobenius $F^\delta$ acts on the Deligne--Lusztig variety $\rX(w)$, making
the cohomology groups $H^i(\rX(w))$ into
$G\times \langle F^\delta\rangle$-modules. Digne and Michel have
extended in \cite{DM85} the previous formula to take into account this action.
By \cite[Cor.~3.9]{Lu78}, the eigenvalues of $F^\delta$ on a unipotent
character $\rho$ in the cohomology of $\rX(w)$ are of the form $\lambda_\rho
q^{\delta m/2}$ where $\lambda_\rho$ is a root of unity which depends only on
$\rho$ and $m$ is a nonnegative integer. We fix an indeterminate $v$
and we shall denote by $v^{\delta m} \rho$ the class in the Grothendieck group
of $G\times \langle F^\delta \rangle$-modules of such a representation.
\par
Let $\cH_v(W)$ be the Iwahori--Hecke algebra of $W$ with equal parameters $v$.
By convention, the standard basis $(t_w)_{w\in W}$ of $\cH_v(W)$ will satisfy
the relation $(t_s + v)(t_s - v^{-1})=0$ for all $s \in S$. For
$\chi\in(\Irr W)^F$ we denote by
$\widetilde \chi_v$ the character of $\cH_v(W) \rtimes \langle F \rangle$
which specializes to $\widetilde \chi$ at $v =1$.
The virtual $G\times \langle F^\delta \rangle$-modules afforded by the
cohomology and the intersection cohomology of Deligne--Lusztig varieties can
be computed by means of these bases. For the following result, see
\cite[\S III, Prop.~1.2, Thm.~1.3 and 2.3]{DM85}:

\begin{thm}[Digne--Michel, Lusztig]   \label{thm:dlvar}
 Let $w \in W$. The class in the Grothendieck group of
 $G\times\langle F^\delta\rangle$-modules of the cohomology of $\rX(w)$ is
 given by
 $$R_w:= \sum_{i \in \ZZ} (-1)^i \big[H^i(\rX(w))\big]
       =v^{\ell(w)}\sum_{\chi\in(\Irr\, W)^F}\widetilde\chi_v(t_w F) R_{\widetilde\chi}.$$
\end{thm}

For $\lambda \in k^\times$, we denote by $R_w[\lambda]$ the virtual $G$-module
obtained from $R_w$ by keeping only the characters on which $F^\delta$ acts
by an eigenvalue congruent to $\lambda$ modulo $\ell$. Recall from
\S \ref{se:finiteredgroups} that it is a virtual projective module, up to
adding and removing non-unipotent modules.
\smallskip

For $x,y \in W$, we write $x < y$ when $x \neq y$ and $x$ is smaller
than $y$ for the Bruhat order. Recall the following result from
\cite[Prop.~1.5]{Du13}.

\begin{prop}   \label{prop:constitPw}
 Let $w \in W$ and $N$ be a simple $kG$-module such that
 $\langle R_y , N \rangle = 0$ for all $y < w$. Then
 $\big\langle (-1)^{\ell(w)} R_w[\lambda] , N \big\rangle \geq 0$ for all
 $\lambda \in k^\times$.
\end{prop}

\begin{rem}
If $w$ is a Coxeter element of $(W,F)$ then any $y < w$ lies in a proper
$F$-stable parabolic subgroup. In that case $R_y$ is obtained by
Harish-Chandra induction. In particular, if $N$ is cuspidal then the condition
$\langle R_y , N \rangle = 0$ will be automatically satisfied for all such $y$.
In other words, when $w$ is a Coxeter element, the projective cover of any
cuspidal unipotent $kG$-module appears with nonnegative multiplicity
in $(-1)^{\ell(w)} R_w$.
\end{rem}

%%%%%%%%%%%%%%%%%%%%%%%%%%%%%%%%%%%%%%%%%%%%%%%%%%%%%%%%%%%%%%%%%%%%%%%%%
\section{Twisted conjugacy classes \label{sec:conjclasses}}

In this section we deduce from \cite{Kim,GKP00,HN12} some results on the
$F$-conjugacy classes in the symmetric group. The character of the virtual
module $R_w$ depends only on the $F$-conjugacy class of $w$, and these results
will be needed for studying the decomposition matrices of $\SU_n(q)$
from the point of view of \S\ref{se:dlchar}. We keep the notations from the
previous section.

\subsection{Ordering the conjugacy classes}
Given $x,y \in W$ and $s \in S$ we write $x \rightarrow^s y$ if $y = sxF(s)$
and $\ell(y) \leq \ell(x)$. By transitivity, we write $x \rightarrow y$ if
there exists a sequence
$x=x_0 \rightarrow^{s_1} x_1 \cdots \rightarrow^{s_r} x_r = y$.
If moreover $y  \rightarrow x$, then $\ell(x) = \ell(y)$ and we write
$x \approx y$. Note that $x \approx y$ if and only if $x$ and $y$ are
conjugate by a sequence of cyclic shifts (see \cite[Remark 2.3]{GKP00}).
\par
Given $w \in W$ we denote by $[w]_F = \{x^{-1} wF(x) \, | \, x \in W\}$ (or
sometimes only by $[w]$) its $F$-conjugacy class, and by $[W]_F$ the set of all
$F$-conjugacy classes of $W$. Given $\cO \in [W]_F$, the subset of elements
with minimal length in $\cO$ will be denoted by $\cO_\mathrm{min}$.
We say that $\cO$ is \emph{cuspidal} (or \emph{elliptic}) if it has empty
intersection with every proper $F$-stable parabolic subgroup of $W$.
The following result is proven in \cite{GP93,GKP00,He07} (see also \cite{HN12}
for a case-free proof).

\begin{thm}[Geck--Pfeiffer, Geck--Pfeiffer--Kim, He]   \label{thm:minimalelt}
 Let $\cO \in [W]_F$. Then for all $w \in \cO$ there exists
 $x \in \cO_{\mathrm{min}}$ such that $w \rightarrow x$.
\end{thm}

Following \cite[\S 4.7]{He07}, we define a partial order on the set of
conjugacy classes by $\cO' \leq \cO$ if and only if for all
$x \in \cO_{\mathrm{min}}$ there exists $x' \in \cO_{\mathrm{min}}'$ such
that $x' \leq x$ in the Bruhat order.

\begin{lem}   \label{lem:cyclicshifts}
  Let $x,y \in W$ and $s,t \in S$ be such that $y = sxt$ and
  $\ell(x) = \ell(y)$.
  If $x' \leq x$ then either $x' \leq y$ or $sx't \leq y$ with
  $\ell(sx't) \leq \ell(x')$.
\end{lem}

\begin{proof}
The result is trivial if $y = x$. Otherwise since $y = sxt$ and
$\ell(y) = \ell(x)$ we can assume without loss of generality that
$sx > x$ and $xt < x$ (see \cite[Lemma 7.2]{Hum90}).
If $x ' \leq xt$ then $x' \leq xt < sxt = y$ and we are done.
Otherwise $x' = zt$ with $z \leq xt$ and $z < x'$. Since $sxt > xt$
we have $sx't = sz \leq sxt = y$ and $\ell(sx't) = \ell(sz) \leq \ell(x')$.
\end{proof}

%\begin{rem}   \label{rem:w0}
% Multiplying on the left by $w_0$ shows that the same results holds
% if one reverses all the inequalities.
%\end{rem}

In particular, for $t = F(s)$, we deduce that if $x \rightarrow^s y$ and
$x'\leq x$ then $x'\leq y$ or $sx'F(s)\leq y$ with $\ell(sx'F(s))\leq\ell(x')$.
Together with the following theorem this yields that when comparing
the classes it is enough to look at one specific element of minimal length.

\begin{thm}[He-Nie \cite{HN12}]   \label{thm:minimalandcuspidal}
  If $\cO \in [W]_F$ is a cuspidal class then every two elements of
  $\cO_\mathrm{min}$ are conjugate by a sequence of cyclic shifts.
\end{thm}

\begin{cor}[He]
  Let $\cO$, $\cO'$ be two $F$-conjugacy classes of $W$. Then
  the following are equivalent:
  \begin{itemize}
    \item[(i)] for all $x \in \cO$, there exists $x' \in \cO'$ such that
      $x' \leq x$, and
    \item[(ii)] there exists $x \in \cO$ and $x' \in \cO'$ such that
     $x' \leq x$.
  \end{itemize}
\end{cor}

\begin{proof}
It is enough to show that (ii) implies (i). Let $I \subset S$
(resp.~$I'$) be the \emph{saturated} support of $x$ (resp.~$x'$), \emph{i.e.},
the union of $F$-orbits of simple reflections in the support of $x$.
Let $y \in \cO_\mathrm{min}$ be another element of minimal length in the class
of $x$ and let $J$ be its saturated support. Choose $w \in W$ such that
$w^{-1} x F(w) = y$ and let $z$ be the unique $I$-reduced-$J$ element
of $W_I w W_J$. Let us decompose $w$ as $w = azb$ with
$a \in W_I$, $b \in W_J$ and $az$ reduced-$J$. We set
$\widetilde x = a^{-1} x F(a) \in W_I$ and $\widetilde y = b y F(b)^{-1}
\in W_J$. Since $az$ is reduced-$J$, we have $\ell(x)
= \ell(az \widetilde y F(az)^{-1}) \geq \ell(\widetilde y)$, which
forces $\widetilde y \in \cO_\mathrm{min}$. Furthermore,
since $z^{-1}$ is reduced-$I$, we have $\ell(\widetilde y)
= \ell(z^{-1} \widetilde x F(z)) \geq \ell(\widetilde x)$ which forces
$\widetilde x \in \cO_\mathrm{min}$. By Theorem \ref{thm:minimalandcuspidal},
we deduce that $\widetilde x \approx x$ (resp. $\widetilde y \approx y$) are
conjugate by cyclic shifts. Note that since
$z \widetilde y = \widetilde x F(z)$ and $z$ is the unique element of minimal
length in $W_I z W_J$ we have $z = F(z)$.
\par
Now, by successive applications of Lemma~\ref{lem:cyclicshifts}, we can
find $\widetilde x'$ such that $\widetilde x' \leq \widetilde x$ and
$\widetilde x' \approx x'$. Since $z$ is $I$-reduced and $\widetilde x',
\widetilde x \in W_I$, we have
$\widetilde x' F(z) \leq \widetilde x F(z) = z \widetilde y$. With $z$
being also reduced-$J$, we deduce that $\widetilde x' F(z)=
z' \widetilde y'$ for some $\widetilde y' \leq \widetilde y$ and
$z' \leq z$. But $z$ is the unique element of minimal length in
$W_I z W_J$, therefore one must have $z' = z$ and
$\widetilde y' = z^{-1} \widetilde x' F(z)$ with
$\ell(\widetilde y') = \ell(\widetilde x')$.
We conclude using again Lemma~\ref{lem:cyclicshifts}
to find $y' \in \cO_{\min}'$ such that $y'\approx \widetilde y'$ and
$y' \leq y$.
\end{proof}

We have seen in \S\ref{se:dlchar} the importance of considering a module $R_w$
when $w$ is minimal for the property that a given projective
indecomposable module appears as a constituent of $R_w$. Proposition
\ref{prop:constitPw} can be rephrased as follows:

\begin{prop}   \label{prop:reformulationPw}
 Let $N$ be a simple unipotent $kG$-module such that
 $\langle R_w,N \rangle \neq 0$ for some $w \in W$. Then there exists
 a class $\cO \in [W]_F$ such that:
 \begin{itemize}
   \item[(a)] $\cO \leq [w]_F$,
   \item[(b)] $\langle (-1)^{\ell(x)}R_x,N \rangle>0$ for all $x \in \cO$, and
   \item[(c)] $\langle R_x,N \rangle = 0$ whenever $[x]_F < \cO$.
 \end{itemize}
\end{prop}

\begin{proof}
Since $R_w$ depends on $[w]_F$ only we can assume that
$w$ has minimal length in its conjugacy class. Take $z < w \in W$ to be
minimal for the Bruhat order and for the property that $\langle R_z,N \rangle
\neq 0$, and let us consider $\cO = [z]_F$. The assertions of the proposition
will follow from Proposition \ref{prop:constitPw} if we can prove that
$z \in \cO_{\mathrm{min}}$.
\smallskip

By minimality of $z$, every $z' < z$ satisfies $\langle R_{z'},N \rangle= 0$.
Let $z \rightarrow^s y$ with $\ell(y) = \ell(z)$. Then by
Lemma~\ref{lem:cyclicshifts}, the relation $y' < y$ forces $y' < z$ or
$sy'F(s) < z$, and therefore $\langle R_{y'},N \rangle= 0$ since
$R_{sy'F(s)} = R_{y'}$. Now if $z$ is not a minimal element in $\cO$ then by
Theorem~\ref{thm:minimalelt} there exists $z \rightarrow x$ with
$\ell(x) = \ell(z)$ and $t \in S$ such that $txF(t) < x$.
This would force $\langle R_{z},N \rangle = \langle R_{txF(t)},N \rangle = 0$,
hence a contradiction.
\end{proof}

\begin{rem}
 It is not clear whether $\cO$ as in Proposition~\ref{prop:reformulationPw}
 is unique, except when $\ell \nmid |G|$ (see \cite[Prop.~3.3.21]{DMR}).
\end{rem}

\subsection{The case of the symmetric group}\label{se:conjclassesSn}
From now on we shall assume that $\bG = \SL_n(\overline{\FF}_q)$, so that
$W$ can be identified with the symmetric group on $n$ letters $\fS_n$. We write
$s_i = (i,i+1)$ for $i = 1,\ldots,n-1$ and $S= \{s_1, \ldots, s_{n-1}\}$. The
twisted Frobenius structure on $\SL_n(\overline{\FF}_q)$ induces an
automorphism of the Coxeter system $(W,S)$, given by $F(w) = w_0 w w_0$,
where $w_0$ is the longest element of $W$.
The map $w \longmapsto w w_0$ induces a bijection between $F$-conjugacy
classes and usual conjugacy classes of $\fS_n$, which in turn are parametrized
by partitions of $n$. The $F$-conjugacy class corresponding to the partition
$\lambda \vdash n$ will be denote by $\cO_\la$. Under this parametrization,
cuspidal classes correspond to partitions of $n$ with only odd terms (see for
example \cite[\S 7.14]{He07}).
\par
One should be able to describe the partial ordering on conjugacy classes
introduced above in terms of the partitions. It is likely that when restricted
to the cuspidal
conjugacy classes, it coincides with the dominance order on partitions.
The following conjecture has been checked by computer for $n \leq 10$.

\begin{conj}   \label{conj:dominanceorder}
  Let $\la$, $\mu$ be two partitions having only odd parts, and
  $\cO_\la$, $\cO_\mu$ be the corresponding cuspidal classes. Then
  $\cO_\la \leq \cO_\mu$ if and only if $\mu \trianglelefteq \lambda$.
\end{conj}

\begin{rem}
Lusztig defined in \cite{Lu11} a map from the set of (usual) conjugacy classes
of any Weyl group $W$ to the set of unipotent classes of the corresponding
split reductive group $\bG$, and then generalized the construction to twisted
conjugacy classes in \cite{Lu12}. Computations in small-rank classical groups
and exceptional groups give evidence for this map to preserve the order, when
restricted to cuspidal classes. Conjecture~\ref{conj:dominanceorder} predicts
that this property holds for groups of type $\tw2A_n$.
\end{rem}

Kim \cite{Kim} gives an explicit representative of $\cO_\la$ which has
minimal length in $\cO_\la$. The element is given as a permutation but we will
need a reduced expression of it, that is in terms of the simple reflections
$s_i$. For $I \subset S$, we shall denote by $w_I$ the longest element of the
parabolic subgroup $W_I$.

\begin{lem}   \label{lem:easy}
 Let $r\ge0$ be an integer. Then
 $$\begin{aligned}
 \sigma_{2r+1}	:=&\, (1,n,2,n-1,3,\ldots,n-r+1,r+1)
                 = s_1 \cdots s_r w_0 w_{\{r+1,\ldots,n-1-r\}}, \\
 \sigma_{2r}	:=&\, (1,n,2,n-1,3,\ldots,n-r+1)
                 = s_1 \cdots s_{r-1} w_0 w_{\{r+1,\ldots,n-1-r\}}.
 \end{aligned}$$
\end{lem}

If $I = \{i,i+1,\ldots, j\}$ is a set of consecutive integers, we will denote
by $\sigma_{2r+1}^I$ (resp. $\sigma_{2r}^I$) the analogue of $\sigma_{2r+1}$
(resp. $\sigma_{2r}$) viewed as an element of the parabolic subgroup
$W_{\{i,\ldots,j\}}$.
\par
Recall from \cite{Kim} that a composition $\la = (\la_1,\la_2,\ldots,\la_m)$
of $n$ is said to be \emph{maximal} if there exists $c \in \{0,\ldots,m\}$
such that
\begin{itemize}
 \item $\la_1,\ldots,\la_c$ are even integers (in any ordering), and
 \item $\la_{c+1},\ldots,\la_m$ is a decreasing sequence of odd integers.
\end{itemize}
For such a composition, we can consider $I_j =
\{\lfloor \frac{\la_1 + \cdots + \la_{j-1}}{2}\rfloor+1,
\ldots, n-1-\lceil \frac{\la_1 + \cdots + \la_{j-1}}{2}\rceil\}$ and we
set
$$\sigma_{\la} := \sigma_{\la_1}^{I_1}\cdots \sigma_{\la_c}^{I_c}
  \sigma_{\la_{c+1}}^{I_{c+1}} F(\sigma_{\la_{c+2}}^{I_{c+2}})
  \cdots F^{m-c-1}(\sigma_{\la_{m}}^{I_m}).$$

\begin{thm}[Kim {\cite[Thm.~2.1]{Kim}}]
 For any maximal composition $\lambda$ of $n$, the element
 $\sigma_\lambda w_0$ is an element of $\cO_\la$ of minimal length.
\end{thm}

From the expression of $\sigma_{2r+1}$ one can prove the following result
towards Conjecture~\ref{conj:dominanceorder}:

\begin{prop}   \label{prop:order3columns}
  Assume $\la = 3^k 1^{n-3k}$ and let $\mu$ be a partition of $n$ with
  odd parts. Then $\cO_\la \leq \cO_\mu$ if and only if
  $\mu \trianglelefteq \la$.
\end{prop}

\begin{proof}
By \cite[Lemma 3.2]{GKP00} the lengths of the cycles
$\sigma_{\lambda_i}$ add up to the length of $\sigma_\lambda$. Therefore if
$\mu = 3^{l} 1^{n-3l}$ with $l \leq k$ then $\sigma_\mu$ is a left divisor
of $\sigma_\la$ and therefore in particular we have $\cO_\la \leq \cO_\mu$.
\par
Conversely, if $\cO_\la \leq \cO_\mu$ then the following lemma forces $\mu$ to
be of the form $3^l 1^{n-3l}$, and by the previous argument we have necessarily
$l \leq k$, otherwise $\cO_\mu$ would be strictly smaller than $\cO_\la$.
\end{proof}

\begin{lem}   \label{lem:orderfirstterm}
 Let $\la=(\la_1\geq\la_2\geq\cdots)$ and $\mu=(\mu_1\geq\mu_2\geq\cdots)$ be
 two partitions with odd parts. Then $\cO_\la \leq \cO_\mu$ forces
 $\mu_1 \leq \la_1$.
\end{lem}

\begin{proof}
Let $y \in W$ be such that $yw_0 \in \cO_\lambda$ and
$y w_0 \leq \sigma_\mu w_0$. Let $r = (\mu_1-1)/2$. From Lemma~\ref{lem:easy}
and the definition of $\sigma_\mu$, we have
$$\sigma_\mu w_0  = (s_1 \cdots s_r w_{\{r+1,\ldots,n-r-1\}} w_0) \cdot
  \Big(\prod_{j=2}^{m} F^{j-1}(\sigma_{\mu_j}^{I_j}) \Big) \cdot w_0. $$
In this decomposition each $\sigma_{\mu_j}^{I_j}$ is an element of $W_{I_j}$
with $I_j = \{\lfloor \frac{\mu_1 + \cdots + \mu_{j-1}}{2}\rfloor+1,
\ldots, n-1-\lceil \frac{\mu_1 + \cdots + \mu_{j-1}}{2}\rceil\}$. For
$j \geq 2$, we have $\mu_1 + \cdots + \mu_{j-1} \geq \mu_1 = 2r+1$, and
therefore $ \sigma_\mu w_0 \in s_1 \cdots s_r W_{\{r+1,\ldots, n-r-1\}}$.
Now, since $yw_0 \leq \sigma_\mu w_0$ and $yw_0$ does not lie in any proper
$F$-stable parabolic subgroup, we deduce that $yw_0 = s_1 \cdots s_r w$ for
some $w \in W_{\{r+1,\ldots, n-r-1\}}$.
Since $w(i) = i$ for $i \notin \{r+1,\ldots,n-r\}$, then $y$ contains a cycle
of the form $(1,n,2,n-1,\ldots,n-r+1,r+1,y(r+1), \ldots)$ which has length at
least $2r+1 = \mu_1$. This forces $\lambda_1 \geq \mu_1$.
\end{proof}

\begin{rem}
 The proof of the previous lemma does not use the fact that $yw_0$ is
 of minimal length in its conjugacy class. Note however that the assumption
 that it is cuspidal is crucial. Indeed, the order on non-cuspidal
 classes differs from the dominance order in general (see for example
 the following lemma).
\end{rem}

\begin{prop}   \label{prop:orderingfor321}
 Assume that $\la = 3^k 2 1^{n-2-3k}$. If $\cO_\la \leq \cO_\mu$, then $\mu$
 is one of the following partitions:
 \begin{itemize}
  \item[(1)] $3^{l} 2 1^{n-2-3l}$ with $l \leq k$,
  \item[(2)] $53^{l} 1^{n-5-3l}$ with $l\leq k-1$, or
  \item[(3)] $3^{l} 1^{n-3l}$.
 \end{itemize}
 Moreover, if $\mu$ is one of the partitions in (1) or (2) then
 $\cO_\la \leq \cO_\mu$.
\end{prop}

\begin{proof}
Since the partition $\la = 3^k 2 1^{n-2-3k}$ contains an even number, the
corresponding class $\cO_\la$ cannot be cuspidal. Therefore if $w \in \cO_\la$
has minimal length then there exists a proper $F$-stable subset $J$ of the
set of simple reflections such that $w \in W_J$ and $w$ is minimal
in its $F$-conjugacy class in $W_J$ (see \cite[Lemma~7.3 and Th.~7.5]{He07}).
Since $ww_J$ and $w_J w_0$ are two permutations with disjoint support,
then $ww_J$ must be a cycle of type $3^k 1^{n-2-3k}$ and $w_Jw_0$
a single transposition, which forces $J = \{s_2,\ldots,s_{n-2}\}$
and $w_J w_0 = (1,n)$.
\par
Assume now that $w \leq \sigma_\mu w_0$, with $\mu$ a maximal composition of
$n$. Then the saturated support of $\sigma_\mu w_0$ contains $J$. From
Lemma~\ref{lem:easy} we observe that $\sigma_\mu w_0 \in W_J$ if and only if
$\mu_1 = 2$. In that case $\sigma_\mu w_0 w_J$ is a permutation of type
$(\mu_2,\mu_3,\ldots)$. Moreover, since the saturated support of
$\sigma_\mu w_0$ is exactly $J$ and since it has minimal length then
$(\mu_2,\mu_3,\ldots)$ is actually a partition of $n$ with odd terms and we
deduce from Proposition~\ref{prop:order3columns} that $\mu = 3^{l}21^{n-2-3l}$
with $l \leq k$. Conversely, if $\mu = 3^{l}21^{n-2-3l}$ with $l \leq k$ then
$\sigma_\mu$ is a left divisor of $\sigma_\la$, so that $\cO_\la \leq \cO_\mu$
(as maximal compositions of $n$, $\mu$ is a truncation of $\la$).
\par
If the saturated support of $\sigma_\mu w_0$ is the set of all simple
reflections, then $\cO_\mu$ is a cuspidal conjugacy class.
If $\mu \neq 1^n$ we consider $s_1 \sigma_\mu w_0 \in W_J$. Then
$s_1 \sigma_\mu w_0w_J$ is a permutation of type $(\mu_1-2,\mu_2,\mu_3,\ldots)$.
Let $r = (\mu_1-1)/2$ and $t= (\mu_2-2)/2 \geq r$. The product of the two first
factors of $\sigma_\mu$, namely of $\sigma_{2r+1}$ and
$F\big(\sigma_{2t+1}^{\{r+1,\ldots,n-r-2\}}\big)$ is given by
$$\begin{aligned}
%%  \sigma_{2r+1} F\big(\sigma_{2t+1}^{\{r+1,\ldots,n-r-2\}}\big) &=\
  &s_1 \cdots s_r w_0 w_{\{r+1,\ldots,n-r-1\}}
    F\big(s_{r+1} \cdots s_{r+t} w_{\{r+1,\ldots, n-r-2\}} w_{\{r+t+1,\ldots,n-r-t-2\}}\big)\\
  &=\ (s_1 \cdots s_r)\cdot (s_{n-r-1} \cdots s_{n-r-t})
    \cdot (s_{r+1} \cdots s_{n-r-1}) w_{\{r+t+1,\ldots,n-r-t-2\}}w_0\\
  &= (s_1 \cdots s_{n-r-1}) \cdot (s_{n-r-2} \cdots s_{n-r-t-2})
    w_{\{r+t+1,\ldots,n-r-t-2\}}w_0.\\
\end{aligned}$$
Therefore $s_1 \sigma_\mu w_0 \in (s_2 \cdots s_{n-r-1}) \cdot (s_{n-r-2} \cdots s_{n-r-t-1})
W_{\{r+t+1,n-r-t-2\}}$. As $w \leq s_1 \sigma_\mu w_0$ is cuspidal in $W_J$,
we deduce that $w = s_2 \cdots s_r y z$ with
$y \leq s_{r+1} \cdots s_{n-r-1} \cdots s_{n-r-t-1}$ and
$z \in W_{\{r+t+1,n-r-t-2\}}$. As in the proof of Lemma~\ref{lem:orderfirstterm}
we see that $ww_J$ contains the cycle $(2,n-1,3,n-2,\ldots,r+1,w(n-r),\ldots)$.
Since by  definition $ww_J$ is a product of $3$-cycles we must have $t\leq r\leq 2$,
and we are left with two possibilities: if $r=1$ then $\mu = 3^{l}1^{n-3l}$;
if $r=2$ we must have $w(n-r) = 2$ otherwise $ww_J$ would contain a cycle of length at
least $4$. This forces $w =  s_2 \cdots s_{n-3} y'z$ with $y' \leq s_{n-4} \cdots s_{n-3-t}$
and $z \in W_{\{t+3,n-4-t\}}$. Now $t=2$ is impossible, otherwise $ww_J$ would contain either
the transposition $(n-2,4)$ (if $y' \leq s_{n-5}$) or the cycle
$(n-2,4,n-3,5,w(n-4),\ldots)$ of length at least $4$. Therefore $t\leq 1$
and $\mu = 53^{l}1^{n-5-3l}$. In that case $s_1 \sigma_\mu w_0 w_J$ is an element
of minimal length in the cuspidal conjugacy class of $W_J$ corresponding to
the partition $3^{l+1}1^{n-5-3l}$, therefore we must have $\ell+1 \leq k$
by Proposition~\ref{prop:order3columns}. Conversely, for such a partition
$\mu$ one has $s_1 \sigma_\mu w_0 w_J = \sigma_{3^{l+1}1^{n-5-3l}} \leq
\sigma_{3^{k}1^{n-2-3k}}  = \sigma_\la w_0 w_J$ so that $\cO_\la \leq \cO_\mu$.
\end{proof}

\begin{rem} \label{rem:partitions321}
 It is likely that if $\la = 3^k21^{n-2-3k}$ and $\mu = 3^{l}1^{n-3l}$ then
 $\cO_\la \leq \cO_\mu$ if and only of $l \leq k+1$. This has been checked with
 $\Chevie$ for $n \leq 18$.
\end{rem}

%%%%%%%%%%%%%%%%%%%%%%%%%%%%%%%%%%%%%%%%%%%%%%%%%%%%%%%%%%%%%%%%%%%%%%%%%
\section{General results for $\SU_n(q)$}   \label{sec:general}
We now consider decomposition numbers of special unitary groups $\SU_n(q)$
for unitary primes $\ell$. Recall that a prime $\ell$ is \emph{unitary} for
$\SU_n(q)$ if the order of $-q$ modulo $\ell$ is odd. In the case of linear
primes, it is known that the unipotent part of the decomposition matrix of
$\SU_n(q)$ is the same
as that of the so-called $q$-Schur algebra, which gives an easy way to compute
it (see \cite{GrHi}). No analogous approach is known for unitary primes.
\par
Throughout, we will assume that $\ell>n$, for the following reasons. It
is known that the decomposition matrix for $\ell\le n$ will in general be
different, since
\begin{itemize}
 \item[(a)] the decomposition matrices for Hecke algebras do change when
  $\ell$ divides the order of the Weyl group,
 \item[(b)] certain $\ell$-elements which force relations on decomposition
  numbers do not exist, and
 \item[(c)] for $\ell|n$, the unipotent characters will no longer form a
  basic set for the unipotent blocks of $\SU_n(q)$, so that even the indexing
  sets for the decomposition matrix change.
\end{itemize}
Under our assumption on $\ell$, it turns out that in all examples the
$\ell$-modular decomposition matrix of the unipotent characters only depends
on the order $d_\ell(-q)$ of $-q$ modulo $\ell$, but not on $q$ and $\ell$
individually.

\subsection{Unipotent characters of unitary groups}
Recall that the set of unipotent characters of a finite reductive group depends
only on its isogeny class. Here we have the following stronger statement:

\begin{prop}
 Let $\ell>n$. Then the unipotent characters form a basic set for the unipotent
 blocks of $\SU_n(q)$, $\GU_n(q)$ and $\PGU_n(q)$, and the corresponding
 square part of the $\ell$-modular decomposition matrix is the same for all
 three groups for any fixed $n$ and $q$.
\end{prop}

\begin{proof}
The unipotent characters form a basic set for the unipotent blocks in all three
cases by \cite[Thm.~A]{Ge93}. Since unipotent characters restrict irreducibly
from $\GU_n(q)$ to $\SU_n(q)$, the statement on decomposition matrices for that
pair follows. Moreover, unipotent characters have the center of $\GU_n(q)$
in their kernel, so can be considered as characters of $\PGU_n(q)$ as well.
This completes the proof.
\end{proof}

We thus may and will switch freely between $\SU_n(q)$ and $\GU_n(q)$ in our
proofs.

\smallskip

Recall that the irreducible unipotent characters (resp.~the
irreducible characters of the symmetric group $\fS_n$) are parametrized
by partitions $\mu\vdash n$. We shall denote the corresponding character
by $\rho_\mu$ (resp. $\chi_\mu$). The valuation $a(\mu)$ (resp. the degree
$A(\mu)$) of the polynomial degree of $\rho_\mu$ can be explicitly computed
in terms of $\mu$ (see \cite[\S 4.4]{Lu84}).
\par
The Frobenius endomorphism $F$ acts on the Weyl group $W$ by $F(w) = w_0 w w_0$,
where $w_0$ is the longest element of $W$. In particular, every
irreducible character of $W$ is stable by $F$. Following \cite[17.2]{Lu86},
one can choose a preferred extension $\widetilde \chi_\mu$ of
$\chi_\mu$ to the group $W \rtimes \langle F \rangle$ which is trivial on
$F^\delta$. It is defined by the property that $\widetilde \chi_\mu(wF) =
(-1)^{a(\mu)} \chi_\mu(ww_0)$. Up to a sign, the almost character corresponding
to $\widetilde \chi_\mu$ is the unipotent character $\rho_\mu$.

\begin{lem}   \label{lem:almostchar}
 We have $R_{\widetilde \chi_\mu}= (-1)^{a(\mu)+A(\mu)} \rho_{\mu}$.
\end{lem}

\begin{proof}
We compare the value at the identity element on both sides. For this, let
$f\mapsto f^{(-)}$ denote the evaluation at $-q$ on $\QQ[q]$. We have
$R_{\bT_{wF}}^\bG(1)(1)=R_{\bT_{ww_0}}^\bG(1)(1)^{(-)}$, and also
$\widetilde\chi_\mu(wF)=(-1)^{a(\mu)}\chi_\mu(ww_0)$ (see \cite[17.2]{Lu86}).
Thus
$$R_{\widetilde\chi_\mu}(1)
  =\frac{(-1)^{a(\mu)}}{|W|}\left(\sum_{w\in W}
   \chi_\mu(ww_0)R_{\bT_{ww_0}}^\bG(1)(1)\right)^{(-)}
  =(-1)^{a(\mu)}\psi_\mu(1)^{(-)},$$
where $\psi_\mu$ is the unipotent character of $\GL_n(q)$ indexed by $\mu$.
The claim follows.
\end{proof}

In particular, every unipotent character is uniform (\emph{i.e.}, a linear
combination of Deligne--Lusztig characters $R_w$).

\subsection{$\ell$-reduction of characters}
It was conjectured by Geck \cite{GeThesis} in general and shown for $\GU_n(q)$
in \cite{Ge91} that the $\ell$-modular reduction
of an ordinary cuspidal unipotent character is irreducible; for unitary groups
we have the following stronger statement:

\begin{prop}   \label{prop:smallirr}
 Let $\rho$ be a unipotent character of the unitary group $\SU_n(q)$ which has
 minimal $a$-value in its (ordinary) Harish-Chandra series. Then the
 $\ell$-modular reduction of $\rho$ is irreducible.
\end{prop}

\begin{proof}
Let $\mu$ be a partition of $n$ and $\mu^\star$ be the conjugate partition.
Let $C_{\mu}$ be the unipotent class of $G=\SU_n(q)$ with Jordan form
$\mu^\star$. Recall from
\cite[\S 6.4]{GHM} that the smallest split Levi subgroup $\bL$ of $\bG$ such
that $\bL \cap C_{\mu} \neq \emptyset$ is called the \emph{Gelfand--Graev
vertex} of $C_\mu$. By \cite[Prop. 5.5]{GHM}, one can choose
$u \in \bL \cap C_{\mu}$ such that if $\Gamma_u$ is the generalized
Gelfand--Graev representation of $\bL^F$ associated with $u$ then
\begin{itemize}
 \item $\rho_\mu$ occurs with multiplicity one in the character of
  $\RLG(\Gamma_u)$, and
 \item if $\rho$ is an irreducible constituent of $\RLG(\Gamma_u)$ then
  there exists a unipotent element $x \in G$ such that $\rho(x)\neq 0$ and
  $C_\mu\subset \overline{\bG\cdot x}$. In particular, if $\rho = \rho_\lambda$
  is a unipotent character, this forces $\lambda \trianglelefteq \mu$.
\end{itemize}
\par
If $P_\mu$ denotes the unique indecomposable summand of $\RLG(\Gamma_u)$ which
involves $\rho_\mu$ in its character then the map $\mu \longmapsto P_\mu$ gives
a bijection between the partitions of $n$ and the isomorphism classes of
projective indecomposable modules lying in the sum of the unipotent blocks of
$\SU_n(q)$.
\par
Let $(L,\eta)$ denote a Harish-Chandra source of $\rho$, that is, $L$ is a
Levi subgroup of $G$ with a cuspidal unipotent character $\eta$ such that
$\rho$ occurs in $\RLG(\eta)$. Thus $\eta$ is parametrized by a
triangular partition $\lambda = (d,d-1,\ldots,3,2,1)$, and $\rho$, having
minimal $a$-value in the $(L,\eta)$-series, is parametrized by
$\lambda' = (d+m,d-1,\ldots,3,2,1)$, where
$m = n - |\lambda| = n - d(d+1)/2$ is even.
\par
Let $\mu$ be a partition of $n$ different from $\lambda'$ and such that
$\lambda' \trianglelefteq \mu$. The Gelfand--Graev vertex of $C_\mu$ is
obtained as follows: we write $\mu = \widetilde \mu + 2\nu$ where the
dual partition of $\widetilde \mu$ has distinct terms. Then the Gelfand--Graev
vertex of $C_\mu$ has type $\tw2A_{|\widetilde\mu|-1}(q)\times A_{\nu_1-1}(q^2)
\times A_{\nu_2-1}(q^2) \times \cdots \times A_{\nu_r-1}(q^2)$. Now since
$\lambda' \trianglelefteq \mu$ and $\mu^\star$ is the concatenation of
$(\widetilde \mu^\star,\nu^\star,\nu^\star)$ we deduce that $\widetilde
\mu^\star \trianglelefteq \lambda'^\star$. In particular, the largest
term in $\widetilde \mu^\star$ is less than $d$ and since
$\widetilde \mu^\star$ has distinct terms we must have $|\widetilde\mu^\star|
= |\widetilde\mu| \leq d(d+1)/2$ with equality if and only if
$\widetilde \mu^\star = \widetilde \mu = \lambda$. In that case the size
of $\nu$ is exactly $m/2$ and $\lambda_1' \leq \mu_1 = \widetilde \mu_1
+ 2 \nu_1$ forces $\nu_1 = m/2$ and therefore $\mu = \lambda$. Since
this is impossible, it proves that the Gelfand--Graev vertex
of $C_\mu$ cannot contain $\bL$ or any of its rational conjugates.
Consequently, $R_\bL^\bG(\Gamma_u)$ (and hence $P_\mu$) cannot
have any constituent lying in the Harish-Chandra series of $(L,\eta)$
and in particular the decomposition number $d_{\lambda',\mu}$ must be zero.
\end{proof}

%%%%%%%%%%%%%%%%%%%%%%%%%%%%%%%%%%%%%%%%%%%%%%%%%%%%%%%%%%%%%%%%%%%%%%%%%
\section{The special case $\ell \mid (q+1)$   \label{sec:qplus1}}

Throughout this section we will assume that $q\equiv-1\pmod\ell$ with $\ell>n$.

\subsection{Non-unipotent characters}
Let $\bla = (\la^1,\ldots, \la^m)$ be a multipartition of $n$, that is, an
$m$-tuple of partitions $\la^i$ of size $n_i$ such that $\sum n_i = n$.
We assume here that $\bG = \GL_n$, with $G:=\bG^F=\GU_n(q)$, so that one can
identify $(\bG,F)$ with its Langlands dual $(\bG^*,F^*)$. Note also that
centralizers of semisimple elements of $\bG$ are automatically connected.
Recall that $\bT$ denotes a quasi-split maximal torus of $(\bG,F)$. Since
$\ell > n$, there exists a semisimple $\ell$-element of $\bT^{w_0 F}$ such
that $(C_\bG(s),w_0 F)$ is a connected reductive group of semisimple type
$\tw2A_{n_1-1}(q) \times \cdots \times \tw2A_{n_m -1}(q)$
(by taking $m$ distinct
eigenvalues for $s$ with respective multiplicities $n_1,\ldots,n_m$).
The Weyl group $W(s)$ of $C_\bG(s)$ is just the Young subgroup
$\fS_{n_1}\times \cdots \times \fS_{n_m}$ of $\fS_n$.
Let $w_s$ be the longest element of $W(s)$ and let $F_s$ be the
Frobenius endomorphism of $C_\bG(s)$ induced by $w_s w_0 F$.
Then $\bT$ is a quasi-split torus of
$(C_\bG(s),F_s)$ and we can consider Deligne--Lusztig characters
$R_{\bT_{wF_s}}^{C_\bG(s)}(1)$ for various $w \in W(s)$. Let $\rho_{\bla}^s$
be the unipotent character of $C_\bG(s)^{F_s}$ corresponding to $\bla$.
As a uniform function it can be written by Lemma \ref{lem:almostchar} as
$$\begin{aligned}
\rho_{\bla}^s 	=&\, (-1)^{a(\bla)+A(\bla)} R_{\widetilde \chi_{\bla}} \\
		=&\, (-1)^{a(\bla)+A(\bla)} \frac{1}{|W(s)|}
	             \sum_{w \in W(s)} \widetilde \chi_{\bla} (w F_s)
	R_{\bT_{wF_s}}^{C_\bG(s)}(1)\\
		=&\, (-1)^{A(\bla)} \frac{1}{|W(s)|}
	\sum_{w \in W(s)} \chi_{\bla} (ww_s) R_{\bT_{wF_s}}^{C_\bG(s)}(1). \\
\end{aligned}$$
We denote by $\rho_{\bla}$ the irreducible character of $G$ corresponding
to $\rho_{\bla}^s$ via the Jordan decomposition. Note that by
\cite[Thm.~13.23]{DM91} the virtual character $R_{\bT_{wF_s}}^{C_\bG(s)}(1)$
corresponds to $(-1)^{\ell(w_s w_0)} R_{\bT_{ww_sw_0F}}^{\bG}(\theta)$ where
$\theta$ is an $\ell$-character of $\bT^{wF_s}$ corresponding to $s$.
\par
Let us consider the partition $\la$ of $n$ obtained by concatenation of the
$\la^i$'s (equivalently, its dual $\la^\star$ is the sum of the dual partitions
${\la^i}^\star$). Then:
\begin{itemize}
\item[(a)] The unipotent support of $\rho_{\bla}$ is the unipotent class
 with Jordan normal form $\la^\star$. We deduce from \cite[2.4]{Ge93} that the
 irreducible Brauer character $\vhi_{\la}$ appears with multiplicity one in
 $\rho_{\bla}^0$ and that no other irreducible Brauer character
 $\vhi_{\mu}$ with $a(\mu) \geq a(\la)$ can appear.
\item[(b)] The restriction $\rho_{\bla}^0$ of $\rho_{\bla}$ to the set of
 $\ell'$-elements can be computed
 in terms of restrictions of unipotent characters by
 \begin{equation}\begin{aligned}
   \rho_{\bla}^0 =& \ \Big( \frac{(-1)^{A(\bla)+\ell(w_s w_0)}}{|W(s)|}
     \sum_{w \in W(s)} \chi_{\bla} (ww_s) R_{\bT_{ww_sw_0F}}^{\bG}(1)\Big)^0 \\
     =& \  \frac{(-1)^{A(\bla)+\ell(w_s w_0)}}{n_1! \cdots n_m!}\sum_{(w_1,\ldots,w_m)
     \in \fS_{n_1}\times \cdots \times \fS_{n_m}} \chi_{\la^1}(w_1) \cdots \chi_{\la^m}(w_m)
      \Big(R_{\bT_{w_1 \cdots w_m w_0F}}^\bG(1)\Big)^0.
\end{aligned}\label{eq:lreduction}\end{equation}
\end{itemize}
As a consequence of the orthogonality relations of Deligne--Lusztig characters,
$\rho_{\bla}^0$ will be orthogonal to many virtual projective
characters $R_w$, especially when $\chi_{\bla}$ vanishes on many conjugacy
classes, which is the case for triangular partitions.

\smallskip
Now assume that all $\la^j$ are triangular partitions. Then up to permutation,
the multipartition $\bla=(\la^1, \ldots,\la^m)$ is uniquely determined by $\la$.
For each $\la^j$ one can consider the partition
$(n_j)$ if $n_j$ is odd or $(n_j-1,1)$ otherwise, and
we shall denote by $\bar\bla$ their concatenation. It is a partition of $n$
with no even terms, and as such it corresponds to a cuspidal $F$-conjugacy
class of $\fS_n$ (see \S \ref{se:conjclassesSn}). Let us define
$$\cC_{\lambda}
  = \big\{ [ww_0]_F \in [\fS_n]_F \, | \,  w \in \fS_{n_1}\times \cdots
    \times \fS_{n_m} \ \text{and} \ \chi_{\bla}(w) \neq 0\big\}.$$
Since a triangular partition is a $2$-core, every character $\chi_{\lambda^j}$
is of 2-defect zero, and thus we deduce that $\cC_{\lambda}$ contains only
cuspidal classes. By~(\ref{eq:lreduction}) these are precisely the classes
of elements $ww_0$ such that $(R_{\bT_{w w_0F}}^\bG)^0$ occurs in
$\rho_{\bla}^0$. Under rather strong assumptions on $\lambda$ and
$\cC_\lambda$, one can show that $\rho_{\bla}^0 = \vhi_{\la}$.

\begin{prop}  \label{prop:nonunip}
 Let $\la$ be a partition of $n$ such that $\lambda^\star$ is a sum of
 triangular partitions. We assume that:
 \begin{itemize}
   \item[(i)] every element of $\cC_\la$ is of the form
     $[\sigma_{\bar \bmu}w_0]_F$ for some $\mu \trianglelefteq \la$ such that
     $\mu^\star$ is a sum of triangular partitions,
  \end{itemize}
  and for every such $\mu$ we assume that
 \begin{itemize}
  \item[(ii)] every element of $\cC_\mu$ is of the form
     $[\sigma_{\bar \bnu}w_0]_F$ for some $\nu \trianglelefteq \mu$ such that
     $\nu^\star$ is a sum of triangular partitions,
  \item[(iii)] $\cC_\mu = \{\cO \text{ cuspidal } \mid [\sigma_{\bar\bmu}w_0]_F\leq \cO\}$.
 \end{itemize}
 Then $\rho_{\bla}^0 = \vhi_{\la}$.
\end{prop}

The proof is based on an argument that we shall use in many other situations.

\begin{lem}   \label{lem:gettingrelations}
  Let $\rho$ be character of $\SU_n(q)$ and $\cC$ be a set of $F$-conjugacy
  classes such that $\rho^0 = \sum_{[w]_F \in \cC} m_w (R_{\bT_{wF}}^\bG)^0$
  for some rational numbers $m_w$.
  Let $\mu$ be a partition of $n$ and $x \in W$ such that
  \begin{itemize}
    \item $P_\mu$ occurs in the decomposition of $R_x$, and
    \item for all $\cO \in \cC$ we have $\cO \nleq [x]_F$.
  \end{itemize}
  Then $\langle P_\mu, \rho^0 \rangle = 0$. In other words, $\vhi_\mu$ does not occur in
  the $\ell$-restriction of $\rho$.
\end{lem}

\begin{proof}
For each $x \in W$ we may and will choose a projective character
$\widetilde R_x$ whose unipotent part is $R_x$. Let $\cC_{\nleq}$ be the set
of classes $[x]_F$ such that $\cO \nleq [x]_F$ for all $\cO \in \cC$.
It is clearly closed under the partial ordering on $F$-conjugacy classes.
In other words, if $[y]_F \leq [x]_F$ and $[x]_F \in \cC_{\nleq}$ then
$[y]_F \in \cC_{\nleq}$. By induction on $[x]_F \in \cC_{\nleq}$, we show that a
relation $\langle R_x,\vhi_\mu\rangle \neq 0$ forces $\langle P_\mu, \rho^0 \rangle =0$.
For $[x] \in \cC_{\nleq}$, the orthogonality relations of Deligne--Lusztig
characters yield
\begin{equation}\label{eq:scalar}
 0 = \langle (-1)^{\ell(x)} \widetilde R_x, \rho_{\bla}^0\rangle
  = \sum_{\mu \vdash n} \langle (-1)^{\ell(x)} \widetilde R_x,
    \vhi_\mu \rangle \langle P_\mu, \rho_{\bla}^0 \rangle.
\end{equation}
Assume by induction that for any $[y]_F < [x]_F$ and any $\mu$ such that
$\langle R_y,\vhi_\mu\rangle \neq 0$ we have $\langle P_\mu,\rho^0\rangle=0$.
If $\langle R_x,\vhi_\mu\rangle \neq 0$, then either $P_\mu$ already occurred
in a Deligne--Lusztig character $R_y$ for $[y]_F < [x]_F$, in which case
$\langle P_\mu, \rho^0 \rangle =0$, or by Proposition~\ref{prop:constitPw}
we have $\langle (-1)^{\ell(x)} \widetilde R_x,\vhi_\mu \rangle > 0$.
In any case, equation~(\ref{eq:scalar}) gives a sum of nonnegative numbers
which add up to zero, therefore $\langle P_\mu, \rho^0 \rangle =0$ must be
zero whenever $\langle \widetilde R_x,\vhi_\mu \rangle$ is not.
\end{proof}

\begin{proof}[Proof of the proposition]
We want to show that $\langle P_\mu,\rho_{\bla}^0\rangle =
\delta_{\lambda,\mu}$. Since $\rho_{\bla}$ is cuspidal, any irreducible
Brauer character occurring in $\rho_{\bla}^0$ is cuspidal, so we only need to
consider partitions $\mu$ such that $\mu^\star$ is multiplicity-free, which
we will assume now.
\par
Lemma~\ref{lem:gettingrelations} applied to $\cC = \cC_\la$, together with the
description (iii) shows that $\langle P_\mu,\rho_{\bla}^0\rangle = 0$ whenever
$P_\mu$ occurs in $R_w$ with $[w]_F \notin \cC_\la$.
Since every unipotent character of $\GU_n(q)$ is a uniform function, the
virtual modules $R_{w}$ for $[w]_F \notin \cC_\la$ span a subspace of
$K_0(kG$-proj$)$ of codimension $|\cC_\la|$. Consequently, there are at most
$|\cC_\la|$ projective indecomposable modules $P_\mu$ which do not occur as
constituents in any $R_w$ for $[w]_F \notin \cC_\la$.
\par
Since $\rho_{\bla}$ has unipotent support corresponding to $\la$, the
coefficient of $\vhi_\la$ on $\rho_{\bla}^0$ is equal to $1$, hence non-zero.
According to the previous argument, this proves that $P_\lambda$ cannot appear
as a constituent of a $R_w$ for $[w]_F \notin \cC_\la$. Now by (i) every
element of $\cC_\la$ is of the form $[\sigma_{\bar \bmu} w_0]_F$.
If we apply the previous argument to $\mu$, we deduce that $P_\mu$ cannot
appear as a constituent of an $R_w$ for $[w]_F \notin \cC_\mu$. It is clear
from the description (iii) that $\cC_\mu \subset \cC_\la$, so that the
$P_\mu$'s for $[\sigma_{\bar \bmu} w_0]_F \in \cC_\la$ are exactly the
projective modules which do no appear. Consequently, all the projective covers
of cuspidal modules $P_\mu$ with $\mu {\not\trianglelefteq}\la$ have to occur
in an $R_w$ for $[w]_F \notin \cC_\la$. This proves that when
$\mu {\not\trianglelefteq}\la$ the coefficient of $\vhi_\mu$ on
$\rho_{\bla}^0$ is zero, so that by (a) above
 we must have $\rho_{\bla}^0 = \vhi_\la$.
\end{proof}

\begin{rem}
The set of partitions $\lambda$ such that $\lambda^\star$ is multiplicity-free
and the set of partitions with odd terms have the same cardinality.
It would be interesting to see whether there exists a bijection $f$ with
$f(\la) = \bar \bla$ when $\la^\star$ is a sum of triangular partitions
and such that $[\sigma_{f(\lambda)} w_0]_F$ is minimal among the classes
$[w]_F$ such that $P_\lambda$ occurs in $R_w$.
\end{rem}

\begin{prop}   \label{prop:partcases}
 The following partitions satisfy the assumptions of
 Proposition~\ref{prop:nonunip}:
 \begin{itemize}
  \item[(1)] The partitions $\lambda = 2^b1^{n-2b}$ with $n \geq 3b$.
  \item[(2)] The partitions $32^i1^j$ with $1\leq i \leq 3$ and $i \leq j \leq 10$
  or $i=4$ and $4 \leq j\leq 7$.
  \item[(3)] The partitions $3^22^21^k$ with $2 \leq k \leq 8$.
  \end{itemize}
\end{prop}

\begin{proof}
First assume that $\lambda = 2^b1^{n-2b}$ with $n \geq 3b$. We first note that
any partition $\mu \trianglelefteq \la$ has the same shape, so we only need to
check the descriptions (i) and (iii) for $\cC_\la$. The multipartition $\bla$
is $(21,21,\ldots,21,1,\ldots,1)$, $\bar \bla = 3^{b}1^{n-3b}$. Since the
character $\chi_{21}$ of $\fS_3$ takes non-zero values exactly on the classes
corresponding to $3$ and $1^3$, we deduce that $\cC_\lambda =
\{[\sigma_{3^{c}1^{n-3c}} w_0]_F \, | \, c\leq b\}$, which proves (i). The
description (iii) follows from Proposition \ref{prop:order3columns}.
\par
For (2) every partition $\mu \trianglelefteq 32^i1^j$ such that
$\mu^\star$ is a sum of triangular partitions has shape (1) or (2).
Here $\bla = (321,21,21,\ldots)$ and $\bar \bla = 53^{i-1}1^{j-i+1}$.
Therefore $[\sigma_\mu w_0]_F \in \cC_\la$ if and only if
$\mu$ is the concatenation of a partition with odd terms of $6$
and $i-1$ partitions with odd terms of $3$ (completed with $1$'s
to get a partition of $n$). Therefore $\mu$ has shape
$53^{i'}1^{j'}$ with $i'\leq i-1$ or $3^{i'}1^{j'}$ with
$i' \leq i+1$. Since $3^{i+1} \trianglelefteq 53^{i-1}1$, these partitions
$\mu$ are exactly the partitions with odd terms such that $\mu
\trianglelefteq \bar \bla$. Moreover, to any such $\mu$ corresponds
a unique partition $\nu \trianglelefteq \la$ such that $\nu^\star$
is a sum of triangular partitions and $\bar \bnu =\mu$, with
$\bnu = (321,21,21,\ldots)$ if $\mu = 53^{i'}1^{j'}$ and
$\bnu = (21,21,\ldots)$ if $\mu = 3^{i'}1^{j'}$. This proves that $\cC_\la$
satifies the description (i) of Proposition~\ref{prop:nonunip}.
The description (iii) follows from Conjecture~\ref{conj:dominanceorder}
which can be checked with the help of \Chevie\ as long as
$i$ and $j$ are small.
\par
The same argument applies to $\la = 3^22^21^k$. In this case
$\bla = (321,321,1,1,\ldots)$, so that $\bar \bla =5^21^k$.
The partitions $\mu$ with odd terms such that $\mu \trianglelefteq
\bar\bla$ are
$$\mu \in \{5^21^k, 53^21^{k-1},531^{k+2}, 51^{k+5}, 3^41^{k-2},
  3^31^{k+1}, 3^21^{k+4},31^{k+7}, 1^{k+10}\}.$$
They are in bijection with the set of partitions
$$ \nu \in \{3^22^21^k, 32^31^{k+1},32^21^{k+3},321^{k+5},2^41^{k+2},
2^31^{k+4},2^21^{k+6},21^{k+8},1^{k+10}\}$$
via the map $\nu \longmapsto \bar \bnu$. Therefore $\cC_\la$
satisfies property (i). Again, (iii) follows from
Conjecture~\ref{conj:dominanceorder} which holds whenever
$k$ is small.
\end{proof}

\begin{rem}   \label{rem:firstprojectives}
 In the proof of Proposition~\ref{prop:nonunip} we show that $P_\la$ does not
 occur in $R_w$ for $[w]_F \notin \cC_{\la}$. Therefore with
 $\la = 2^b1^{n-2b}$ and $n \geq 3b$, we obtain from the previous proposition
 that if $P_{2^b1^{n-2b}}$ occurs in $R_w$, then $ww_0$ is conjugate to
 a cycle of type $3^c1^{n-3c}$ with $c \leq b$. From the particular case $b=0$
 we deduce that $P_{1^n}$ occurs only in $R_{w_0}$.
\end{rem}

The previous strategy can be extended to deal with non-cuspidal
representations. An example is given in the following lemma.

\begin{lem}   \label{lem:partition31}
  Assume that $n \geq 4$.
  Let $I = \{s_2,\ldots, s_{n-2}\}$ and $\bL$ be the corresponding
  standard Levi subgroup of $\bG$. Then the only irreducible Brauer characters
  which can appear in $R_{\bL}^\bG(\rho_{(1,1,\ldots,1)})^0$ are
  $\vhi_{1^n}$, $\vhi_{21^{n-2}}$ and $\vhi_{31^{n-3}}$.
\end{lem}

\begin{proof}
Let $\rho = R_{\bL}^\bG(\rho_{(1,1,\ldots,1)})$. By definition of
$\rho_{(1,1,\ldots,1)}$ we have
$$R_{\bL}^\bG(\rho_{(1,1,\ldots,1)})^0 = R_{\bL}^\bG({\rho_{(1,1,\ldots,1)}}^0)
  = R_{\bL}^\bG((-1)^{\ell(w_0)-\ell(w_I)} R_{\bT_{w_IF}}^\bL(1)^0)
  = -(R_{w_IF})^0$$
where $w_I$ is the longest element of $W_I$.
From Lemma~\ref{lem:gettingrelations}, we deduce that $\langle P_\la,
\rho^0 \rangle = 0$ whenever $P_\la$ occurs in $R_w$ for some
$[w]_F \ngeq [w_I]_F$.
\par
Now, the permutation $w_I w_0$ is just the transposition $(1,n)$, which is a
cycle of type $21^{n-2}$ with maximal length in its conjugacy class (equal to
$2n-3$). Let $\mu$ be a partition of $n$ and $\sigma_\mu$ be the cycle of
type $\mu$ defined in Section~\ref{sec:conjclasses}. Using the explicit
formulae for the length of $\sigma_\mu$ given in \cite[Lemma 3.2]{GKP00},
one can check that $\ell(\sigma_\mu) < 2n-3$ if and only if
$\mu \in \{1^n,31^{n-3}\}$. In particular, the only $F$-conjugacy classes that
are larger than $[w_I]_F$ correspond to the partitions $1^n$ and $31^{n-3}$
(compare with Proposition~\ref{prop:orderingfor321}).
\par
Finally, the span of
$\cR=\{R_{\sigma_\mu w_0}\}_{\mu\notin\{1^n,21^{n-2},31^{n-3}\}}$
has codimension $3$ in the space of projective unipotent characters,
therefore there are at most three partitions $\lambda$ such that
$P_\la$ does not occur in these $R_w$'s. By Remark~\ref{rem:firstprojectives},
we know already that $P_{1^n}$ (resp. $P_{21^{n-2}}$) occurs only in $R_{w_0}$
(resp. in $R_{w_0}$ and $R_{\sigma_{31^{n-3}}w_0}$).
Furthermore, $P_{31^{n-3}}$ cannot occur in any element of $\cR$
otherwise by the previous argument $\langle P_{31^{n-3}}, \rho^0 \rangle$
would be zero, which is impossible since the unipotent part
of $P_{31^{n-3}}$ equals $R_\bL^\bG(\rho_{1^{n-2}})$ by
\cite[Prop. 4.4 and Lemma 4.6]{GHM2} and
$$\begin{aligned}
  \langle R_\bL^\bG(\rho_{1^{n-2}}), -R_{w_I} \rangle
    =&\ - \langle \rho_{1^n}+\rho_{2^21^{n-4}}+\rho_{31^{n-3}},R_{w_I}\rangle\\
    =&\ -\chi_{1^{n}}(w_I w_0) - \chi_{2^21^{n-4}}(w_Iw_0)
		+ \chi_{31^{n-3}}(w_Iw_0) \\
    =&\ 1+\Big(\frac{(n-2)(n-5)}{2}+1\Big)+\Big(-\frac{(n-3)(n-4)}{2}+1\Big)\\
    =&\ 2
\end{aligned}$$
(see the proof of Proposition~\ref{prop:anotherprojective} for the values of
characters of $\fS_n$ on the transposition $w_Iw_0$). We conclude that for all
$\la \notin \{1^n, 21^{n-2}, 31^{n-3}\}$ we have
$\langle P_\la, R_{\bL}^\bG(\rho_{(1,1,\ldots,1)})^0 \rangle =0$, which means
that $\vhi_\la$ is not a constituent of $R_{\bL}^\bG(\rho_{(1,1,\ldots,1)})^0$.
\end{proof}

\begin{rem}
 We have $w_I < s_1w_I$ and $s_1 w_Iw_0 = (1,n,2)$ has maximal length in its
 conjugacy class. So by definition of the order on $F$-conjugacy classes we
 deduce that $[w_I]_F = \cO_{21^{n-2}} < \cO_{31^{n-3}} = [s_1 w_I]_F$
 although $21^{n-2} \trianglelefteq 31^{n-3}$. However this does not contradict
 Conjecture~\ref{conj:dominanceorder} since $[w_I]_F$ is not cuspidal.
\end{rem}

\begin{rem}
 Proposition~\ref{prop:nonunip} and Lemma~\ref{lem:partition31} can be used
 to compute decomposition numbers as follows. Let $\lambda$ be a partition
 of $n$ and $\Psi_{\la}$ be the character of the corresponding PIM. By Brauer
 reciprocity, the unipotent part of $\Psi_{\la}$ is
 $\sum d_{\mu,\la} \rho_{\mu}$. Let $\rho$ be a (non-necessarily unipotent)
 character such that $\rho^0 = \sum x_\mu \rho_{\mu}^0$. Then if one knows
 that $\vhi_\la$ is not a constituent of $\rho^0$ we obtain the relation
 $\langle \Psi_{\la},\rho^0 \rangle = \sum d_{\mu,\la} x_\mu = 0$
 on the decomposition numbers.
\end{rem}

\subsection{Two-columns partitions}
We apply the previous results to compute the bottom-right corner of the
decomposition matrix of $\SU_n(q)$ (corresponding to two-column partitions)
in the case where $\ell \mid q+1$.

\begin{thm}   \label{thm:corner}
 Assume $\ell>n$ and $\ell \mid q+1$. Let $a:=\lfloor n/3\rfloor+1$. Then the
 unipotent parts of the projective indecomposable $\SU_n(q)$-modules indexed by
 the partitions $2^b 1^{n-2b}$ for $b \leq a$ are given by the columns in
 the following matrix (where dots ``." represent zeroes):
 $$\begin{array}{c|cccccc}
  2^{a}1^{n-2a}		&1			&.			&.	&.		&.	&.\\
  2^{a-1}1^{n-2a+2}	&n-2a+1			&1			&.	&.		&.	&.\\
  \vdots			&\vdots		&\ddots			&\ddots	&.		&.	&.\\
  2^2 1^{n-4}		&\binom{n-a-2}{a-2}&\cdots			&n-5	&1		&.	&.\\
  21^{n-2}		&\binom{n-a-1}{a-1}&\binom{n-a}{a-2}	&\cdots	&n-3		&1	&.\\
  1^n			&\binom{n-a}{a}	&\binom{n-a+1}{a-1}	&\cdots	&\binom{n-2}{2}	&n-1	&1\\
 \end{array}$$
\end{thm}

\begin{proof}
For a partition $\la$ of $n$, the unipotent constituents $\rho_\mu$ of the
projective indecomposable module $P_\la$ satisfy $\mu \trianglelefteq \lambda$
(see the proof of Proposition~\ref{prop:smallirr}). Therefore the only
unipotent constituents of $P_{2^b 1^{n-2b}}$ are the unipotent characters
associated with partitions $2^c 1^{n-2c}$ with $0 \leq c \leq b$.
\par
If $c \leq n/3$, the dual partition of $2^c 1^{n-2c}$ is a sum of triangular
partitions, namely $c$ copies of $21$ and $n-3c$ copies of $1$. The
corresponding character $\rho_{(21,\ldots,21,1,\ldots,1)}$ is cuspidal and
by Propositions~\ref{prop:nonunip} and ~\ref{prop:partcases} it remains
irreducible after $\ell$-reduction. More precisely
$(\rho_{(21,\ldots,21,1,\ldots,1)})^0 = \vhi_{2^c 1^{n-2c}}$. This yields
relations on the coefficients of the decomposition matrix, which we shall use
in order to prove the assertion.
\par
The character $\chi_{21}$ of $\fS_3$ has value $2$ on the trivial class,
and $-1$ on the class of $3$-cycles. For $0 \leq k \leq c$, let $x_k$
be any product of $k$ disjoint $3$-cycles in $\fS_{3c}$. Then the value of
the character of $(\fS_3)^c$ corresponding to the
multipartition $(21,21,\ldots,21)$ on $x_k$ is $(-1)^k 2^{c-k}$.
Moreover, for a given $k$, there are $2^k \binom{c}{k}$
elements of $(\fS_3)^c$ in the conjugacy class of $x_k$ in $\fS_{3c}$, so that
by (\ref{eq:lreduction}) above we deduce that
$(\rho_{(21,21,\ldots,21,1,\ldots,1)})^0$ is equal to the
restriction to the set of $\ell'$-elements of the following virtual
unipotent character:
$$\rho_c:=\frac{(-1)^{n(n-1)/2-c}}{3^c} \sum_{k=0}^c (-1)^k
          \binom{c}{k} R_{\bT_{x_k w_0F}}^\bG(1).$$
From the expression of Deligne--Lusztig characters in terms of almost
characters together with Lemma~\ref{lem:almostchar}, we have, for any
partition $\lambda$ of $n$
$$\begin{aligned}
 \langle R_{\bT_{x_k w_0F}}^\bG(1),\rho_{\la} \rangle
  =& (-1)^{a(\la)+A(\la)} \langle R_{\bT_{x_k w_0F}}^\bG(1),
    R_{\widetilde \chi_\la} \rangle\\
  =& (-1)^{a(\la)+A(\la)} \widetilde \chi_\la(x_k w_0 F)
  = (-1)^{A(\la)} \chi_{\la}(x_k).
\end{aligned}$$
With the value $A(2^b 1^{n-2b}) = n(n-1)/2-b$ we deduce that the multiplicity
of $\rho_{2^b 1^{n-2b}}$ in $\rho_c$ is given by
$$\langle \rho_c, \rho_{2^b 1^{n-2b}} \rangle =
  \frac{(-1)^{b+c}}{3^c} \sum_{k=0}^c (-1)^{k} \binom{c}{k} \chi_{2^b1^{n-2b}}(x_k).$$
Let $D$ be the submatrix of the decomposition matrix corresponding to rows and
columns labelled by $2^b1^{n-2b}$ for $a \geq b \geq 0$, and
$C = (\langle \rho_i,\rho_{2^j 1^{n-2j}}\rangle)_{i,j=a,\ldots,0}$.
Since $\rho_i^0 = \vhi_{2^i 1^{n-2i}}$ and since the unipotent characters
form a basic set of characters, we have $C=D^{-1}$. Therefore
Theorem~\ref{thm:corner} is about giving $C^{-1}$ explicitly, which we do
in the next three lemmae.
\end{proof}

\begin{lem}\label{lem:MNrule}
 Let $0 \leq j \leq \lfloor n/2 \rfloor$ and $0 \leq k \leq\lfloor n/3\rfloor$.
 Then the value of $\chi_{2^j1^{n-2j}}$ on a product of $k$ disjoint $3$-cycles
 is given by
  $$\chi_{2^j1^{n-2j}}(x_k) = \sum_{r=0}^{k} \binom{k}{r}
  \Bigg( \binom{n-3k}{j-3r} - \binom{n-3k}{j-3r-1}\Bigg).$$
\end{lem}

\begin{proof}
For any positive integer $n$ and any relative integers $j,k$ we denote by
$M_{n,j,k}$ the right-hand side of the formula. A straightforward computation
using Pascal's rule yields the following relation:
\begin{equation} \label{eq:MNrule}
  M_{n,j,k+1} = M_{n-3,j-3,k}+M_{n-3,j,k}.
\end{equation}
\indent Under the assumptions on $j$ and $k$ given in the Lemma we show by
induction on $k$ that $\chi_{2^j1^{n-2j}}(x_k) = M_{n,j,k}$.
The case $k=0$ corresponds to $M_{n,j,0} = \binom{n}{j} - \binom{n}{j-1} =
\chi_{2^j1^{n-2j}}(1)$ which follows from the hook-length formula. Assume first
that $j \geq 3$ and $n-2j \geq 3$. Then two $3$-hooks (both of height $1$)
can be removed from the partition $2^j1^{n-2j}$ and Murnaghan--Nakayama rule
gives $\chi_{2^j1^{n-2j}}(x_{k+1})
= \chi_{2^{j-3}1^{n-2j+3}}(x_k) + \chi_{2^j1^{n-2j-3}}(x_k)$
and we conclude using (\ref{eq:MNrule}).
\par
Assume $j \in \{0,1,2\}$. Then $M_{n,j,k} = \binom{n-3}{j} - \binom{n-3}{j-1}$
and $M_{n,j-3,k} = 0$. If in addition $n-2j \geq 3$ we have
$\chi_{2^j1^{n-2j}}(x_{k+1}) = \chi_{2^{j}1^{n-2j-3}}(x_k) = M_{n-3,j,k}
  = M_{n-3,j,k} + M_{n-3,j-3,k}$ and again we conclude using formula
(\ref{eq:MNrule}). If $n-2j \in \{0,1,2\}$, then $n \leq 6$ and one checks
easily that the nine partitions to consider all satisfy the formula.
\par
We are left with the case $j\geq 3$ and $n-2j < 3$ which corresponds to the
partitions $2^j$, $2^j1$ and $2^j1^2$. Two hooks (of height $1$ and $2$) can
be removed from the partitions $2^j$, and the Murnaghan--Nakayama rule yields
$\chi_{2^j}(x_{k+1}) = \chi_{2^{j-3}1^3}(x_k) - \chi_{2^{j-2}1}(x_k)
= M_{n-3,j-3,k}-M_{n-3,j-2,k}$. Since $n=2j$, we have
$$\begin{aligned}
  M_{n-3,j-2,k}=&\ \sum_{r=0}^{k} \binom{k}{r} \Bigg( \binom{2j-3-3k}{j-2-3r}
  			- \binom{2j-3k}{j-3-3r}\Bigg) \\
	       =&\ \sum_{r=0}^{k}\binom{k}{r}\Bigg(\binom{2j-3-3k}{j-3(k-r)-1)}
  			- \binom{2j-3k}{j-3(k-r)}\Bigg)\\
	       =&\ -M_{n-3,j,k}
\end{aligned}$$
and we can invoke formula (\ref{eq:MNrule}) to conclude. A similar argument
shows that $M_{n-3,j-1,k} = -M_{n-3,j,k}$ when $n=2j+1$ and $M_{n-3,j,k} = 0$
when $n=2j+2$. We conclude using formula (\ref{eq:MNrule}) and the
Murnaghan--Nakayama rule for $2^j1$ and $2^j1^2$.\end{proof}

\begin{lem} \label{lem:somescalar}
 Let $0 \leq i \leq j \leq a$. Then
 $$(-1)^{i+j} \langle\rho_i,\rho_{2^j 1^{n-2j}}\rangle
   = \binom{n-2i}{j-i} - \binom{n-2i}{j-i-1}.$$
\end{lem}

\begin{proof}
From the previous Lemma, together with the expression for
$\langle\rho_i,\rho_{2^j 1^{n-2j}}\rangle$ given in the end of the proof of
Theorem~\ref{thm:corner} we obtain
$$(-1)^{i+j} \langle\rho_i,\rho_{2^j 1^{n-2j}}\rangle = \frac{1}{3^i}
  \sum_{k=0}^i \sum_{r=0}^k (-1)^{k} \binom{i}{k} \binom{k}{r}
  \Bigg( \binom{n-3k}{j-3r} - \binom{n-3k}{j-3r-1}\Bigg).$$
Let us split this sum into two parts by considering the following integer
$$A_{i,j,n}
  = \sum_{k=0}^i\sum_{r=0}^k (-1)^{k}\binom{i}{k}\binom{k}{r}\binom{n-3k}{j-3r}.
$$
We now show by induction on $i$ that $A_{i,j,n} = 3^i \binom{n-2i}{j-i}$
(without any restriction on $n$ or $j$), which is enough
to conclude. For $i=0$ the relation is straightforward. For $i>0$,
we use Pascal's rule on the binomial coefficients $\binom{i}{k}$
and $\binom{k}{r}$ to derive the following relation
$$ A_{i,j,n} = A_{i-1,j,n}-A_{i-1,j,n-3}-A_{i-1,j-3,n-3}.$$
If we use the induction hypothesis and again Pascal's rule we obtain finally
$$\begin{aligned}
  A_{i,j,n} =&\ 3^{i-1} \Bigg(\binom{n-2i+2}{j-i+1} - \binom{n-2i-1}{j-i+1}
                    - \binom{n-2i-1}{j-i-2}\Bigg)\\
	    =&\ 3^{i} \Bigg(\binom{n-2i-1}{j-i} + \binom{n-2i-1}{j-i-1}\Bigg)
            =\ 3^i \binom{n-2i}{j-i},\\
\end{aligned}$$
which is the expected result.
\end{proof}

\begin{lem}
 Let $C = (c_{i,j})$ be the square matrix of size $a+1$ defined by
 $$c_{i,j} = (-1)^{i+j} \bigg(\binom{n-2i}{j-i} - \binom{n-2i}{j-i-1}\bigg)$$
 for $i \leq j$ and by $c_{i,j} = 0$ otherwise. Then $C$ is invertible and the
 coefficients of $C^{-1} = (d_{i,j})$ are given by
 $$d_{i,j} = \binom{n-i-j}{j-i}$$
 for $i \leq j$ and by $d_{i,j} = 0$ otherwise.
\end{lem}

\begin{proof}
We start by computing the following sum, for $i\leq j$
$$\begin{aligned}
 B_{i,j} = \sum_{k=i}^{j}(-1)^k \binom{n-i-k}{k-i} \binom{n-2k}{j-k}
 =&\ \sum_{k=i}^{j}(-1)^k \frac{(n-i-k)!}{(k-i)!(j-k)!(n-j-k)!} \\
 =&\ \sum_{k=0}^{j-i}(-1)^k \frac{(n-2i-k)!}{k!(j-i-k)!(n-i-j-k)!}\\
 =&\ \sum_{k=0}^{j-i}(-1)^k \binom{j-i}{k} \binom{n-2i-k}{j-i}.
\end{aligned}$$
We claim that this equals $(-1)^{j-i}$. Let $m\geq s \geq 0$ and
$$P_{m,s}(X) = \frac{1}{s!} (m+X)(m+X-1)\cdots (m+X-s+1).$$
This is a polynomial of degree $s$. If we decompose it in the basis of
polynomials $\binom{X}{k}$ for $k=0,\ldots,s$, the coefficient of
$\binom{X}{s}$ is equal to $1$. On the other hand, by the discrete Taylor
formula, it is also equal to
$$\sum_{k=0}^s (-1)^{s-k} \binom{s}{k} P_{m,s}(s-k)
  = \sum_{k=0}^s (-1)^{s-k} \binom{s}{k} \binom{m+s-k}{s}.$$
With $s = j-i$ and $m =n-i-j$ we deduce the previous claim, that is
$B_{i,j} = (-1)^{i-j}$.
\par
Let $D=(d_{i,j})$. The coefficient $(i,j)$ of $DC$ is given by
$B_{i,i} = 1$ if $i=j$, and by $(-1)^{i+j}(B_{i,j}+B_{i,j-1}) = 0$
if $i < j$. On the other hand
since $D$ and $C$ have both triangular shape $(DC)_{i,j} = 0$ whenever $j< i$.
This proves that $D = C^{-1}$.
\end{proof}

A consequence of Theorem \ref{thm:corner} is that the decomposition numbers
of a large family of characters satisfy the analogue of James's row and column
removal rule for $\GL_n(q)$ \cite[Rule~5.8]{Jam90} (see also
\S \ref{sec:observations}):

\begin{cor}    \label{cor:James}
 Assume $\ell>n$ and $\ell | (q+1)$. Let $c \leq b \leq 1+n/3$. Then the
 decomposition number $d_{2^c 1^{n-2c}, 2^b 1^{n-2b}} =
 [\rho_{2^c 1^{n-2c}} : \vhi_{2^b 1^{n-2b}}]$ satisfies James's row removal
 rule:
 $$ d_{2^c 1^{n-2c}, 2^b 1^{n-2b}} = d_{1^{n-2c}, 2^{b-c} 1^{n-2b}}.$$
\end{cor}

\begin{rem}   \label{rem:corner}
 This corollary suggests that Theorem~\ref{thm:corner} should hold
 whenever the partition $2^a 1^{n-2a}$ makes sense, that is when $a \leq n/2$.
 This turns out to be true when $n\le 10$ (see Tables~\ref{tab:2A}--\ref{tab:2A9b}).
\end{rem}

It is generally conjectured that entries of decomposition matrices for a fixed
family of groups of Lie type, like $\SU_n(q)$ ($n$ fixed) are bounded
independently from $q$ and the prime~$\ell$. Theorem~\ref{thm:corner}
shows that nevertheless such a bound will be rather large:

\begin{cor}
 The entries of the $\ell$-modular decomposition matrices of unipotent blocks
 of $\SU_n(q)$, $q\equiv-1\pmod\ell$, are not bounded by any polynomial
 function in $n$.
\end{cor}

Indeed, one entry has the form $\binom{n-a}{a}$, with $a=\lfloor n/3\rfloor+1$.

\subsection{The partition $321^{n-5}$}
As in the previous section, we use the $\ell$-reduction of non-unipotent
characters to compute certain decomposition numbers in the principal
$\ell$-block of $\SU_n(q)$ for general $n$.

\begin{prop}   \label{prop:anotherprojective}
 Assume $\ell | (q+1)$ with either $\ell>n\geq 7$ or $n=6$ and $\ell >7$. Then
 $$\Psi_{321^{n-5}} = \rho_{321^{n-5}}+2\rho_{2^31^{n-6}}+(n-4)\rho_{31^{n-3}}
     + (n-4) \rho_{2^21^{n-4}}+2\rho_{21^{n-2}}+(2n-6)\rho_{1^n}$$
 up to adding non-unipotent characters.
\end{prop}

\begin{proof}
We write the unipotent part of $\Psi_{321^{n-5}}$ as
$$\rho_{321^{n-5}}+y_1\rho_{2^31^{n-6}}+y_2\rho_{31^{n-3}}+ y_3\rho_{2^21^{n-4}}
  +y_4\rho_{21^{n-2}}+y_5\rho_{1^n}.$$
We first assume that $n \geq 9$, so that one can apply
Proposition~\ref{prop:nonunip} to the partitions $2^31^{n-6}$, $2^21^{n-4}$,
$21^{n-2}$ and $1^n$.
\par
The values of the characters $\chi_{321^{n-6}}$ and $\chi_{31^{n-3}}$ of
$\fS_n$ on the classes $3^c 1^{n-3c}$ for $c\leq3$, given in the following
table, can be computed using the Murnaghan--Nakayama rule:
$$\begin{array}{c|cccc}
 \lambda & 3^31^{n-9} & 3^21^{n-6} & 31^{n-3} & 1^n \\\hline
 \vphantom{\Big(}\chi_{321^{n-5}}(\sigma_\lambda) & \frac{(n-9)(n-11)(n-13)}{3}-3 & \frac{(n-6)(n-8)(n-10)}{3}-2 & \frac{(n-3)(n-5)(n-7)}{3}-1 & \frac{n(n-2)(n-4)}{3} \\
 \chi_{31^{n-3}}(\sigma_\lambda) & \binom{n-10}{2} & \binom{n-7}{2} & \binom{n-4}{2} & \binom{n-1}{2} \\
\end{array}$$
Note that one can check that the values are correct even when $n$ is smaller
than $13$.
Now, as in the proof of Theorem~\ref{thm:corner}, we consider the restriction
of the character $\rho_c$ for $c=0,1,2,3$ to the set of $\ell'$-elements and
its coefficient on $\vhi_{321^{n-5}}$, which by Propositions~\ref{prop:nonunip}
and~\ref{prop:partcases} should be zero.
Equivalently, one should have $\langle \rho_c,\Psi_{321^{n-5}}\rangle =0$.
To obtain the necessary relations we need to compute the multiplicity
of $\rho_{321^{n-5}}$ and $\rho_{31^{n-3}}$ in each $\rho_c$ using the
expression of $\rho_c$ in terms of Deligne--Lusztig characters and the values
of the corresponding characters of $\fS_n$. The other scalar products have
already been computed in Lemma~\ref{lem:somescalar}. We obtain
$$\begin{array}{c|cccc}
 c & 0 & 1 & 2 & 3 \\\hline
 \vphantom{\Big(}\langle \rho_{321^{n-5}},\rho_c \rangle & \frac{n(n-2)(n-4)}{3} & -(n-3)(n-4) &  2(n-5) & -2\\[6pt]
 \langle \rho_{2^31^{n-6}},\rho_c \rangle &  \binom{n}{2}-\binom{n}{3} & \binom{n-2}{2}-n+2 & 5-n & 1 \\[6pt]
 \langle \rho_{31^{n-3}},\rho_c \rangle & -\binom{n-1}{2} & n-3 &  -1 & 0\\[6pt]
 \langle \rho_{2^21^{n-4}},\rho_c \rangle & \binom{n}{2}-n & 3-n & 1 & 0\\[6pt]
 \langle \rho_{21^{n-2}},\rho_c \rangle & 1-n & 1 & 0 & 0\\[6pt]
 \langle \rho_{1^{n}},\rho_c \rangle & 1 & 0 & 0 & 0\\
\end{array}$$
and thus the scalar products $\langle \rho_c,\Psi_{321^{n-5}}\rangle =0$ yield
the following relations:
$$ \begin{aligned}
0 =&\ -2+y_1, \\
0 =&\ 2(n-5)+y_1(5-n)-y_2+y_3, \\
0 =&\ -(n-3)(n-4) + y_1\left(\binom{n-2}{2}-n+2\right) +y_2(n-3)+y_3(3-n)+y_4,\\
0 =&\  \frac{n(n-2)(n-4)}{3} +y_1\left(\!\binom{n}{2}\!-\!\binom{n}{3}\!\right)-y_2\binom{n-1}{2}+y_3\left(\!\binom{n}{2}-n\!\right) +y_4(1-n) +y_5,\\
\end{aligned}$$
from which we deduce $y_1=y_4=2$, $y_2=y_3$ and $y_5 = y_2+n-2$.
\par
To obtain the value of $y_2$ we use a further relation that we deduce from
Lemma~\ref{lem:partition31}. Let $I = \{s_2,\ldots,s_{n-2}\}$ and
$w_I = s_{n-1} \cdots s_2s_1s_2 \cdots s_{n-1} w_0$ be the longest element of
the corresponding parabolic subgroup. Then by Lemma~\ref{lem:partition31} we
have $\langle \Psi_{321^{n-5}}, R_{w_I}^0 \rangle=0$. Using the following
values of characters of $\fS_n$ computed with the help of the
Murnaghan--Nakayama rule
$$\begin{array}{c|cccccc}
  \la & 321^{n-5} & 2^31^{n-6} & 31^{n-3} & 2^21^{n-4} & 21^{n-2}\\\hline
  -\chi_{\la}(w_Iw_0) & \frac{(n-2)(n-4)(n-6)}{3} & \frac{(n-2)(n-3)(n-7)}{6}+n-3
  & \binom{n-3}{2}-1 & \frac{(n-2)(n-5)}{2}+1 & n-3\\
\end{array}$$
we get
$$\begin{aligned}
  0 = \langle \Psi_{321^{n-5}}, R_{w_I}^0 \rangle =&\ \chi_{321^{n-5}}(w_Iw_0)-2 \chi_{2^31^{n-6}}(w_Iw_0) -
	    y_2 \chi_{31^{n-3}}(w_Iw_0) \\
	 &\ + y_2\chi_{2^21^{n-4}}(w_Iw_0) - 2 \chi_{21^{n-2}}(w_Iw_0) + (y_2+n-2) \chi_{1^n}(w_Iw_0)\\
	=&\ 2n-8-2y_2 \\
\end{aligned}$$
so that $y_2 = n-4$.
\par
When $n =6,7,8$ we cannot invoke Proposition~\ref{prop:nonunip} with the
partition $2^31^{n-6}$ since it is no longer a concatenation of triangular
partitions. Therefore the relation $0=-2+y_1$ can no longer be deduced via
$\ell$-reduction. Let us consider the projective character $\Psi_{321^5}$ of
$\SU_{10}(q)$. (Here one needs to assume $\ell > 7$ so that one has indeed
$\ell > 10$). The multiplicity of $\rho_{2^31^2}$ (resp.~$\rho_{2^3}$) in the
Harish-Chandra restriction of $\Psi_{321^5}$ to a split Levi subgroup of
semisimple type $\tw2A_7$ (resp.~$\tw2A_5$) is $2$. Therefore $y_1$ can be at
most 2 when $n=6,8$. Now, the restriction $(\chi_\bmu)^0$ for the
multipartition $\bmu= (2,2,2,1,1)$ (resp.~$\bmu = (2,2,2)$) forces $y_1$ to
be at least $2$, therefore it must be equal to $2$. The same argument works
when $n=7$, starting with the projective character $\Psi_{321^4}$ of $\SU_9(q)$.
\end{proof}

\begin{rem}
 Proposition~\ref{prop:anotherprojective} is another example where the analogue
 of James's column removal rule holds, since in this case
 $d_{2^31^{n-6},321^{n-5}} = d_{1^3,21} =2$.
\end{rem}

%%%%%%%%%%%%%%%%%%%%%%%%%%%%%%%%%%%%%%%%%%%%%%%%%%%%%%%%%%%%%%%%%%%%%%%%%
\section{Tables for small rank\label{sec:smallrank}}
In this section, we give the decomposition matrices (respectively
approximations with few unknown entries) of unipotent blocks of $\SU_n(q)$,
$n\le10$, for unitary primes $\ell>n$ (see the introduction to
Section~\ref{sec:general} for why we exclude smaller primes).
In addition, we determine the $\ell$-modular Harish-Chandra series for the
unipotent Brauer characters in all considered
cases. They are known when $\ell|(q+1)$ by \cite[Thm.~4.12]{GHM2} (again
when $\ell>n$), but not for the other unitary primes. Thus our results may
give hints on properties of this distribution into Harish-Chandra series.

The Brauer trees in the cases when the Sylow $\ell$-subgroups are cyclic were
determined by Fong and Srinivasan \cite{FS2}. In the cases considered here,
we recover their results by our methods.

Recall from \S\ref{sec:DL} that for $w\in W$ we denote by $R_w$
(resp.~$R_w[\lambda]$) the virtual character afforded by the cohomology of
the Deligne--Lusztig variety  $\rX(w)$ (resp.~the generalized
$\lambda$-eigenspace of $F^\delta$ on that cohomology). This is computed using
formula~(i) of Theorem~\ref{thm:dlvar} evaluated by the \Chevie-programs
written by Jean Michel \cite{MChv}. Up to adding non-unipotent characters,
$R_w$ is a virtual projective character and Proposition~\ref{prop:constitPw}
gives some control on how it can be decomposed on the basis of projective
indecomposable modules.

\smallskip

\noindent \emph{Convention.} To simplify the proofs, we shall write equalities
between characters whenever they hold up to adding/removing non-unipotent
characters. Since the unipotent characters form a basic set, this will
be sufficient for our purposes.

\subsection{The case $\ell|(q+1)$}

\begin{thm}   \label{thm:2A}
 Let $\ell$ be a prime. Then all unipotent characters of $\SU_n(q)$, where
 $q\equiv-1\pmod\ell$, lie in the principal block. Its $\ell$-modular
 decomposition matrix for
 \begin{itemize}
   \item[--] $2 \leq n \leq 5$ and $\ell > n$,
   \item[--] $n = 6,7$ and $\ell > 7$,
   \item[--] $n = 8$ and $\ell > 11$, and
   \item[--] $n = 9$ and $\ell > 13$
 \end{itemize}
 is as given in Tables~\ref{tab:2A}--\ref{tab:2A8}.
 \par
 The corresponding modular Harish-Chandra series are given in the bottom rows
 of the tables. For $n\le7$ we also print the degrees of the unipotent
 characters.
\end{thm}

As customary, the principal series is indicated by 'ps', cuspidal characters
are denoted by 'c', and for all other series the Harish-Chandra source is given.

\begin{table}
\caption{Unipotent blocks of unitary groups $\SU_n(q)$ for $\ell|(q+1)$,
  $\ell>n$}   \label{tab:2A}
\vskip 1pc
\hbox{
\vbox{\offinterlineskip\halign{$#$\hfil\ \vrule height11pt depth4pt&&
      \hfil\ $#$\hfil\cr
\noalign{\hbox{\noindent \ \ \ \ $\SU_2(q)$}}
\omit& \omit& \vrule height11pt depth4pt width0pt\cr
\omit& \vrule height11pt depth4pt width0pt\cr
\omit& \vrule height11pt depth4pt width0pt\cr
\omit& \vrule height11pt depth4pt width0pt\cr
  2& 1&  1\cr
1^2& q&  1& 1\cr
\noalign{\hrule}
 \omit& & ps& c& \vrule height11pt depth4pt width0pt\cr
   }}\hskip 2.3cm
\vbox{\offinterlineskip\halign{$#$\hfil\ \vrule height11pt depth4pt&&
      \hfil\ $#$\hfil\cr
\noalign{\hbox{\noindent \ \ \ \ $\SU_3(q)$}}
\omit& \vrule height11pt depth4pt width0pt\cr
\omit& \vrule height11pt depth4pt width0pt\cr
\omit& \vrule height11pt depth4pt width0pt\cr
  3&     1&  1\cr
 21& q\Ph1&  .& 1\cr
1^3&   q^3&  1& 2& 1\cr
\noalign{\hrule}
 \omit& & ps& c& c& \vrule height11pt depth4pt width0pt\cr
   }}\hskip 2.3cm
\vbox{\offinterlineskip\halign{$#$\hfil\ \vrule height11pt depth4pt&&
      \hfil\ $#$\hfil\cr
\noalign{\hbox{\noindent \ \ \ \ $\SU_4(q)$}}
\omit& \vrule height11pt depth4pt width0pt\cr
   4&       1& 1\cr
  31&   q\Ph6& 1& 1\cr
 2^2& q^2\Ph4& 1& 1& 1\cr
21^2& q^3\Ph6& .& .& 1& 1\cr
 1^4&     q^6& .& 1& 1& 3& 1\cr
\noalign{\hrule}
 \omit& & ps& 1^2& ps& c& c\cr
   }}  }
\vskip 2pc
\hbox{
\vbox{\offinterlineskip\halign{$#$\hfil\ \vrule height11pt depth4pt&&
      \hfil\ $#$\hfil\cr
\noalign{\hbox{\noindent \ \ \ \ $\SU_5(q)$}}
\omit& \vrule height11pt depth4pt width0pt\cr
\omit& \vrule height11pt depth4pt width0pt\cr
\omit& \vrule height11pt depth4pt width0pt\cr
\omit& \vrule height11pt depth4pt width0pt\cr
\omit& \vrule height11pt depth4pt width0pt\cr
   5&           1& 1\cr
  41&   q\Ph1\Ph4& .& 1\cr
  32&  q^2\Ph{10}& .& .& 1\cr
31^2& q^3\Ph4\Ph6& 1& 2& 1& 1\cr
2^21&  q^4\Ph{10}& 1& 2& .& 1& 1\cr
21^3& q^6\Ph1\Ph4& .& 1& .& .& 2& 1\cr
 1^5&      q^{10}& .& 2& 1& 1& 3& 4& 1\cr
\noalign{\hrule}
 \omit& & ps& 21& ps& 1^3& c& c& c& \vrule height11pt depth4pt width0pt\cr
   }}\hskip 1.2cm
\vbox{\offinterlineskip\halign{$#$\hfil\ \vrule height11pt depth4pt&&
      \hfil\ $#$\hfil\cr
\noalign{\hbox{\noindent \ \ \ \ $\SU_6(q)$}}
\omit& \vrule height11pt depth4pt width0pt\cr
     6&                 1& 1\cr
    51&          q\Ph{10}& 1& 1\cr
    42&     q^2\Ph3\Ph6^2& 1& 1& 1\cr
  41^2&    q^3\Ph4\Ph{10}& .& .& 1& 1\cr
   3^2&    q^3\Ph3\Ph{10}& 1& 1& 1& .& 1\cr
   321& q^4\Ph1^3\Ph3\Ph4& .& .& .& .& .& 1\cr
   2^3&    q^6\Ph3\Ph{10}& .& .& 1& 1& 1& 2& 1\cr
  31^3&    q^6\Ph4\Ph{10}& .& 1& 1& 3& 1& 2& .& 1\cr
2^21^2&     q^7\Ph3\Ph6^2& .& 1& 1& 3& 1& 2& 1& 1& 1\cr
  21^4&     q^{10}\Ph{10}& .& .& .& 1& .& 2& 1& .& 3& 1\cr
   1^6&            q^{15}& .& .& .& 3& 1& 6& 1& 1& 6& 5& 1\cr
\noalign{\hrule}
 \omit& & ps& 1^2& ps& 21^2& 1^2& c& ps& 1^4& c& c& c& \vrule height11pt depth4pt width0pt\cr
   }} }
\end{table}

\begin{table}[ht]
\caption{$\SU_7(q)$, $\ell| (q+1)$, $\ell>7$}   \label{tab:2A6}
$$\vbox{\offinterlineskip\halign{$#$\hfil\ \vrule height11pt depth4pt&&
      \hfil\ $#$\hfil\cr
    7&                      1& 1\cr
   61&          q\Ph1\Ph3\Ph6& .& 1\cr
   52&         q^2\Ph4\Ph{14}& .& .& 1\cr
 51^2&     q^3\Ph3\Ph6\Ph{10}& 1& 2& 1& 1\cr
   43&     q^3\Ph1\Ph3\Ph{14}& .& .& .& .& 1\cr
  421&      q^4\Ph{10}\Ph{14}& 1& 2& .& 1& .& 1\cr
 3^21&     q^5\Ph3\Ph6\Ph{14}& 1& 2& 1& 1& 2& 1& 1\cr
 41^3& q^6\Ph1\Ph3\Ph4\Ph{10}& .& 1& .& .& 1& 2& .& 1\cr
 32^2&     q^6\Ph3\Ph6\Ph{14}& .& .& 1& .& 2& .& 1& .& 1\cr
321^2&      q^7\Ph{10}\Ph{14}& .& .& .& .& .& .& .& .& 1& 1\cr
 2^31&     q^9\Ph1\Ph3\Ph{14}& .& 1& .& .& .& 2& .& 1& .& 2& 1\cr
 31^4&  q^{10}\Ph3\Ph6\Ph{10}& .& 2& 1& 1& 2& 3& 1& 4& 1& 3& .& 1\cr
2^21^3&     q^{11}\Ph4\Ph{14}& .& 2& 1& 1& 2& 4& 1& 4& .& 3& 2& 1& 1\cr
 21^5&     q^{15}\Ph1\Ph3\Ph6& .& .& .& .& 1& 2& .& 1& .& 2& 3& .&4& 1\cr
  1^7&                 q^{21}& .& .& .& .& 2& 3& 1& 4& 1& 8& 4& 1&10&6& 1\cr
\noalign{\hrule}
 \omit& & ps& 21& ps& 1^3& 21& 2^21& 1^3& 21^3& ps& c& c& 1^5& c& c& c\cr
   }}$$
\end{table}

\begin{table}[ht]
{\small\caption{$\SU_8(q)$, $\ell| (q+1)$, $\ell>11$}   \label{tab:2A7}
$$\vbox{\offinterlineskip\halign{$#$\hfil\ \vrule height11pt depth3pt&&
      \hfil\ $#$\hfil\cr
     8& 1\cr
    71& 1& 1\cr
    62& 1& 1& 1\cr
  61^2& .& .& 1& 1\cr
    53& 1& 1& 1& .& 1\cr
   4^2& 1& 1& 1& .& 1& 1\cr
   521& .& .& .& .& .& .& 1\cr
   431& .& .& .& .& .& 1& .& 1\cr
  51^3& .& 1& 1& 3& 1& .& 2& .& 1\cr
  42^2& .& .& 1& 1& 1& 1& 2& 1& .& 1\cr
 421^2& .& 1& 1& 3& 1& .& 2& .& 1& 1& 1\cr
  3^22& .& .& 1& 1& 1& .& 2& .& .& 1& .& 1\cr
3^21^2& .& 1& 1& 3& 2& 1& 4& 3& 1& 1& 1& 1& 1\cr
  41^4& .& .& .& 1& .& .& 2& 1& .& 1& 3& .& .& 1\cr
 32^21& .& .& .& .& 1& 1& 2& 3& .& .& .& .& 1& .& 1\cr
 321^3& .& .& .& .& .& .& 1& .& .& .& .& .& .& .& 2& 1\cr
   2^4& .& .& .& .& 1& 1& 4& 4& .& 1& 1& 1& 1& .& 3& .& 1\cr
2^31^2& .& .& .& 1& .& .& 4& 1& .& 1& 3& 1& .& 1& 3& 2& 1& 1\cr
  31^5& .& .& .& 3& 1& .& 8& 3& 1& 1& 6& 1& 1& 5& 3& 4& .& .& 1\cr
2^21^4& .& .& .& 3& 1& .& 8& 3& 1& 1& 7& 1& 1& 5& 3& 4& 1& 3& 1& 1\cr
  21^6& .& .& .& .& .& .& 2& 1& .& .& 3& .& .& 1& 3& 2& 1& 6& .& 5& 1\cr
   1^8& .& .& .& .& .& .& 6& 3& .& .& 6& 1& 1& 5& 8&10& 1&10& 1&15& 7& 1\cr
\noalign{\hrule}
 \omit& ps& 1^2& ps& 21^2& 1^2& ps& 321& 21^2& 1^4& ps& 2^21^2& 1^2& 1^4& 21^4& c& c& ps& c& 1^6& c& c& c\cr
   }}$$}
\end{table}

\begin{table}[ht]
{\small\caption{$\SU_9(q)$, $\ell| (q+1)$, $\ell>13$}   \label{tab:2A8}
$$\vbox{\offinterlineskip\halign{$#$\hfil\ \vrule height10pt depth2pt&&
      \hfil\ $#$\hfil\cr
     9&   & 1\cr
    81&   & .& 1\cr
    72&   & .& .& 1\cr
  71^2&   & 1& 2& 1& 1\cr
    63&   & .& .& .& .& 1\cr
   621&   & 1& 2& .& 1& .& 1\cr
    54&   & .& .& .& .& .& .& 1\cr
   531&   & 1& 2& 1& 1& 2& 1& 1& 1\cr
  61^3&   & .& 1& .& .& 1& 2& .& .& 1\cr
  52^2&   & .& .& 1& .& 2& .& 1& 1& .& 1\cr
  4^21&   & 1& 2& 1& 1& 2& 1& .& 1& .& .& 1\cr
 521^2&   & .& .& .& .& .& .& .& .& .& 1& .& 1\cr
   432&   & .& .& .& .& .& .& .& .& .& .& .& .& 1\cr
 431^2&   &.& .& .& .& 1& .& .& .& .& .& 2& .& 1& 1\cr
 42^21&   &.& 1& .& .& 1& 2& .& .& 1& .& 2&  2& .& 1& 1\cr
   3^3&   & .& .& 1& .& 2& .& .& 1& .& 1& 1& .& 2& .& .& 1\cr
  51^4&   & .& 2& 1& 1& 2& 3& .& 1& 4& 1& .& 3& .& .& .& .& 1\cr
 3^221&   &.& 1& .& .& .& 2& .& .& 1& .& .&  2& .& .& 1& .& .& 1\cr
 421^3&   & .& 2& 1& 1& 2& 4& .& 1& 4& .& 1& 3& .& .& 2& .& 1& .& 1\cr
  32^3&   & .& .& .& .& .& .& .& .& .& 1& .& 1& 2& .& .& 1& .& 2& .& 1\cr
3^21^3&   & .& 2& 1& 1& 4& 4& 1& 2& 4& 1& 4& 6& 2& 4& 2& 1& 1& 2& 1& .& 1\cr
  41^5&   & .& .& .& .& 1& 2& .& .& 1& .& 2& 2& 1& 1& 3& .& .& .& 4& .& .& 1\cr
32^21^2&  & .& .& .& .& 2& .& 1& 1& .& 1& 3& 3& 2& 4& .& 1& .& 2& .& 1& 1& .& 1\cr
 321^4&   & .& .& .& .& .& .& .& .& .& .& .& 1& .& .& .& .& .& 2& .& 1& .& .& 3     & 1\cr
  2^41&   & .& .& .& .& 2& 1& 1& 1& .& .& 3& 3& .& 4& 2& .& .& 2& 1& .& 1& .& 3& .& 1\cr
2^31^3&   & .& .& .& .& 1& 2& .& .& 1& .& 2& 4& .& 1& 4& .& .& 4& 4& .& .& 1& 6& 2     & 2& 1\cr
  31^6&   & .& .& .& .& 2& 3& .& 1& 4& 1& 3&11& 2& 4& 4& 1& 1& 6&10& 1& 1& 6& 6& 5& .& .& 1\cr
2^21^5&   & .& .& .& .& 2& 3& .& 1& 4& 1& 4&11& 2& 4& 6& 1& 1& 6&11& .& 1& 6& 9& 5& 3& 4& 1& 1\cr
  21^7&   & .& .& .& .& .& .& .& .& .& .& 2& 2& 1& 1& 3& .& .& 3& 4& .& .& 1& 12& 2     & 4&10& .& 6& 1\cr
   1^9&   & .& .& .& .& .& .& .& .& .& .& 3& 8& 2& 4& 4& 1& .& 12&10& 1& 1& 6& 30& 12& 5&20& 1& 21& 8& 1\cr
\noalign{\hrule}
 \omit& & p& 21& p& 1^3& 21& 2^21& p& 1^3& 21^3& p& 2^21& 321^2& 21& 21^3& 2^31& 1^3& 1^5& c& 2^21^3& p& 1^5& 21^5& c& c& c& c& 1^7& c& c& c\cr
   }}$$}
\end{table}

\begin{proof}
It is well-known that all unipotent characters of $\SU_n(q)$ lie in the
principal $\ell$-block, see e.g.~\cite{BMM}.
The decomposition matrix for $\SU_2(q)\cong\SL_2(q)$ is well-known, and the
last unknown entries in the one for $\SU_3(q)$ were determined by Okuyama and
Waki \cite{OW02}. It is also an immediate consequence of
Theorem~\ref{thm:corner}. It was shown by Geck \cite{Ge91} that the
decomposition matrix of the unipotent block of $\SU_n(q)$ is unitriangular
and that the unipotent characters form a basic set. The distribution of
modular characters into Harish-Chandra series was proved in
\cite[Thm.~4.12]{GHM2}.
\par
Harish-Chandra induction of PIMs from proper Levi subgroups yields most of
the columns in the tables. Next, Theorem~\ref{thm:corner} and
Proposition~\ref{prop:anotherprojective} yield the PIMs in
the lower right-hand corner of the decomposition matrices. For $\SU_4(q)$ and
for $\SU_5(q)$ all other listed projectives are indecomposable, since any
direct summand would have to lie in the same or a smaller Harish-Chandra
series, and clearly there is no such decomposition possible. This completes
the proof for $\SU_4(q)$ and for $\SU_5(q)$.
\par
For $\SU_6(q)$, after computing all Harish-Chandra induced projectives, we are
left to show that the fourth column is indecomposable.
Using the $\ell$-reduction of the non-unipotent irreducible character
associated with the multipartition $\bmu = (2,2,2)$ one finds
that the $(7,4)$-entry of the decomposition matrix is at least $1$ (so
that the fourth column is indecomposable).
\par
For $\SU_7(q)$ a combination of Theorem~\ref{thm:corner} and
Proposition~\ref{prop:anotherprojective} gives the PIMs corresponding
to cuspidal modules.
\par
For $\SU_8(q)$ the previous arguments leave to determine the missing entries
--- denoted by $c_1,\ldots,c_7$ --- for the column corresponding to the
cuspidal Brauer character $32^21$. Relations on the $c_i$'s are first obtained using
Propositions~\ref{prop:nonunip} and \ref{prop:partcases}: the restriction to
$\ell'$-elements of the non-unipotent characters corresponding to the
multipartitions $(321,1,1)$, $(21,21,1,1)$, $(21,1,1,1,1,1)$ and
$(1,1,1,1,1,1,1,1)$ are the irreducible Brauer characters $\vhi_{321^3}$,
$\vhi_{2^21^4}$,  $\vhi_{21^6}$ and $\vhi_{1^8}$ of $\SU_8(q)$. From the fact
that the coefficient of each of these on $\vhi_{32^21}$ is zero, one
gets four relations on the $c_i$'s, from which we deduce $c_1 = 2$,
$c_5 = -2c_2+3c_3+c_4-3$, $c_6 = -5c_2+6c_3$ and $c_7 = 2-9c_2+10c_3+c_4$.
Using Lemma~\ref{lem:partition31} we obtain another relation given
by $10-16c_1-4c_2+10c_3+9c_4-10c_5+5c_6-c_7 =0$, which together
with the previous relations yields $c_4 = 3$. For the lower bounds,
one uses the multipartitions $(2,2,2,2)$ and $(21,21,2)$
to get respectively $c_2 \ge3$ and $c_3 \ge c_2$. To deduce the missing
entries we compute the column of $\SU_{12}(q)$ corresponding to the
partition $32^21^5$ and then look at its Harish-Chandra restriction, assuming
that $\ell > 12$ (or equivalently $\ell > 11$).
Apart from $31^9$, all the partitions of 12 that are smaller than $32^21^5$
satisfy the assumptions of Proposition~\ref{prop:nonunip}, and therefore
we obtain almost
as many relations as unknown entries for the corresponding column. The
missing relation is obtained from Lemma~\ref{lem:partition31}. This yields
$$\Psi_{32^21^5} = \rho_{32^21^5} + 6\rho_{321^7}+21\rho_{31^9}+3\rho_{2^41^4}
+12 \rho_{2^31^6}+  24 \rho_{2^21^8}+21\rho_{21^{10}}+84\rho_{1^{12}}.$$
Since $c_1 =2$, we deduce
that the restriction of this character to $\SU_{8}(q)$ has necessarily
$\Psi_{32^21}+4\Psi_{321^3}$ as a direct summand, giving upper bounds
on the multiplicity of some unipotent constituents in $\Psi_{32^21}$,
namely $c_2 \leq 3$ and $c_3 \leq 4$. Together with the previous lower
bounds, we deduce that $c_2 =3$ and $c_4 \in\{3,4\}$.
The relations on the $c_i$'s become $c_5 = 3c_3-6$, $c_6 = 6c_3-15$
and $c_7 = 10c_3-22$. Finally, to determine $c_3$ we decompose the
Deligne--Lusztig character $R_w$ for $w=s_1 s_2 s_3 s_4$ a Coxeter element.
We find
$$ R_w = \Psi_{8} - \Psi_{71} -\Psi_{62}+\Psi_{61^2} +\Psi_{53}-\Psi_{521}
  +\Psi_{42^2}+\Psi_{32^21}-\Psi_{2^4}+(3-c_3)\Psi_{2^31^2}$$
so that by Proposition~\ref{prop:constitPw} we must have $c_3 \leq 3$,
and therefore $c_3 = 3$.
\par
For $\SU_9(q)$, let us denote by $d_1,\ldots,d_{17}$ the remaining missing
entries for the two columns corresponding to the cuspidal Brauer characters
$3^221$ and $32^21^2$ of $\SU_9(q)$. (Note that the multiplicity of
$\rho_{421^3}$ and $\rho_{41^5}$ in $\Psi_{32^21^2}$ is zero.) As for
$\SU_8(q)$, the $\ell$-reduction of suitably chosen
non-unipotent characters gives relations on the $d_i$'s, from which
we deduce $d_3= d_1+d_2-2$, $d_4 = d_1$, $d_6 = 4-2d_2+2d_5$,
$d_7 = d_1+d_2+2$, $d_8 = 4-2d_2+3d_5$, $d_9 = 3-4d_2+4d_5$,
$d_{10} = 8+d_1-4d_2+5d_5$, $d_{11} = 3$, $d_{13}= 2d_{12}$,
$d_{14} = 6$, $d_{15} = 3d_{12}$, $d_{16} = 4d_{12}$  and
$d_{17} = 5d_{12} + 15$. The lower bounds are obtained by looking
at the $\ell$-reduction of non-unipotent characters $\rho_{\bmu}$
for the following multipartitions:
$$\begin{array}{c|c}
  \text{multipartition } \bmu & \text{relations} \\\hline
  (2,2,1,1,1,1,1)	& d_2\ge2\\
  (2,2,2,2,1)		& d_{12}\ge3 \\
  (3,2,2,2)		& d_1\ge2\\
  (21,2,2,2)		& d_5\ge d_2\\
\end{array}$$
To get the value of $d_{12}$ one computes the column of $\SU_{13}(q)$
corresponding to the partition $32^21^6$ (assuming that $\ell > 13$),
and then looks at its Harish-Chandra
restriction, as we did for $\SU_8(q)$. It is given by
$$\Psi_{32^21^6} = \rho_{32^21^6}+7\rho_{321^8}+28\rho_{31^{10}}+3\rho_{2^41^5}
  +14 \rho_{2^31^7}+ 31 \rho_{2^21^9}+24\rho_{21^{11}}+108\rho_{1^{13}}$$
and its restriction to $\SU_{9}(q)$ involves $\Psi_{32^21^2}+4\Psi_{321^4}$.
This forces $d_{12} \leq 3$ and therefore $d_{12} = 3$ by the previous
inequality. From the previous relations we deduce
$d_{13} = 6$, $d_{15} = 9$, $d_{16}=12$ and $d_{17} = 30$.
\par
To compute $d_2$ we compute the coefficient of $P_{3^21^3}$ in the
Deligne--Lusztig character $R_w$ for $w = s_1s_2s_3s_5s_4s_5$. It is given by
$\langle \widetilde R_w,\vhi_{3^21^3} \rangle = 4-2d_2$. Now one checks that
for all $v < w$, the PIM $P_{3^21^3}$ does not occur in $R_v$, so that by
Proposition~\ref{prop:constitPw} we must have $d_2 \leq 2$, and therefore
$d_2 = 2$ by the previous inequalities. The other relations on the $d_i$'s
become $d_3= d_1$, $d_6 = 2d_5$, $d_7 = d_1+4$, $d_8 = 3d_5$, $d_9 = 4d_5-5$
and $d_{10} = d_1+5d_5$.
\par
Finally, to compute $d_1$ and $d_5$ we consider the PIM of $\SU_{15}(q)$
corresponding to the partition $3^221^7$. It has at most $12$ unipotent
constituents, including $\rho_{3^21^{9}}$, $\rho_{32^31^6}$ and
$\rho_{2^41^7}$. As usual, we obtain relations on the multiplicities of these
characters using Propositions~\ref{prop:nonunip} and \ref{prop:partcases}
assuming that $\ell > 15$ (or equivalently $\ell > 13$).
With the restriction to $\ell'$-elements of the non-unipotent characters
associated with the multipartitions
$$\begin{aligned}
  &(321,21,21,1,1,1),\ (321,21,1,1,1,1,1,1),\ (321,1,1,1,1,1,1,1,1,1),\\
  &(21,21,21,21,21),\ (21,21,21,21,1,1,1),
\end{aligned}$$
we can deduce that $\langle \Psi_{3^221^7}, \rho_{32^31^6} \rangle =2$ and
$\langle\Psi_{3^221^7},\rho_{3^21^9}\rangle
=\langle\Psi_{3^221^7},\rho_{2^41^7} \rangle =:d$.
Let $\Psi$ be the restriction of $\Psi_{3^221^7}$ to a standard Levi subgroup
of semisimple type $\tw2A_8$. We have $\langle \Psi , \rho_{32^3} \rangle = 2$
and $\langle\Psi,\rho_{3^21^3}\rangle = \langle\Psi,\rho_{2^41}\rangle = d$.
Moreover, $\la = 3^221$ is the highest partition such that $\rho_\la$ is a
constituent of $\Psi$. From the shape of the decomposition matrix and the
value of $d_2$ we deduce that $\Psi_{3^221}+(d-2)\Psi_{3^21^3}$ is a direct
summand of $\Psi$. This forces $d_1$ (resp.~$d_5$) to be bounded by $2$
(resp.~by $d-(d-2) = 2$). With the previous inequalities this forces
$d_1 = d_5 = 2$ which finishes the determination of all the decomposition
numbers of the principal block of $\SU_9(q)$.
\end{proof}

The situation is more complicated for $\SU_{10}(q)$ but we can use our methods
to compute all columns of the decomposition matrix but the one for the ordinary
cuspidal character. For this specific column we are left with 4 possibilities.

\begin{thm}   \label{thm:2A9,d=2}
 Assume $\ell>17$ and $\ell|(q+1)$. Then the $\ell$-modular decomposition matrix
 of the principal $\ell$-block of $\SU_{10}(q)$ is given in Tables~\ref{tab:2A9}
 and~\ref{tab:2A9b}, where $\alpha, \beta \in \{0,1\}$. \par
 In the tables we have also given the labelling of unipotent characters by
 bipartitions of~$10$.
\end{thm}

\begin{table}[ht]
{\small\caption{Decomposition matrix for $\SU_{10}(q)$, $\ell| (q+1)$, $\ell>17$}   \label{tab:2A9}
$$\vbox{\offinterlineskip\halign{$#$\hfil\ \vrule height9pt depth2pt&&
      \hfil\ $#$\hfil\cr
     10& 5.     &1\cr
     91& .5     & 1& 1\cr
     82& 4.1    & 1& 1& 1\cr
   81^2& 41.    & .& .& 1& 1\cr
     73& 1.4    & 1& 1& 1& .& 1\cr
  721&\tw2A_5:2.& .& .& .& .& .& 1\cr
   71^3& .41    & .& 1& 1& 3& 1& 2& 1\cr
     64& 3.2    & 1& 1& 1& .& 1& .& .& 1\cr
    631& 32.    & .& .& .& .& .& .& .& 1& 1\cr
   62^2& 31.1   & .& .& 1& 1& 1& 2& .& 1& 1& 1\cr
  621^2& 3.1^2  & .& 1& 1& 3& 1& 2& 1& .& .& 1& 1\cr
   61^4& 31^2.  & .& .& .& 1& .& 2& .& .& 1& 1& 3& 1\cr
    5^2& 2.3    & 1& 1& 1& .& 1& .& .& 1& .& .& .& .& 1\cr
541&\tw2A_5:1^2.& .& .& .& .& .& .& .& .& .& .& .& .& .& 1\cr
    532& 1^2.3  & .& .& 1& 1& 1& 2& .& .& .& 1& .& .& .& .& 1\cr
  531^2& 1.31   & .& 1& 1& 3& 2& 4& 1& 1& 3& 1& 1& .& 1& 2& 1& 1\cr
  52^21& .32    & .& .& .& .& 1& 2& .& 1& 3& .& .& .& 1& 2& .& 1& 1\cr
521^3&\tw2A_5:1.1\!&.& .& .& .& .& 1& .& .& .& .& .& .& .& 1& .& .& 2& 1\cr
   51^5& .31^2  & .& .& .& 3& 1& 8& 1& .& 3& 1& 6& 5& .& 2& 1& 1& 3& 4& 1\cr
   4^22& 21.2   & .& .& 1& 1& 1& 2& .& 1& 1& 1& .& .& 1& 2& 1& .& .& .& .& 1\cr
 4^21^2& 2.21   & .& 1& 1& 3& 2& 4& 1& 1& 3& 1& 1& .& 1& 2& 1& 1& .& .& .& 1& 1\cr
   43^2& 2^2.1  & .& .& .& .& .& .& .& 1& 1& .& .& .& 1& 2& .& .& .& .& .& 1& .& 1\cr
   4321& \tw2A_9& .& .& .& .& .& .& .& .& .& .& .& .& .& .& .& .& .& .& .& .& .& .\cr
  431^3& 2^21.  & .& .& .& .& .& .& .& .& 1& .& .& .& .& 2& .& .& .& .& .& 1& 3& 1\cr
   42^3& 21.1^2 & .& .& .& .& 1& 4& .& 1& 4& 1& 1& .& 1& 4& 1& 1& 3& .& .& 1& 1& 1\cr
42^21^2& 21^2.1 & .& .& .& 1& .& 4& .& .& 2& 1& 3& 1& .& 4& 1& .& 3& 2& .& 1& 3& 1\cr
  421^4& 2.1^3  & .& .& .& 3& 1& 8& 1& .& 3& 1& 7& 5& .& 2& 1& 1& 3& 4& 1& .& 1& .\cr
   41^6& 21^3.  & .& .& .& .& .& 2& .& .& 1& .& 3& 1& .& 2& .& .& 3& 2& .& .& 3& 1\cr
   3^31& 1.2^2  & .& .& .& .& 1& 2& .& 1& 3& .& .& .& 1& 4& 1& 1& 1& .& .& 1& 1& 3\cr
 3^22^2& 1^2.21 & .& .& .& .& 1& 4& .& 1& 4& 1& 1& .& 1& 6& 2& 1& 3& .& .& 1& 1& 3\cr
3^221^2& 1^3.2  & .& .& .& 1& .& 4& .& .& 1& 1& 3& 1& .& 2& 1& .& 3& 2& .& .& .& .\cr
 3^21^4& 1.21^2 & .& .& .& 3& 1&10& 1& .& 6& 1& 7& 5& 1&10& 2& 2& 6& 8& 1& 1& 7& 3\cr
32^31&\tw2A_5:.2& .& .& .& .& .& 1& .& .& .& .& .& .& .& .& .& .& 2& 1& .& .& .& .\cr
32^21^3& .2^21  & .& .& .& .& .& 2& .& .& 3& .& .& .& 1& 8& 1& 1& 4& 4& .& 1& 6& 3\cr
321^5&\tw2A_5:.1^2& .& .& .& .& .& .& .& .& .& .& .& .& .& 1& .& .& 2& 1& .& .& .& .\cr
   31^7& .21^3  & .& .& .& .& .& 6& .& .& 3& .& 6& 5& .& 8& 1& 1&11&14& 1& .& 6& 3\cr
    2^5& 1^3.1^2& .& .& .& .& .& 2& .& .& 1& .& .& .& .& 6& 1& .&6& 2& .& 1& 4& 4\cr
 2^41^2& 1^2.1^3& .& .& .& .& .& 2& .& .& 3& .& 1& .& 1&10& 1& 1& 6& 4& .& 1& 7& 4\cr
 2^31^4& 1^4.1  & .& .& .& .& .& 2& .& .& 1& .& 3& 1& .& 4& .& .& 6& 4& .& .& 3& 1\cr
 2^21^6& 1.1^4  & .& .& .& .& .& 6& .& .& 3& .& 6& 5& .& 8& 1& 1&11& 14& 1& .& 7& 3\cr
   21^8& 1^5.   & .& .& .& .& .& .& .& .& .& .& .& .& .& 2& .& .& 3& 2& .& .& 3& 1\cr
 1^{10}& .1^5   & .& .& .& .& .& .& .& .& .& .& .& .& .& 6& .& .& 8&10& .& .& 6& 3\cr
\noalign{\hrule}
 \omit& & ps& 1^2& ps& 21^2& 1^2& 321& 1^4& ps& 21^2& ps& 2^21^2& 21^4& 1^2& 321& 1^2& 1^4& 32^21& 321^3& 1^6& ps& 2^21^2& 21^2\cr
   }}$$}
\end{table}

\begin{table}[ht]
{\small\caption{Decomposition matrix for $\SU_{10}(q)$, $\ell| (q+1)$, $\ell>17$, cntd.}   \label{tab:2A9b}
$$\vbox{\offinterlineskip\halign{$#$\hfil\ \vrule height9pt depth2pt&&
      \hfil\ $#$\hfil\cr
   4321& \tw2A_9&                       1\cr
  431^3& 2^21.  &                       2& 1\cr
   42^3& 21.1^2 &                       2& .& 1\cr
42^21^2& 21^2.1 &                       2& 1& 1& 1\cr
  421^4& 2.1^3  &                       2& .& 1& 3& 1\cr
   41^6& 21^3.  &                       6& 1& 1& 6& 5& 1\cr
   3^31& 1.2^2  &                       2& .& .& .& .& .& 1\cr
 3^22^2& 1^2.21 &                       2& .& 1& .& .& .& 1& 1\cr
3^221^2& 1^3.2  &                       2& .& 1& 1& .& .& .& 1& 1\cr
 3^21^4& 1.21^2 &                       8& 5& 1& 3& 1& .& 1& 1& 3& 1\cr
32^31&\tw2A_5:.2&              3\pl\alpha& .& .& .& .& .& .& .& 2& .& 1\cr
32^21^3& .2^21  &             6\pl2\alpha& 5& .& .& .& .& 1& .& 3& 1&   2& 1\cr
321^5&\tw2A_5:.1^2           &1\pl3\alpha&.& .& .& .& .& .& .& 2& .&   3&   4& 1\cr
   31^7& .21^3  &            14\pl4\alpha& 5& 1&10&15& 7& 1& 1& 8& 1& 4& 10& 6& 1\cr
    2^5& 1^3.1^2&            10\pl4\alpha& 1& 1& 1& .& .& 1& 1& 3& .& 4&   .&   .& .& 1\cr
 2^41^2& 1^2.1^3&    10\pl4\alpha\mn\beta& 5& 1& 3& 1& .& 1& 1& 3& 1& 4&   3&   .& .& 1& 1\cr
 2^31^4& 1^4.1  &   10\pl4\alpha\mn3\beta& 1& 1& 7& 5& 1& .& 1& 4& .& 4&   8&   2& .& 1& 3& 1\cr
 2^21^6& 1.1^4  &   20\pl4\alpha\mn6\beta& 5& 1&13&16& 7& 1& 1& 8& 1& 4& 13& 6& 1& 1& 6& 5& 1\cr
   21^8& 1^5.   &  14\pl4\alpha\mn10\beta& 1& .& 6& 5& 1& .& .& 3& .& 4&  15&   2& .& 1& 10& 15& 7& 1\cr
 1^{10}& .1^5   & 26\pl10\alpha\mn15\beta& 5& .&10&15& 7& 1& 1&15& 1& 10& 45& 14& 1& 1& 15& 35& 28& 9& 1\cr
\noalign{\hrule}
 \omit& & c& 21^4& ps& 2^31^2& 2^21^4& 21^6& 1^4& 1^2& c& 1^6& c& c& c& 1^8& ps& c& c& c& c& c\cr
   }}$$
   Here, $\alpha,\beta\in\{0,1\}$. }
\end{table}

In order to prove this theorem, and especially in order to determine the column
corresponding to the partition $4321$, we will need to restrict a projective
indecomposable module from $\SU_{18}(q)$ to $\SU_{10}(q)$. Although we will
not need to determine explicitly the character of this module, we will require
relations on the decomposition numbers of $\SU_{18}(q)$ which are given in the
following Lemma.

\begin{lem}\label{lem:2A17}
 Assume $\ell > 17$ and $\ell | q+1$. Let us write the unipotent part
 of the projective character associated with the partition $4321^9$ as
 $$\begin{aligned}
   \Psi_{4321^9} =&\ \rho_{4321^9}+a_1 \rho_{431^{11}}+a_2\rho_{42^31^8}
   			+a_3\rho_{42^21^{10}}+a_4 \rho_{421^{12}} + a_5\rho_{41^{14}}
			+a_6\rho_{3^31^{9}}+a_7 \rho_{3^22^21^8} \\
		   &\  +a_8 \rho_{3^221^{10}}+a_9\rho_{3^31^{9}} + a_{10}\rho_{32^41^7}
			+a_{11} \rho_{32^31^9}+a_{12}\rho_{32^21^{11}}
			+a_{13} \rho_{321^{13}}+a_{14}\rho_{31^{15}}\\
		   &\	+a_{15}\rho_{2^61^6}
			+a_{16}\rho_{2^51^8}+a_{17}\rho_{2^41^{10}}+a_{18}\rho_{2^31^{12}}
			+a_{19}\rho_{2^21^{14}}+a_{20}\rho_{21^{16}}+a_{21}\rho_{1^{18}}.\\
 \end{aligned}$$
 Then the $a_i$'s are subject to the following relations: $a_2 = 2$,
 $a_3 = a_1$, $a_7 = a_6$, $a_{11} = 2a_8$ and
 $4+a_1+a_6+3a_8-a_{15}-a_{16}= 0$.
\end{lem}

\begin{proof}
We proceed as follows to obtain relations on the $a_i$'s: we consider several
characters $\rho$ such that $\rho^0$ can be expressed in terms of ``a few''
$\ell$-restrictions of Deligne--Lusztig characters $R_{\bT_{wF}}^\bG(1)$ and
then apply Lemma~\ref{lem:gettingrelations} to get
$\langle \Psi_{4321^9}, \rho^0\rangle = 0$. We start with the non-unipotent
characters $\rho$ associated with the multipartitions
$$\begin{aligned}
  &(321,321,1,1,1,1,1,1), (321,21,21,21,1,1,1), (321,21,21,1,1,1,1,1,1),\\
  &(21,21,21,21,21,21), (21,21,21,21,21,1,1,1)\\
  &(21,21,21,21,1,1,1,1,1,1), (21,21,21,1,1,1,1,1,1,1,1,1).\\
\end{aligned}$$
By Propositions~\ref{prop:nonunip} and~\ref{prop:partcases} their restrictions
to the set of $\ell'$-elements are irreducible Brauer characters different
from $\vhi_{4321^9}$, so that we obtain seven relations
$\langle \Psi_{4321^9}, \rho^0\rangle = 0$, from which we deduce
$$\begin{aligned}
  a_{7} = &\ -2+a_2+a_6, 		&a_{16} =&\ 3a_1-2a_3+3a_8,\\
  a_{10} =&\ a_6, 			&a_{17} = &\ -2+a_2+3a_4-2a_9+3a_{12},\\
  a_{11} =&\ 2a_1-2a_3+2a_8,\qquad&a_{18} =&\ -a_1+a_3-a_5+2a_{13},\\
  a_{15} =&\ 2+a_2+a_6. \\
\end{aligned}$$
Further relations could be easily obtained but we will not need them for
our purpose.
\par
We now want to consider the non-unipotent characters $R_\bL^\bG (\eta)$ where
$\bL$ is a Levi subgroup of semi-simple type $\tw2A_{15}$ and $\eta$ is
a lift of the Brauer character $\vhi_{2^k 1^{16-2k}}$. It is not clear whether
$\vhi_{4321^{9}}$ occurs in $R_\bL^\bG (\eta)^0$ or not, unless $k = 0$.
However we can adapt the proof of Lemma~\ref{lem:partition31} to these
characters. With $I = \{s_2,\ldots,s_{16}\}$, the product $w_Iw_0$ is just
the transposition $(1,18)$, a cycle of type $21^{16}$. Then from the definition
of $\eta$, we deduce that $R_\bL^\bG (\eta)^0$ can be written as a linear
combination of Deligne--Lusztig characters $R_{\bT_{\sigma_\mu w_0F}}^\bG(1)^0$
for $\mu \in \{21^{16},321^{13},\ldots,3^{k}21^{16-3k}\}$. We write $\cC$ for
the corresponding set of $F$-conjugacy classes, and $\cC_\nleq$ for the set
of classes $[w]_F$ satisfying $\cO \nleq [w]_F$ for all $\cO \in \cC$.
By Proposition~\ref{prop:orderingfor321} and Remark~\ref{rem:partitions321},
we have  $[w]_F \in \cC_{\nleq}$ if and only if $ww_0$ is not a cycle of type
$\mu$ with
$$\mu \in \{1^{18},31^{15}, \ldots,3^{k+1}1^{15-3k}\}\cup \{51^{13},\ldots,
  53^{k-1}1^{16-3k}\} \cup\{21^{16},321^{13}, \ldots,3^{k}21^{16-3k}\} .$$
\par
We now want to determine which PIM $P_\la$ occurs in some Deligne--Lusztig
character $R_w$ for $[w]_F\in\cC_\nleq$, since by
Lemma~\ref{lem:gettingrelations} this will force
$\langle \Psi_\la, R_\bL^\bG (\eta)^\circ\rangle = 0$.
By Remark~\ref{rem:firstprojectives}, we know that for
$\la \in \{1^{18}, 21^{16},\ldots,2^{k+1}1^{16-2k}\} \cup \{ 321^{13},
\ldots,32^{k}1^{15-2k}\}$ and $k \leq 4$, the corresponding PIM is
the projective cover of a cuspidal module, and it can occurs only in
a Deligne--Lusztig character $R_{\sigma_\mu w_0}$ with $\mu \in
\{1^{18},31^{15}, \ldots,3^{k+1}1^{15-3k}\} \cup \{51^{13},\ldots,
53^{k-1}1^{16-3k}\}$, therefore it can not appear in $R_w$ for
$[w]_F\in \cC_{\nleq}$. By induction on $k$, we can assume that
for $\la \in \{31^{15},41^{14},421^{12},\ldots,42^{k-2}1^{18-2k}\}$, $P_\la$
does not occur in some $R_w$ with $[w]_F \in \cC_{\nleq} \cup \{3^{k+1}1^{15-3k},
53^{k-1}1^{16-3k},3^{k}21^{16-3k}\}$ so in particular not in some
$R_w$ with $[w]_F\in \cC_{\nleq}$. We claim that $P_{42^{i-1}1^{16-2i}}$
does not occur in $\cC_\nleq$ either. Otherwise by Lemma~\ref{lem:gettingrelations}
we would have $\langle P_{42^{k-1}1^{16-2k}},R_\bL^\bG (\eta)^0\rangle = 0$. But
for $k>0$, the projective indecomposable module $P_{42^{k-1}1^{16-2k}}$ is just
the Harish-Chandra induction of $P_{2^{k}1^{16-2k}}$ since it is
indecomposable by \cite[Prop.~4.4 and Lemma 4.6]{GHM2}. The latter is known
explicitly by Theorem~\ref{thm:corner}, and
we can check that $\langle \Psi_{42^{k-1}1^{16-2k}},
R_\bL^\bG (\eta)^0\rangle = \langle R_\bL^\bG (\Psi_{2^{k}1^{16-2k}}),
R_\bL^\bG (\eta) \rangle =2$.
\par
Finally, we have found $3k+3$ partitions $\lambda$ such that $P_\la$ cannot
occur in $R_w$ for $[w]_F \in \cC_\nleq$. Since $3k+3$ is exactly the number
of $F$-conjugacy classes that are not in $\cC_\nleq$, we deduce from the
usual dimension argument that for every $\lambda$ outside the set
$$ \{1^{18},21^{16},\ldots,2^{k+1}1^{16-2k}\} \cup \{ 321^{13},
  \ldots,32^{k}1^{15-2k}\}\cup \{31^{15}, 41^{14},421^{12},\ldots,
  42^{k-1}1^{16-2k}\}$$
the projective module $P_\lambda$ occurs in some Deligne--Lusztig character
$R_w$ with $w \in \cC_\nleq$. By Lemma~\ref{lem:gettingrelations} this forces
in particular $\langle \Psi_{4321^9}, R_\bL^\bG (\eta) \rangle =0$.
\par
This scalar product with $k=3$ and $k=4$ yields two relations on the $a_i$'s,
which together with the previous give $a_2=2$ and $a_1= a_3$.
\end{proof}

\begin{proof}[Proof of Theorem~\ref{thm:2A9,d=2}]
As before, Harish-Chandra induction of PIMs from proper Levi subgroups
yields most of the columns in the tables. In addition, the lower-right corner
is obtained by Theorem~\ref{thm:corner} and
Proposition~\ref{prop:anotherprojective} so that the only missing
columns correspond to the projective characters $\Psi_{4321}$,
$\Psi_{3^221^2}$, $\Psi_{32^31}$ and $\Psi_{32^21^3}$. We will
denote by $e_1,\ldots,e_{47}$ the missing entries in these columns.
\par
Recall that if $\rho_\mu$ is a unipotent constituent of $\Psi_\la$ then
$\mu\trianglelefteq\la$ (see the proof of Proposition~\ref{prop:smallirr}),
so that $e_{42}= 0$. For the other relations, we use
Propositions~\ref{prop:nonunip} and~\ref{prop:partcases} with the partitions
$32^21^3$, $321^5$, $2^31^4$, $2^21^6$, $21^8$ and $1^{10}$. This
yields\footnote{Relations on $e_1,\ldots, e_{19}$ are also obtained, but
they are not sufficient to determine entries in the column
of $\Psi_{4321}$ and we shall rather deal with them at the end of the proof.}
$e_{22} = e_{20}+2e_{21}-4$,
$e_{23} = 3e_{21}-4$, $e_{27} = -3e_{20}-2e_{25}+3e_{26}+10$,
$e_{28} = 23-6e_{20}-4e_{21}+e_{24}-5e_{25}+6e_{26}$,
$e_{29} = 30 -10e_{20}- 9e_{25}+10e_{26}$,
$e_{30} = 45-15e_{20}+2e_{21}+e_{24}-14e_{25}+15e_{26}$, $e_{31} =2$,
$e_{32} = 3$, $e_{36} = -2e_{34}+3e_{35}$, $e_{37} = e_{33}-5e_{34}+6e_{35}-4$,
$e_{38} = -9e_{34}+10e_{35}$, $e_{39} = e_{33} -14e_{34}+15e_{35} +2$,
$e_{40} =4$, $e_{44} = 3e_{43}-1$, $e_{45}=e_{41}+6e_{43}-15$,
$e_{46} = 10e_{43}-15$ and $e_{47} = e_{41}+15e_{43}-10$.
Further relations are obtained from Lemma~\ref{lem:partition31}. Together
with the previous ones, they yield $e_{24} = e_{20}+4e_{21}-3$,
$e_{33} = 4$ and $e_{41} = 10$.

In addition, we look at the restriction to $\ell'$-elements of other
non-unipotent characters. They no longer give irreducible Brauer characters,
but they still yield the following inequalities:
$$\begin{array}{c|c}
  \text{multipartition } \bmu & \text{relations} \\\hline
  (2,2,1,1,1,1,1,1)	& e_{20} \geq 3 \\
  (2,2,2,2,1,1)	& e_{43} \geq 3 \\
  (2,2,2,2,2)		& e_{34} \geq 4 \\
  (3,2,2,2,1)		& e_{21} \geq 2 \\
  (21,2,2,2,1)		& e_{26} \geq e_{25}+e_{20}-e_{21}-1\\
  (2^3,2,2)		& e_{25} \geq 2e_{21}-1\\
  (2^31, 21)		& e_{35} \geq e_{34} \\
\end{array}$$
As in the case of $\SU_8(q)$, we compute the character of well-chosen PIMs
for unitary groups of larger rank and then deduce upper bounds for the $e_i$'s
by Harish-Chandra restriction. We start with the restriction of the projective
character $\Psi_{32^21^5}$ of $\SU_{12}(q)$, which we have computed
previously. It contains $\Psi_{32^21^3}$, and since the coefficient of
$\rho_{2^41^2}$ in this restriction is $3$, we deduce that $e_{43} \leq 3$,
hence $e_{43} = 3$ with the previous inequalities. This yields
$e_{44} = 8$, $e_{45} = 13$, $e_{46} = 15$ and $e_{47} =45$.
\par
The partitions $32^31^7$, $32^21^9$ and $321^{11}$ of $16$ satisfy the
assumptions of Proposition~\ref{prop:nonunip} (see
Proposition~\ref{prop:partcases}.(2)), so that we can compute the projective
characters $\Psi_{32^31^7}$ and $\Psi_{3^221^8}$ of $\SU_{16}(q)$
(using Lemma~\ref{lem:partition31} in addition). They are given by:
$$\begin{aligned}
  \Psi_{32^31^7} =&\ \rho_{32^31^7}+8\rho_{32^21^9}+36\rho_{321^{11}}+120\rho_{31^{13}}
  			+4\rho_{2^51^6}+24\rho_{2^41^8} + 76\rho_{2^31^{10}}\\
		   &\ +156\rho_{2^21^{12}}+180\rho_{21^{14}}+660\rho_{1^{16}},\\
  \Psi_{3^221^8} =&\ \rho_{3^221^8}+x\rho_{3^21^{10}}
  		      + 2\rho_{32^31^7}+x\rho_{32^21^9}+2\rho_{321^{11}}+(x+11)\rho_{31^{13}}\\
  		   &\ +3\rho_{2^51^6}+x\rho_{2^41^8} + 4\rho_{2^31^{10}}
		      +(x+11)\rho_{2^21^{12}}+3\rho_{21^{14}}+(x+24)\rho_{1^{16}},\\
\end{aligned}$$
for some undetermined $x$. Using again Harish-Chandra restriction to
$\SU_{10}(q)$ one finds $e_{21} \leq 2$, $e_{25} \leq 3$ and $e_{34} \leq 4$.
The previous inequalities force equalities to hold and we deduce the following
part of the decomposition matrix:
$$ \begin{array}{c|cccccccc}
 3^221^2	&          1\\
 3^21^4	&     e_{20}& 1\\
 32^31	&          2& .&      1\\
 32^21^3&     e_{20}& 1&      2& 1\\
 321^5	&          2& .&      3& 4& 1\\
 31^7	& e_{20}\pl2& 1&      4&10& 6& 1\\
 2^5	&          3& .&      4& .& .& .& 1\\
 2^41^2	&     e_{26}& 1& e_{35}& 3& .& .& 1 & 1\\
\end{array}$$
Let us have a closer look at the restriction of $\Psi_{3^221^8}$: the
coefficients of $\rho_{3^21^4}$, $\rho_{32^21^3}$ and $\rho_{2^41^2}$ in this
restriction are equal. No matter how many copies of $\Psi_{3^21^4}$,
$\Psi_{32^31}$, $\Psi_{32^21^3}$ we remove from this restriction, we observe
from the previous matrix that the multiplicity of $\rho_{3^21^4}$ will
be larger that the multiplicity of $\rho_{32^31}$, which in turn will be
larger than the multiplicity of $\rho_{32^21^3}$, so that
$e_{20} \geq e_{22} \geq e_{26}$. On the other hand,
$e_{26} \leq e_{25} +e_{20}-e_{21}-1$, that is $e_{26} \leq e_{20}$ since
$e_{21} = 2$ and $e_{25} = 3$. We deduce that $e_{26} = e_{20}$, and the
relations on the $e_i$'s become $e_{22} = e_{20}$, $e_{23} = 2$, $e_{27} = 4$,
$e_{24} = e_{20}+5$,
$e_{28} = e_{20}+5$, $e_{29} = 3$, $e_{30} = e_{20}+12$, $e_{36} = 3e_{35}-8$,
$e_{37} = 6e_{35}-20$, $e_{38} = 10e_{35}-36$ and
$e_{39} = 15e_{35}-50$.
\par
To obtain the values of $e_{20}$ and $e_{35}$ we decompose Deligne--Lusztig
characters $R_w$ for various $w$. Starting with a Coxeter element
$w=s_1s_2s_3s_4s_5$, we find
$$\begin{aligned}
 R_w =&\ \Psi_{10}-\Psi_{91}-\Psi_{82}+\Psi_{81^2}+\Psi_{73}-\Psi_{721}+\Psi_{62^2}
 	+\Psi_{52^21}-\Psi_{42^3}-\Psi_{32^31}\\
      &\ +\Psi_{2^5}+(e_{35}-4)\Psi_{2^41^2}.\\
\end{aligned}$$
We deduce that $e_{35} \leq 4$ from Proposition~\ref{prop:constitPw}, and
therefore $e_{35} = 4$ by the previous inequalities.
\par
Now let $w' = s_1s_2s_3s_4s_6s_5s_6$. We can first check that $\Psi_{4321}$
occurs with multiplicity $-1$ in $R_{w'}$. Then for $v < w'$, we can decompose
$R_v$ and we find that the only PIMs which can occur in $R_v$ correspond to
partitions lying in the set $\{\mu \ntrianglelefteq 4321\} \cup
\{42^3, 42^21^2,3^22^2,32^31,2^5\}$. We deduce from Proposition~\ref{prop:constitPw}
that for any partition $\la$ outside of this set, the multiplicity of
$\Psi_\la$ in $R_{w'}$ is non-positive. This yields $14$ inequalities giving
upper bounds for $13$ variables among $e_1,\ldots,e_{19}$. As usual, lower
bounds are obtained by $\ell$-restrictions of non-unipotent characters
corresponding to well-chosen partitions. We give, in the following table,
the partition to consider for each variable.
$$\begin{array}{r|l||r|l}
e_1 & (4,3,1,1,1)&       e_{12} & (321,1,1,1,1) \\
e_4 & (4,2,1,1,1,1)&     e_{13} & (3,1,1,1,1,1,1,1) \\
e_5 & (4,1,1,1,1,1,1)&   e_{15} & (2^31,21) \\
e_6 & (3,3,3,1)&         e_{16} & (21,21,21,1) \\
e_8 & (3,3,2,1,1)&       e_{17} & (21,21,1,1,1,1) \\
e_9 & (3,3,1,1,1,1)&     e_{18} & (21,1,1,1,1,1,1,1) \\
e_{11} & (321,21,1)&     e_{19} & (1,1,1,1,1,1,1,1,1,1) \\
\end{array}$$
Except for $e_{15}$, the lower and upper bounds all match and we obtain
$e_1 = e_6 = 2$, $e_4 = -2e_2+3e_3$, $e_5 = 4-5e_2+6e_3$,
$e_8 = -e_2+e_3+e_7$,  $e_9 = 6-3e_2+3e_3+e_7$, $e_{11} = 2e_{10}$,
$e_{12} = 3e_{10}-8$, $e_{13} = 10(e_3-e_2)+e_7+4e_{10}$,
$e_{16} = -2e_{14}+3e_{15}$, $e_{17} = 10-5e_{14}+6e_{15}$,
$e_{18} = 6+15(e_2-e_3)-e_7-9e_{14}+10e_{15}$ and
$e_{19} = 20e_2-21e_3+6e_{10} -14e_{14}+15e_{15}$.
For $e_{15}$ we can only deduce that $-1 \leq e_{15}-e_{14} \leq 2(e_3-e_2)$.
\par
As before, to get extra upper bounds one needs to consider the Harish-Chandra
restriction of some projective character in a larger group. Here, we can
restrict $\Psi_{4321^9}$ using its expression given in Lemma~\ref{lem:2A17}.
We give in the following table the multiplicity of $\Psi_\la$ in this
restriction for various partitions $\la$:
$$\begin{array}{r|l}
  \la 		& \text{coefficient}\\\hline
  42^3		& 2-e_2 \\
  42^21^2	& e_2-e_3\\
  3^22^2	& e_2-e_7 \\
  32^31		& 4-e_{10}\\
  2^5		& -4-e_2+e_3+e_7+4e_{10}-e_{14}. \\
\end{array}$$
Since these coefficients must be nonnegative, we deduce upper bounds for
$e_2$, $e_3$, $e_7$, $e_{10}$ and $e_{14}$. Lower bounds are obtained by
$\ell$-reduction of non-unipotent characters corresponding to the following
multipartitions:
$$\begin{array}{r|l}
  \bla 	& \text{relation}\\\hline
  (4,2,2,2)	& e_2 \geq 2 \\
  (2,21,21)	& e_3 \geq 2\\
  (32,32)	& e_7 \geq e_2 \\
  (321,1^4)	& e_{10} \geq 3 \\
  (2,2,2,2,2)& e_{14} \geq 4-5e_2+e_3+e_7+4e_{10}.\\
\end{array}$$
This yields successively $2 = e_2=e_3=e_7$, $e_{10} \in \{3,4\}$ and
$e_{14} = 4e_{10}-2$. If we set $\alpha = e_{10}-4$ and
$\beta = e_{14}-e_{15}$, we deduce that $\alpha,\beta \in \{0,1\}$ and any
$e_i$ with $i=1,\ldots,19$ can be expressed in terms of $\alpha$ and $\beta$.
The explicit relations are written in Table~\ref{tab:2A9b}
\par
Finally we consider the Deligne--Lusztig character for
$w''= s_1 s_2 s_3 s_6 s_5 s_4 s_6 s_5 s_6$. When $v < w''$, one can check
that the PIM $\Psi_{3^21^4}$ does not occur in $R_v$, therefore by
Proposition~\ref{prop:constitPw} it must occur in $R_{w''}$ with a
non-positive coefficient. This coefficient is actually given by $6e_{20}-18$
so that $e_{20} \leq 3$. With the previous inequalities this proves
that $e_{20} = 3$.
\end{proof}

\begin{rem}   \label{rem:equalrows}
One striking observation to make on the decomposition matrices is the
repetition of large portions of certain rows. We point out a few examples
in the largest case, that is, for $\SU_{10}(q)$ when $q\equiv-1\pmod\ell$.
If $\la$ is one of the partitions $71^3,64,531^2,31^7$, the first part of the
corresponding row of the decomposition matrix repeats itself for the
partition $\mu$ obtained by moving one box from the first row of $\la$ to the
second row. The same happens for other partitions $\la$ as well. This process
does not preserve Harish-Chandra series; we do not have any explanation for
this phenomenon.
\end{rem}

\subsection{The case $\ell|(q^2-q+1)$}

\begin{thm}   \label{thm:2An,d=6}
 Let $\ell>n$ be a prime. Then the $\ell$-modular decomposition matrices for the
 unipotent $\ell$-blocks of $\SU_n(q)$, $3\le n\le10$, $\ell|(q^2-q+1)$, are
 as given in Tables~\ref{tab:2An,d=6}--\ref{tab:2A9,d=6,def2}.\par
 We also indicate the modular Harish-Chandra series in the bottom rows of the
 tables.
\end{thm}

\begin{table}[ht]
\caption{$\SU_n(q)$, $3\le n\le 5$, $\ell| (q^2-q+1)$, $\ell>n$}
  \label{tab:2An,d=6}
$$\begin{array}{cccccccc}
 \SU_3(q):\qquad & 3& \vr& 1^3& \vr& \bigcirc& \vr& 21\\
                 &  & ps& & c& & c\\
   & & \\
   & & \\
 \SU_4(q):\qquad & 4& \vr& 2^2& \vr& 1^4& \vr& \bigcirc\\
                 & & ps& & ps& & c\\
   & & \\
   & & \\
 \SU_5(q):\qquad & 5& \vr& 2^21& \vr& \bigcirc& \vr& 21^3\\
%%               &   & ps& & 1^3& & 21\\
   & & \\
                 & 32& \vr& 1^5& \vr& \bigcirc& \vr& 41\\
%%                 &   & ps& & 1^3& & 21\\
   & & \\
   & & \\
 \SU_9(q):\qquad & 72& \vr& 421^3& \vr& \bigcirc& \vr& 42^21\\
   & & \\
                 & 521^2& \vr& 2^21^5& \vr& \bigcirc& \vr& 431^2\\
                 &   & ps& & 1^3& & 21\\
\end{array}$$
\end{table}

\begin{table}[ht]
\caption{$\SU_6(q)$, $\ell| (q^2-q+1)$, $\ell>6$}   \label{tab:2A5,d=6}
$$\vbox{\offinterlineskip\halign{$#$\hfil\ \vrule height11pt depth4pt&&
      \hfil\ $#$\hfil\cr
     6&                 1& 1\cr
    51&          q\Ph{10}& .& 1\cr
    42&     q^2\Ph3\Ph6^2& .& .& 1\cr
  41^2&    q^3\Ph4\Ph{10}& 1& .& .& 1\cr
   3^2&    q^3\Ph3\Ph{10}& 1& 1& .& .& 1\cr
   321& q^4\Ph1^3\Ph3\Ph4& .& .& .& .& .& 1\cr
   2^3&    q^6\Ph4\Ph{10}& 1& .& .& 1& 1& .& 1\cr
  31^3&    q^6\Ph3\Ph{10}& .& 1& .& .& 1& .& .& 1\cr
2^21^2&     q^7\Ph3\Ph6^2& .& .& .& .& .& .& .& .& 1\cr
  21^4&     q^{10}\Ph{10}& .& .& .& 1& .& .& 1& .& .& 1\cr
   1^6&            q^{15}& .& .& .& .& 1& 2& 1& 1& .& .& 1\cr
\noalign{\hrule}
 \omit& & ps& ps& ps& ps& ps& c& B& 1^4& ps& c& c\cr
   }}$$
Here, $B$ stands for the (cuspidal) Steinberg PIM of $\SL_3(q^2)$.
\end{table}

\begin{table}[ht]
\caption{$\SU_7(q)$, $\ell| (q^2-q+1)$, $\ell>7$}   \label{tab:2A6,d=6}
$$\vbox{\offinterlineskip\halign{$#$\hfil\ \vrule height11pt depth4pt&&
      \hfil\ $#$\hfil\cr
    7&                      1& 1\cr
   61&          q\Ph1\Ph3\Ph6& .& 1\cr
   52&         q^2\Ph4\Ph{14}& 1& .& 1\cr
 51^2&     q^3\Ph3\Ph6\Ph{10}& .& .& .& 1\cr
   43&     q^3\Ph1\Ph3\Ph{14}& .& .& .& .& 1\cr
  421&      q^4\Ph{10}\Ph{14}& 1& .& .& .& .& 1\cr
 3^21&     q^5\Ph3\Ph6\Ph{14}& .& .& .& 1& .& .& 1\cr
 41^3& q^6\Ph1\Ph3\Ph4\Ph{10}& .& .& .& .& 1& .& .& 1\cr
 32^2&     q^6\Ph3\Ph6\Ph{14}& .& .& .& .& .& .& .& .& 1\cr
321^2&      q^7\Ph{10}\Ph{14}& .& .& 1& .& .& .& .& .& .& 1\cr
 2^31&     q^9\Ph1\Ph3\Ph{14}& .& .& .& .& .& .& .& 1& .& .& 1\cr
 31^4&  q^{10}\Ph3\Ph6\Ph{10}& .& .& .& .& .& .& .& .& 1& .& .& 1\cr
2^21^3&     q^{11}\Ph4\Ph{14}& 1& .& 1& .& .& 1& .& .& .& .& .& .& 1\cr
 21^5&     q^{15}\Ph1\Ph3\Ph6& .& .& .& .& .& .& .& .& .& .& .& .& .& 1\cr
  1^7&                 q^{21}& .& .& 1& .& .& .& .& .& .& 1& 2& .& 1& .& 1\cr
\noalign{\hrule}
 \omit& & ps& 21& ps& ps& 21& 1^3& 1^3& 21& ps& B& c& 1^3& 1^3& 21& c\cr
   }}$$
\end{table}

\begin{table}[ht]
\caption{$\SU_8(q)$, $\ell| (q^2-q+1)$, $\ell>8$}   \label{tab:2A7,d=6}
$$\vbox{\offinterlineskip\halign{$#$\hfil\ \vrule height11pt depth4pt&&
      \hfil\ $#$\hfil\cr
     8&                      1& 1\cr
    71&               q\Ph{14}& .& 1\cr
    62&     q^2\Ph4\Ph8\Ph{10}& .& .& 1\cr
  61^2&     q^3\Ph3\Ph6\Ph{14}& .& .& .& 1\cr
    53&     q^3\Ph4\Ph8\Ph{14}& 1& .& .& .& 1\cr
   4^2&     q^4\Ph3\Ph8\Ph{14}& .& 1& 1& .& .& 1\cr
  521& q^4\Ph1^3\Ph3\Ph4^2\Ph8& .& .& .& .& .& .& 1\cr
   431&  q^5\Ph8\Ph{10}\Ph{14}& 1& .& .& .& .& .& .& 1\cr
  51^3&  q^6\Ph3\Ph{10}\Ph{14}& .& .& .& .& 1& .& .& .& 1\cr
  42^2&   q^6\Ph4^2\Ph8\Ph{14}& .& .& 1& .& .& 1& .& .& .& 1\cr
421^2&q^7\Ph3\Ph6^2\Ph8\Ph{10}& .& .& .& .& .& .& .& .& .& .& 1\cr
  3^22& q^7\Ph3\Ph6\Ph8\Ph{14}& .& .& .& 1& .& .& .& .& .& .& .& 1\cr
3^21^2&   q^8\Ph4^2\Ph8\Ph{14}& 1& .& .& .& 1& .& .& .& 1& .& .& .& 1\cr
 41^4&q^{10}\Ph3\Ph{10}\Ph{14}& .& .& .& .& .& .& .& .& .& 1& .& .& .& 1\cr
 32^21&  q^9\Ph8\Ph{10}\Ph{14}& .& 1& .& .& .& 1& .& .& .& .& .& .& .& .& 1\cr
321^3&q^{11}\Ph1^3\Ph3\Ph4^2\Ph8&.&.& .& .& .& .& .& .& .& .& .& .& .& .& .& 1\cr
   2^4&  q^{12}\Ph3\Ph8\Ph{14}& 1& .& .& .& .& .& .& 1& .& .& .& .& 1& .& .& .& 1\cr
2^31^2&  q^{13}\Ph4\Ph8\Ph{14}& .& .& .& .& .& 1& .& .& .& 1& .& .& .& 1& .& .& .& 1\cr
  31^5&  q^{15}\Ph3\Ph6\Ph{14}& .& .& .& .& .& .& .& .& .& .& .& 1& .& .& .& .& .& .& 1\cr
2^21^4&  q^{16}\Ph4\Ph8\Ph{10}& .& .& .& .& .& .& 2& .& 1& .& .& .& 1& .& .& .& 1& .& .& 1\cr
  21^6&          q^{21}\Ph{14}& .& .& .& .& .& .& .& 1& .& .& .& .& .& .& .& .& 1& .& .& .& 1\cr
   1^8&                 q^{28}& .& .& .& .& .& 1& .& .& .& .& .& .& .& .& 1& 2& .& 1& .& .& .& 1\cr
\noalign{\hrule}
 \omit& & ps& ps& ps& ps& ps& ps& \!321& ps& 1^4& ps& ps& ps& ps& 21^4\!& 1^4& \!321& B& B& 1^4& 1^6& 21^4\!& 1^6& \cr
   }}$$
\end{table}

\begin{table}[ht]
{\small\caption{$\SU_9(q)$, $\ell| (q^2-q+1)$, $\ell>9$}   \label{tab:2A8,d=6}
$$\vbox{\offinterlineskip\halign{$#$\hfil\ \vrule height11pt depth3pt&&
      \hfil\ $#$\hfil\cr
     9&                             1& 1\cr
    81&                 q\Ph1\Ph4\Ph8& .& 1\cr
  71^2&            q^3\Ph4\Ph8\Ph{14}& 1& .& 1\cr
    63&      q^3\Ph1\Ph4^2\Ph8\Ph{18}& .& 1& .& 1\cr
   621&      q^4\Ph{10}\Ph{14}\Ph{18}& 1& .& .& .& 1\cr
    54&         q^4\Ph8\Ph{14}\Ph{18}& 1& .& .& .& .& 1\cr
  61^3&    q^6\Ph1\Ph3\Ph4\Ph8\Ph{14}& .& 1& .& 1& .& .& 1\cr
  52^2&   q^6\Ph4^2\Ph8\Ph{10}\Ph{18}& 1& .& 1& .& .& 1& .& 1\cr
  4^21&     q^6\Ph4\Ph8\Ph{14}\Ph{18}& 1& .& 1& .& 1& .& .& .& 1\cr
   432& q^7\Ph1\Ph4\Ph8\Ph{14}\Ph{18}& .& .& .& 1& .& .& .& .& .& 1\cr
   3^3&     q^9\Ph3\Ph8\Ph{14}\Ph{18}& .& .& 1& .& .& .& .& 1& 1& .& 1\cr
  51^4&  q^{10}\Ph3\Ph8\Ph{10}\Ph{14}& .& .& 1& .& .& .& .& 1& .& .& .& 1\cr
3^221&q^{10}\Ph1\Ph4\Ph8\Ph{14}\Ph{18}&.& .& .& .& .& .& 1& .& .& .& .& .& 1\cr
  32^3&  q^{12}\Ph4\Ph8\Ph{14}\Ph{18}& .& .& .& .& .& 1& .& 1& .& .& 1& .& .& 1\cr
3^21^3&q^{12}\Ph4^2\Ph8\Ph{10}\Ph{18}& 1& .& 1& .& 1& 1& .& 1& 1& .& 1& 1& .& .& 1\cr
  41^5& q^{15}\Ph1\Ph3\Ph4\Ph8\Ph{14}& .& .& .& 1& .& .& 1& .& .& 1& .& .& .& .& .& 1\cr
 321^4&   q^{16}\Ph{10}\Ph{14}\Ph{18}& .& .& .& .& .& 1& .& .& .& .& .& .& .& 1& .& .& 1\cr
  2^41&      q^{16}\Ph8\Ph{14}\Ph{18}& 1& .& .& .& 1& 1& .& .& .& .& .& .& .& .& 1& .&  .& 1\cr
2^31^3&   q^{18}\Ph1\Ph4^2\Ph8\Ph{18}& .& .& .& .& .& .& 1& .& .& .& .& .& 1& .& .& 1&  .& .& 1\cr
  31^6&         q^{21}\Ph4\Ph8\Ph{14}& .& .& .& .& .& 1& .& 1& .& .& 1& 1& 2& 1& 1& .&  .&   .& .& 1\cr
  21^7&            q^{28}\Ph1\Ph4\Ph8& .& .& .& .& .& .& .& .& .& 1& .& .& .& .& .& 1&  .& 2&   1& .& 1\cr
   1^9&                        q^{36}& .& .& .& .& .& 1& .& .& .& .& 1& .& 2& 1& 1& .&  3& 1& 2& 1& .& 1\cr
\noalign{\hrule}
 \omit& & ps& 21& ps& 21& 1^3& ps& 21& ps& 1^3& \!21.B& \!1^3.B& 1^3& 2^31& B& 1^3& 21& c& c& c& 1^7& c& c\cr
   }}$$
}
\end{table}

\begin{table}[ht]
{\small\caption{$\SU_{10}(q)$, $\ell| (q^2-q+1)$, $\ell>10$, principal block}   \label{tab:2A9,d=6}
$$\vbox{\offinterlineskip\halign{$#$\hfil\ \vrule height11pt depth3pt&&
      \hfil\ $#$\hfil\cr
     10& 5.      & 1\cr
     82& 4.1     & 1& 1\cr
     73& 1.4     & 1& .& 1\cr
    721& \tw2A_5:2.& .& .& .& 1\cr
   71^3& .41     & .& .& 1& .& 1\cr
  621^2& 3.1^2   & .& 1& .& .& .& 1\cr
    5^2& 2.3     & 1& 1& 1& .& .& .& 1\cr
  52^21& .32     & .& .& 1& .& 1& .& 1& 1\cr
  521^3&\tw2A_5:1.1& .& .& .& 1& .& .& .& .& 1\cr
 4^21^2& 2.21    & 1& 1& 1& .& 1& 1& 1& .& .& 1\cr
   43^2& 2^2.1   & 1& 1& .& .& .& .& 1& .& .& .& 1\cr
   4321& \tw2A_9 & .& .& .& .& .& .& .& .& .& .& .& 1\cr
  431^3& 2^21.   & 1& .& .& .& .& .& .& .& .& .& 1& .& 1\cr
   42^3& 21.1^2  & 1& 1& .& .& .& 1& 1& .& .& 1& 1& .& .& 1\cr
  421^4& 2.1^3   & .& .& .& .& 1& 1& .& .& 2& 1& .& .& .& 1& 1\cr
   41^6& 21^3.   & .& .& .& .& .& .& .& .& .& .& 1&d_4& 1& 1& .& 1\cr
   3^31& 1.2^2   & 1& .& 1& .& 1& .& 1& 1& .& 1& .&d_5& .& .& .& .& 1\cr
  321^5&\tw2A_5:.1^2& .& .& .& .& .& .& .& .& 1& .& .&d_6& .& .& .& .& .& 1\cr
    2^5& 1^3.1^2 & 1& .& .& .& .& .& 1& .& .& 1& 1&d_7& 1& 1& .& .& 1& .& 1\cr
 2^31^4& 1^4.1   & .& .& .& .& .& .& 1& .& .& .& 1&d_8& 1& 1& .& 1& .& .& 1& 1\cr
 2^21^6& 1.1^4   & .& .& .& 2& 1& .& 1& 1& 2& 1& .&d_9& .& 1& 1& .& 1& .& 1& .& 1\cr
 1^{10}& .1^5    & .& .& .& 2& .& .& 1& 1& .& .& .&d_{10}&.& .& .& .& 1& 2& 1& 3& 1 &1\cr
\noalign{\hrule}
 \omit& & ps& ps& ps& 321& 1^4& ps& ps& 1^4& 321& ps& ps& 4321& 21^4& 1^3& 1^6& 21^4& 1^4.1^3& c& 1^3& c& 1^6& c\cr
   }}$$
Here $d_{10} = 6-3d_4+3d_5+2d_6-3d_7+3d_8+d_9$}
 \end{table}

\begin{table}[ht]
\caption{$\SU_{10}(q)$, $\ell| (q^2-q+1)$, blocks $31$ and $21^2$, $\ell>10$}   \label{tab:2A9,d=6,def2}
\hbox{\hskip 1cm
\vbox{\offinterlineskip\halign{$#$\hfil\ \vrule height11pt depth3pt&&
      \hfil\ $#$\hfil\cr
     91& 1\cr
     64& 1& 1\cr
   62^2& .& 1& 1\cr
   61^4& .& .& 1& 1\cr
   4^22& 1& 1& 1& .& 1\cr
3^221^2& .& .& 1& 1& 1& 1\cr
  32^31& .& .& .& .& .& .& 1\cr
32^21^3& 1& .& .& .& 1& .& .& 1\cr
   31^7& .& .& .& .& 1& 1& 2& 1& 1\cr
\noalign{\hrule}
 \omit& ps& ps& ps& 21^4& ps& 1^3& 321& 1^4& 1^6\cr
   }}\hskip 2cm
\vbox{\offinterlineskip\halign{$#$\hfil\ \vrule height11pt depth3pt&&
      \hfil\ $#$\hfil\cr
   81^2& 1\cr
    631&1& 1\cr
    541& .& .& 1\cr
    532& 1& .& .& 1\cr
   51^5& .& .& .& 1& 1\cr
 3^22^2& 1& 1& .& 1& .& 1\cr
 3^21^4& .& .& 2& 1& 1& 1& 1\cr
 2^41^2& .& 1& 2& .& .& 1& 1& 1\cr
   21^8& .& 1& .& .& .& .& .& 1& 1\cr
\noalign{\hrule}
 \omit& ps& ps& 321& ps& 1^4& ps& 1^6& 1^3& 21^4\cr
   }} }
\end{table}

\begin{proof}
As before, we label the characters of PIMs by the partition labelling the
corresponding Brauer character.
For $n\le5$, the Sylow $\ell$-subgroups of $\SU_n(q)$ are cyclic. The Brauer
trees in these cases were determined by Fong and Srinivasan \cite{FS2}, but
can also easily be found by Harish-Chandra induction.
Now let $n=6$. Harish-Chandra induction of PIMs from proper Levi subgroups
gives all the columns in Table~\ref{tab:2A5,d=6} except for the 6th and the
last two, corresponding to the characters $\Psi_{321}$, $\Psi_{21^4}$ and
$\Psi_{1^6}$. By triangularity, we just have to determine the five entries
below the diagonal for the 6th projective, which we denote by $a_1,\ldots,a_5$,
and the one missing entry in the 10th projective. As observed by Gerhard Hiss,
the (projective) tensor product of $\Psi_6$ with $\rho_{321}$ has the
following unipotent constituents: $(q+1)\rho_{321}+q^2(q-1)\rho_{1^6}$.
This implies that $a_1=\cdots=a_4=0$. Furthermore, the
(projective) tensor product of $\Psi_{42}$ with $\rho_{21^4}$ has unipotent
constituents
$$\rho_{41^2}+2\rho_{2^3}+\rho_{2^21^2}+(q^2+2)\rho_{21^4}+\rho_{1^6}$$
hence decomposes as $\Psi_{41^2}+\Psi_{2^3}+\Psi_{2^21^2}+q^2\Psi_{21^4}$,
and necessarily the last entry in the 10th column vanishes as well.
\par
From the cohomology of the Deligne--Lusztig variety for $w=s_1 s_2 s_3$ we get
the virtual projective character
$$\begin{aligned}
 R_w=&\rho_6 - \rho_{42} -\rho_{321}
     +\rho_{2^21^2} -\rho_{1^6}\\
    =& \Psi_6 - \Psi_{42} - \Psi_{41^2} - \Psi_{3^2} -\Psi_{321} + \Psi_{2^3}
     + \Psi_{31^3}+ \Psi_{2^21^2} - (2-a_5) \Psi_{1^6}
\end{aligned}$$
where $\Psi_{321} =  \rho_{321} + a_5\rho_{1^6}$ so that $a_5 \leq 2$.
If $\ell>6$ then $a_5 \geq 2$, which force $a_5 = 2$.
\par
For $n=7$, there are only two entries to be determined, belonging to
the projective character $\Psi_{2^31}$ whose unipotent part equals
$\rho_{2^31}+b_1\rho_{2^21^3}+b_2\rho_{1^7}$.
The existence of an $\ell$-character in general position in $T_w$ for
$w = s_2s_3s_4s_3s_2s_5s_6$ forces the relation $b_2 \geq b_1 + 2$ to hold.
On the other hand the character $R_{s_1 s_2 s_3}[1]$ coming from the
cohomology of the Deligne--Lusztig variety $\rX(s_1 s_2 s_3)$  is given by
$$\begin{aligned}
R_{s_1 s_2 s_3}[1] =&\rho_{7} + \rho_{41^3} -\rho_{1^7}\\
   =& \Psi_{7} - \Psi_{52} - \Psi_{421} +\Psi_{41^3} + \Psi_{321^2} - \Psi_{2^31}
    + (1+b_1) \Psi_{2^21^3} + (b_2 - b_1 - 2)\Psi_{1^7}.\\
\end{aligned}$$
From Proposition~\ref{prop:constitPw} we deduce the relation
$b_2 - b_1 - 2 \leq 0$ so that $b_2 = b_1 + 2$.
\par
We claim that $b_1 = 0$. For this, we view $\GU_7(q)$
as a split Levi subgroup of $\SU_9(q)$. The projective
character $\Psi_{2^31^3}$, cut by the principal block of $\SU_9(q)$,
involves only the unipotent characters $\rho_{2^31^3}$ (with multiplicity $1$),
$\rho_{21^7}$ and $\rho_{1^9}$. Now, the unipotent part of the Harish-Chandra
restriction of $\Psi_{2^31^3}$ to $\GU_7(q)$ has only $\rho_{2^31}$,
$\rho_{21^5}$ and $\rho_{1^7}$ as unipotent constituents, and not
$\rho_{2^21^3}$, which forces $b_1 = 0$ and therefore $b_2 = 2$.
\par
In the case $n=8$, it turns out that no simple unipotent $kG$-module
is cuspidal, so that the decomposition matrix can be determined using
representations of various Hecke algebras and the decomposition matrices
of proper Levi subgroups like $\SU_6(q)$.
\par
For $n=9$, let  $w = s_1 s_2 s_1 s_3 s_2 s_4 s_3 s_8 s_7s_6 s_5 s_4$.
Then $\bT^{wF}$ contains a Sylow $\ell$-subgroup of $G$. If $\ell > 9$,
there exist semisimple $\ell$-elements $t$, $t'$ and $t''$ in $\bT^{wF}$
such that $(C_\bG(t),wF)$ (resp. $(C_\bG(t'),wF)$, resp. $(C_\bG(t''),wF)$)
has semisimple type $A_0(q)$, (resp. $\tw2A_2(q)$, resp. $\tw2A_2(q^3)$).
Through the Jordan decomposition of characters,
we get three non-unipotent cuspidal characters $\rho$, $\rho'$ and $\rho''$.
The restriction to the set of $\ell'$-elements of the last two is expressed
in terms of the basic set of unipotent characters as follows:
$$\begin{aligned}
  (\rho')^0 =&\ \big(-\rho_{81} + \rho_{63} - 2\rho_{54}+2\rho_{52^2}-\rho_{432}
		-2\rho_{3^3}-2\rho_{51^4}+\rho_{3^221}+2\rho_{3^21^3}-2\rho_{2^41}\\
	     &\ -\rho_{2^31^3} +\rho_{21^7}\big)^0, \\
  (\rho'')^0=&\ \big(-\rho_{63}+\rho_{621}+\rho_{54}+\rho_{61^3}-\rho_{52^2}
		-\rho_{4^21}+\rho_{432}+2\rho_{3^3}+\rho_{51^4}-\rho_{3^221}-\rho_{32^3}\\
	     &\-\rho_{3^21^3}-\rho_{41^5}+\rho_{321^4}+\rho_{2^41}+\rho_{2^31^3}\big)^0.
\end{aligned}$$
Since the restriction of $\rho'$ (resp. $\rho''$) has a positive
coefficient on the Brauer character $\vhi_{21^7}$ (resp. the characters
$\vhi_{321^4}$, $\vhi_{2^41}$ and $\vhi_{2^31^3}$), we deduce that
$\vhi_{321^4}$, $\vhi_{2^41}$, $\vhi_{2^31^3}$ and $\vhi_{21^7}$
are cuspidal. Let us denote the entries below the diagonal for these columns
by $c_1,\ldots,c_{13}$.
\par
In order to obtain information on the $c_i$ we start with Harish-Chandra
restriction of projective characters of $\SU_{11}(q)$. The partitions
that are smaller than $321^6$ and that have the same $3$-core (that is $1^2$)
are $321^6$, $2^31^5$ and $1^{11}$. Therefore the corresponding unipotent
characters are the only unipotent constituents of $\Psi_{321^6}$. Since
$\rho_{31^6}$ does not occur in the restriction of these characters to
$\SU_{9}(q)$, we deduce that $c_3=0$. The only unipotent constituent of
$\Psi_{21^{9}}$ is $\rho_{21^9}$, whose restriction is $\rho_{21^7}$, which
proves that $\rho_{21^7}$ is the only unipotent constituent of $\Psi_{21^7}$
and forces $c_{13} = 0$.
For $c_2$ and $c_4$ we need to go up to $\SU_{13}(q)$ and consider the
projective character $\Psi_{32^31^4}$ whose unipotent part is
$\rho_{32^31^4}+x\rho_{32^21^6}+y\rho_{31^{10}}$ for some $x,y\ge0$ (since
$32^31^4$, $32^21^6$ and $31^{10}$ are the only partitions which are smaller
than $32^31^4$ and have $31$ as a $3$-core). The Harish-Chandra restriction
of this character to $\SU_9(q)$, cut by the block, is
$\rho_{32^3} + 2\rho_{321^4}+(2x+y)\rho_{31^6}+(x+2y)\rho_{1^9}$. From the
decomposition matrix of $\SU_9(q)$ we deduce that one copy of $\Psi_{321^4}$
occurs in this restriction, which forces $c_2 = c_4 =0$.
\par
For the other relations we use the virtual characters
$R_w[\lambda]$. Starting with $w=s_1 s_2 s_3 s_4$ we have
$$\begin{aligned}
R_w[1] 	=&\ -\rho_{71^2}-\rho_{41^5}+\rho_{1^9} \\
	=&\ -\Psi_{71^2}+\Psi_{52^2}+\Psi_{4^21}-\Psi_{432}-\Psi_{41^5}+\Psi_{2^31^3}
	+(1-c_{11})\Psi_{21^7} +(2-c_{12})\Psi_{1^9}\\
\end{aligned}$$
so that by Proposition~\ref{prop:constitPw} we obtain $c_{11} \leq 1$ and
$c_{12}\leq 2$. On the other hand, the
expression of $(\rho')^0$ in terms of irreducible Brauer characters
forces $c_{11} \geq 1$ and therefore $c_{11} = 1$. Also, the
$\ell$-reduction of $\rho$ gives $c_{11} + c_{12} \geq 3$.
Together with the previous inequality, this forces $c_{12} = 2$.
\par
With a class of eigenvalues congruent to $q^4$ modulo $\ell$ we have
$$\begin{aligned}
R_w[q^4] 	=&\ -\rho_{81}+\rho_{51^4}+\rho_{21^7} \\
		=&\ -\Psi_{81}+\Psi_{63}-\Psi_{432}+\Psi_{51^4}-\Psi_{3^21^3}
                  +\Psi_{2^41}-c_6\Psi_{2^31^3}\\
	 	 &\ +(2-c_8+c_6)\Psi_{21^7}+(1-c_9-2c_6)\Psi_{1^9}.\\
\end{aligned}$$
From Proposition~\ref{prop:constitPw} we deduce $c_6 = 0$, $c_8 \leq 2$ and
$c_9 \leq 1$. Again, we use $\ell$-reduction of non-unipotent
characters for finding lower bounds: $\rho$ yields the relation
$c_8-c_6 \leq 2$ so that $c_8 = 2$, and $\rho''$ the relation
$-3c_6+c_8+c_9 \geq 3$ which forces $c_9 = 1$.
\par
The last eigenspace of $F^2$ corresponds to the eigenvalues congruent to
$q^2$ modulo $\ell$ and is given by
$$\begin{aligned}
R_w[q^2] 	=&\ \rho_{9}+\rho_{61^3}-\rho_{31^6} \\
		=&\ \Psi_9-\Psi_{71^2}-\Psi_{621}-\Psi_{54}+\Psi_{61^3}
                 +\Psi_{52^2}+\Psi_{4^21}-\Psi_{3^3}-\Psi_{3^221}+\Psi_{32^3}\\
		 &\ +\Psi_{3^21^3}-\Psi_{41^5}+\Psi_{2^31^3}.\\
\end{aligned}$$
In particular, none of $\Psi_{321^4}$, $\Psi_{21^7}$ and $\Psi_{1^9}$ appears
in $R_w$ for $w = s_1 s_2 s_3 s_4$.
\par
Let us now consider the virtual characters coming from the cohomology of the
Deligne--Lusztig variety associated with $w = s_1 s_2 s_3 s_4 s_5 s_4$.
Its contribution to the principal block $b$ is given by
$$\begin{aligned}
bR_w	=&\ \rho_9+\rho_{81}-\rho_{63}+\rho_{52^2}+\rho_{3^21^3}
	   +\rho_{2^31^3} -\rho_{21^7}+\rho_{1^9}\\
	=&\ \Psi_{9}+\Psi_{81}-\Psi_{71^2}-2\Psi_{63}-\Psi_{621}-\Psi_{54}
          +\Psi_{61^3}+2\Psi_{52^2}+\Psi_{4^21}+2\Psi_{432}-2\Psi_{3^3}\\
	 &\ -\Psi_{51^4}-\Psi_{3^221}+\Psi_{32^3}+3\Psi_{3^21^3}
	    -\Psi_{41^5}-2\Psi_{2^41}+2\Psi_{2^31^3}\\
\end{aligned}$$
and we observe that no new projective indecomposable module appears.
\par
With $w= s_1 s_2 s_3 s_4 s_5 s_6 s_5 s_4$, the decomposition of $R_w$, cut
by the block is given by
$$\begin{aligned}
bR_w	=&\ \rho_9+\rho_{63}-\rho_{61^3}-\rho_{4^21}+\rho_{3^3}
	  -\rho_{32^3}+\rho_{41^5}-\rho_{2^31^3}+\rho_{1^9}\\
	=&\ \Psi_{9}-\Psi_{71^2}+\Psi_{63}-\Psi_{621}-\Psi_{54}-2\Psi_{61^3}
	+\Psi_{52^2}-\Psi_{432}+\Psi_{3^3}+\Psi_{3^221}-2\Psi_{32^3}\\
	 &\ +3\Psi_{41^5}+3\Psi_{321^4}+\Psi_{2^41}
	    -4\Psi_{2^31^3}-3\Psi_{31^6} +3\Psi_{21^7}+3(3-c_5)\Psi_{1^9}.\\
\end{aligned}$$
From Proposition~\ref{prop:constitPw} we deduce that $3-c_5\geq 0$.
From the $\ell$-reduction of $\rho''$ we get $-3+c_5 \geq 0$ and therefore
$c_5 = 3$.
\par
Finally, for $n=10$, we denote by $d_i$ the entries in the columns corresponding
to the projective characters $\Psi_{4321}$, $\Psi_{321^5}$ and $\Psi_{2^31^4}$.
Note that we must have $d_{11} = 0$ since $2^5 \ntrianglelefteq 321^5$. Other
easy relations are obtained by Harish-Chandra restriction, as
in the proof of the decomposition matrix of $\SU_{9}(q)$. The restriction
of $\Psi_{4321^3}$ (resp.~$\Psi_{432^31}$, resp.~$\Psi_{321^7}$,
resp.~$\Psi_{2^31^6}$) yields $d_1 =0$ (resp.~$d_2 = d_3 =0$,
resp.~$d_{13} =0$, resp.~$d_{15} =0$). The other relations are obtained by
decomposing the virtual projective modules $R_{w}$.
As usual, we start with a Coxeter element $w = s_1 s_2 s_3 s_4 s_5$. We find
$$\begin{aligned}
  bR_w	=&\ \rho_{10} - \rho_{82} - \rho_{721} + \rho_{621^2} + \rho_{521^3}
  	  - \rho_{421^4} - \rho_{321^5} + \rho_{2^21^6} - \rho_{1^{10}} \\
	=&\ \Psi_{10}-2\Psi_{82}-\Psi_{73}-\Psi_{721}+\Psi_{71^3}+3\Psi_{621^2}
	  +2\Psi_{5^2}-2\Psi_{52^21}+2\Psi_{521^3}-4\Psi_{4^21^2}\\
	 &\ -\Psi_{43^2}+\Psi_{42^3}-6\Psi_{421^4}+3\Psi_{3^31}-3\Psi_{321^5}-2\Psi_{2^5}
	  +3d_{12}\Psi_{2^31^4}\\
	 &\ +6\Psi_{2^21^6}
	   +3(d_{14}-d_{12}d_{16}-2)\Psi_{1^{10}}.\\
\end{aligned}$$
Since $\Psi_{2^31^4}$ and $\Psi_{1^{10}}$ are the characters of projective
covers of cuspidal modules, by Proposition~\ref{prop:constitPw} we must
have $d_{12} \leq 0$ and
$d_{14}-d_{12}d_{16}-2 \leq 0$. This forces $d_{12} = 0$ and
$d_{14}\leq 2$. The converse relation is obtained by considering
as usual the $\ell$-reduction of the Deligne--Lusztig induction $\rho$ of an
$\ell$-character in general position, which exists as soon as $\ell > 10$.
We obtain $d_{14} = 2$. As a byproduct, neither $\Psi_{2^31^4}$ nor
$\Psi_{1^{10}}$ occurs in $R_w$.
\par
Other relations are obtained by considering the Deligne--Lusztig character
associated with $w' = s_1s_2s_3s_4s_6s_5s_6$, and for different eigenvalues
of $F^2$. The virtual projective module corresponding to the generalized
$q^4$-eigenspace decomposes as
$$\begin{aligned}
  bR_w[q^4] =&\ \rho_{10}+\rho_{4^21^2}+\rho_{321^5} \\
  	    =&\ \Psi_{10}-\Psi_{82}-\Psi_{73}+\Psi_{71^3}+\Psi_{621^2}+\Psi_{5^2}
	      -\Psi_{52^21}-\Psi_{4^21^2}-\Psi_{43^2}-\Psi_{421^4}+\Psi_{41^6}\\
	     &\ +\Psi_{321^5}-\Psi_{2^31^4}+\Psi_{2^21^6}
	        +(d_{16}-3)\Psi_{1^{10}},\\
\end{aligned}$$
which, by Proposition~\ref{prop:constitPw}, forces $d_{16}-3 \leq 0$.
One the other hand, we have $d_{16}-3 \geq 0$ by looking at
the coefficient of $\vhi_{2^31^4}$ in $\rho^0$.
Therefore $d_{16}=3$. Finally, the generalized $q^2$-eigenspace yields
$$\begin{aligned}
  bR_w[q^2] =&\ -\rho_{73}-\rho_{4321}+\rho_{2^31^4}\\
  	    =&\ -\Psi_{73}+\Psi_{71^3}+\Psi_{5^2}-\Psi_{52^21}-\Psi_{4^21^2}
	      -\Psi_{43^2}-\Psi_{4321}+\Psi_{431^3}+\Psi_{42^3}\\
	     &\ -\Psi_{421^4}-(1-d_4)\Psi_{41^6} +(1+d_5)\Psi_{3^31}
              +d_6\Psi_{321^5}-(d_5-d_7+2)\Psi_{2^5}\\
	     &\ +x\Psi_{2^31^4}+(1-d_7+d_9)\Psi_{2^21^6} -(2d_6+d_9-d_{10}+3x)\Psi_{1^{10}}
\end{aligned}$$
with $x = -d_4+d_5-d_7+d_8+2$. From Proposition~\ref{prop:constitPw}
we obtain $x \leq 0$ and $2d_6+d_9-d_{10}+3x \geq 0$. The latter is
exactly the coefficient of $\vhi_{4321}$ in $-\rho^0$, and therefore it must
be zero, which yields the relation
$-3d_4+3d_5+2d_6-3d_7+3d_8+d_9-d_{10}+6=0$.
\end{proof}

\subsection{The case $\ell|(q^4-q^3+q^2-q+1)$}

For convenience, we recall the Brauer trees for unipotent $\ell$-blocks of
cyclic defect from \cite{FS2}. Note that the result there was only shown for
odd prime powers $q$, but it is easily seen to hold for even $q$ as well.

\begin{thm}   \label{thm:2An,d=10}
 Let $\ell>n$ be a prime. Then the $\ell$-modular decomposition matrix for the
 unipotent $\ell$-blocks of $\SU_n(q)$, $5\le n\le10$,
 $\ell|(q^4-q^3+q^2-q+1)$, are as given in
 Tables~\ref{tab:2An,d=10,def1}--\ref{tab:2A9,d=10}.
\end{thm}

\begin{table}[ht]
\caption{$\SU_n(q)$, $5\le n\le 10$, $\ell| (q^4-q^3+q^2-q+1)$, $\ell>n$}
  \label{tab:2An,d=10,def1}
$$\begin{array}{cccccccccccc}
 \SU_5(q):\qquad & 5& \vr& 31^2& \vr& 1^5& \vr& \bigcirc& \vr& 21^3& \vr& 41\\
   & & \\
 \SU_7(q):\qquad & 7& \vr& 3^21& \vr& 2^21^3& \vr& \bigcirc& \vr& 21^5&\vr& 43\\
   & & \\
          \qquad & 52& \vr& 32^2& \vr& 1^7& \vr& \bigcirc& \vr& 2^31& \vr& 61\\
   & & \\
 \SU_9(q):\qquad & 9& \vr& 4^21& \vr& 421^3& \vr& \bigcirc& \vr& 41^5&\vr& 431^2\\
   & & \\
\qquad & 72& \vr& 3^3& \vr& 2^21^5& \vr& \bigcirc& \vr& 2^31^3&\vr& 63\\
   & & \\
\qquad & 71^2& \vr& 531& \vr& 2^41& \vr& \bigcirc& \vr& 21^7&\vr& 432\\
   & & \\
\qquad & 54& \vr& 32^21^2& \vr& 31^6& \vr& \bigcirc& \vr& 3^221&\vr& 81\\
   & & \\
\qquad & 521^2& \vr& 32^3& \vr& 1^9& \vr& \bigcirc& \vr& 42^21&\vr& 61^3\\
                 &  & ps& & ps& & 1^5& & 21^3& & 21\\
   & & \\
   & & \\
 \SU_6(q):\qquad & 6& \vr& 42& \vr& 2^21^2& \vr& 1^6& \vr& \bigcirc& \vr& 321\\
   & & \\
 \SU_8(q):\qquad & 8& \vr& 4^2& \vr& 3^21^2& \vr& 31^5& \vr& \bigcirc& \vr& 321^3\\
   & & \\
          \qquad & 61^2& \vr& 42^2& \vr& 2^4& \vr& 1^8& \vr& \bigcirc& \vr& 521\\
   & & \\
 \SU_{10}(q):\qquad &
 82& \vr& 64& \vr& 3^31& \vr& 32^21^3& \vr& \bigcirc& \vr& 321^5\\
   & & \\
 & 631& \vr& 43^2& \vr& 2^41^2& \vr& 2^21^6& \vr& \bigcirc& \vr& 721\\
                 &  & ps& & ps& & ps& & 1^6& & 321\\
   & & \\
   & & \\
 \SU_8(q):  \qquad & 71& \vr& 53& \vr& 3^22& \vr& 2^31^2& \vr& 21^6& \vr& \bigcirc\\
                 &  & ps& & ps& & ps& & ps& & c\\
   & & \\
\end{array}$$
\end{table}

\begin{table}[ht]
{\caption{$\SU_{10}(q)$, $\ell| (q^4-q^3+q^2-q+1)$, $\ell>10$}   \label{tab:2A9,d=10}
$$\vbox{\offinterlineskip\halign{$#$\hfil\ \vrule height11pt depth4pt&&
      \hfil\ $#$\hfil\cr
     10& 5.     & 1\cr
     91& .5     & .& 1\cr
   81^2& 41.    & 1& .& 1\cr
   71^3& .41    & .& 1& .& 1\cr
   61^4& 31^2.  & .& .& 1& .& 1\cr
    5^2& 2.3    & 1& 1& .& .& .& 1\cr
541& \tw2A_5:1^2.& .& .& .& .& .& .& 1\cr
  531^2& 1.31   & .& 1& .& 1& .& 1& .& 1\cr
521^3&\tw2A_5:1.1& .& .& .& .& .& .& 1& .& 1\cr
   51^5& .31^2  & .& .& .& 1& .& .& .& 1& .& 1\cr
   4^22& 21.2   & 1& .&1 & .& .& 1& .& .& .& .& 1\cr
   4321& \tw2A_9& .& .& .& .& .& .& .& .& .& .& .& 1\cr
42^21^2& 21^2.1 & .& .& 1& .& 1& .& .& .& .& .& 1& .& 1\cr
   41^6& 21^3.  & .& .& .& .& 1& .& .& .& .& .& .&  .& 1& 1\cr
 3^22^2& 1^2.21 & .& .& .& .& .& 1& .& 1& .& .& 1& .& .& .& 1\cr
 32^31&\tw2A_5:.2& .& .& .& .& .& .& .& .& 1& .& .& .& .& .& .& 1\cr
   31^7& .21^3  & .& .& .& .& .& .& .& 1& .& 1& .&  .& .& .& 1& .& 1\cr
    2^5& 1^3.1^2& .& .& .& .& .& .& .& .& .& .& 1& a& 1& .& 1& .& .& 1\cr
   21^8&    1^5.& .& .& .& .& .& .& .& .& .& .& .& a\pl2& 1& 1& .&  .& .& 1& 1\cr
 1^{10}&    .1^5& .& .& .& .& .& .& .& .& .& .& .& a& .& .& 1&  2& 1& 1& .& 1\cr
\noalign{\hrule}
 \omit& & ps& ps& ps& ps& ps& ps& \!321& ps& \!321& 1^6& ps& c& ps& 21^6& ps& c& 1^6& 1^5& c& c\cr
   }}$$
 }
\end{table}

\begin{proof}
It remains to consider $\SU_{10}(q)$.
We denote by $a_i$ (resp. $b_i$) the entries in the column corresponding to
the projective characters $\Psi_{4321}$ (resp.~$\Psi_{32^31}$).
The projective character $\Psi_{21^{10}}$ of $\SU_{12}(q)$ has only
$\rho_{21^{10}}$ as unipotent constituent, so that by Harish-Chandra
restriction $\Psi_{21^{8}}$ involves only $\rho_{21^8}$, which gives the
second to last column of the decomposition matrix.
For the same reason, the Harish-Chandra restriction of $\Psi_{4321^5}$
(resp.~$\Psi_{432^31}$, resp.~$\Psi_{32^31^5}$, resp.~$\Psi_{32^51}$)
yields $a_1 = a_2 = a_4 =0$ (resp.~$a_3 =a_5 = 0$, resp.~$b_1 =0$,
resp.~$b_2 = b_3 =0$).
We use the Deligne--Lusztig characters $R_w$ for suitable $w \in W$ to obtain
relations among the remaining indeterminates $a_6$, $a_7$, $a_8$ and $b_4$.
\par
As usual, we start by decomposing the virtual projective module $R_{w}$
corresponding to a Coxeter element $w = s_1 s_2 s_3 s_4 s_5$. We find the
following unipotent constituents from the principal block $b$:
$$\begin{aligned}
  bR_w	=&\ \rho_{10} + \rho_{521^3} - \rho_{1^{10}} \\
	=&\ \Psi_{10}-\Psi_{81^2}+\Psi_{61^4}-\Psi_{5^2}+\Psi_{531^2}
          +\Psi_{521^3}-\Psi_{51^5}+\Psi_{4^22}-\Psi_{42^21^2}\\
	 &-\Psi_{3^22^2}-\Psi_{32^31}+\Psi_{31^7}+\Psi_{2^5}-(2-b_4)\Psi_{1^{10}}.\\
\end{aligned}$$
We deduce from Proposition \ref{prop:constitPw} that $2-b_4\geq 0$.
Assume $\ell > 10$. Then there exists an $\ell$-character in general position,
and we denote by $\rho$ the corresponding non-unipotent character obtained
via Deligne--Lusztig induction. The coefficient of $\vhi_{32^31}$ on $\rho^0$
gives $-2+b_4 \geq 0$, which forces $b_4 = 2$. In particular, $\Psi_{4321}$,
$\Psi_{21^8}$ and $\Psi_{1^{10}}$ do not occur in $bR_w$.
\par
Finally we consider the generalized $q^4$-eigenspace with the element
$w'= s_4s_5s_4s_6s_7s_8s_9$, which decomposes as follows:
$$\begin{aligned}
  bR_{w'}[q^4]	=&\ \rho_{91}-\rho_{4321}-\rho_{21^8}\\
  		=&\ \Psi_{91}-\Psi_{71^3}-\Psi_{5^2}+\Psi_{531^2}+\Psi_{4^22}
                   -\Psi_{4321}-\Psi_{42^21^2}+\Psi_{41^6}-\Psi_{3^22^2}\\
		 &\ +(a_6+1)\Psi_{2^5}-(2+a_6-a_7)\Psi_{21^8}
		    -(a_6-a_8)\Psi_{1^{10}}.\\
\end{aligned}$$
By Proposition~\ref{prop:constitPw} we must have $2+a_6-a_7 \geq 0$ and
$a_6-a_8\geq 0$. On the other hand, the $\ell$-reduction of $\rho$ yields
the relation $-a_6+a_8 \geq 2+a_6-a_7$, and therefore we deduce that
$2+a_6-a_7 = -a_6+a_8= 0$. This gives $a_7 = a_6+2$ and $a_8 = a_6$.
In Table~\ref{tab:2A9,d=10} we have written $a$ in place of the only remaining
unknown $a_6$.
\end{proof}

\subsection{James's row and column removal rule for $\SU_n(q)$
\label{sec:observations}}
Let $\la = (\la_1 \geq \cdots \geq \la_r > 0)$
and $\mu = (\mu_1 \geq \cdots \geq \mu_s > 0)$ be two partitions of $n$.
Following \cite[Thm~6.18]{Jam90}, we would like to relate the decomposition
number $d_{\lambda,\mu} = \langle \rho_\lambda ,\Psi_\mu \rangle$ of
$\SU_n(q)$ to a decomposition number of a smaller unitary group. Two
cases can be considered:
\begin{itemize}
\item (Row removal) If $\lambda_1 = \mu_1$, we set $\bar \la =
(\la_2 \geq \cdots \geq \la_r > 0)$ and $\bar \mu = (\mu_2 \geq \cdots \geq \mu_s > 0)$.
They are partitions of $\bar n = n-\lambda_1$.

\item (Column removal) If $s=r$, we set $\underline \la =
(\la_1-1 \geq \la_2-1 \geq \cdots \geq \la_r -1)$ and
$\underline \mu = (\mu_1-1 \geq \mu_2-1 \geq \cdots \geq \mu_s -1)$.
They are both partitions of $\underline n = n-r$.
\end{itemize}

\noindent
We say that $\la$ and $\mu$ satisfy \emph{James's row} (resp.~\emph{column})
\emph{removal rule} if we are in the first (resp.~second) case and
$d_{\lambda,\mu} = d_{\bar \la, \bar \mu}$ (resp.~$d_{\lambda,\mu}
= d_{\underline \la, \underline \mu}$). James showed in \cite[Rule~5.8]{Jam90}
that this rule holds for $\ell$-decomposition matrices of $\GL_n(q)$.
From the matrices that we have determined in this section we observe that
this should also be the case for $\SU_n(q)$.

\begin{prop}
 James's row and column removal rule holds for the decomposition matrices of
 $\SU_n(q)$ with $n\leq 10$ given in Tables \ref{tab:2A}--\ref{tab:2A9,d=10}.
\end{prop}

\begin{rem}
If James's rule holds in general, it would give some of the entries
that we could not determine, namely
$d_4 = d_{41^6,4321} = d_{1^6,321} = 2$ and
$d_5 = d_{3^31,4321} = d_{2^3,321} = 0$ in Table~\ref{tab:2A9,d=6}.
\end{rem}

\begin{rem}
 Other rules used by James in \cite{Jam90} to determine the decomposition
 matrices of unipotent blocks of $\SL_n(q)$ for $n \leq 10$ will no longer
 hold for $\SU_n(q)$. For example, by Proposition~\ref{prop:smallirr},
 $d_{\la,\mu} = 0$ whenever $\la$ is a triangular partition and
 $\mu \neq \lambda$ but $d_{\mu^\star,\la^\star}$ can be non-zero, as shown in
 Tables~\ref{tab:2A}--\ref{tab:2A9,d=10}. This proves that the analogue of
 \cite[Rule~5.7]{Jam90} does not hold for $\SU_n(q)$.
\end{rem}

%%%%%%%%%%%%%%%%%%%%%%%%%%%%%%%%%%%%%%%%%%%%%%%%%%%%%%%%%%%%%%%%%%%%%%%%%

\end{document}